\documentclass [11pt,reqno]{amsart}
\usepackage {amsmath,amssymb,verbatim,geometry,color,esint}
\usepackage[all]{xy}
\usepackage{mathrsfs}
\usepackage[backref,pagebackref,pdftex,hyperindex]{hyperref}
\usepackage[pdftex]{graphicx}
\usepackage{tikz}
\usetikzlibrary{math}

\usepackage{accents}

\newcommand{\A}{\mathbb{A}}

 \newcommand{\C}{\mathbb{C}}
  \newcommand{\DD}{\mathbb{D}}

\newcommand{\N}{\mathbb{N}}
\renewcommand{\P}{\mathbb{P}}
 \newcommand{\Q}{\mathbb{Q}}
 \newcommand{\R}{\mathbb{R}}
 \newcommand{\Z}{\mathbb{Z}}

\newcommand{\tf}{\widetilde{\f}}
\newcommand{\tp}{\widetilde{\p}}

\newcommand{\tphi}{\widetilde{\phi}}

\newcommand{\cE}{\mathcal{E}}

\newcommand{\cH}{\mathcal{H}}

\newcommand{\cK}{\mathcal{K}}
\newcommand{\cL}{\mathcal{L}}
\newcommand{\cM}{\mathcal{M}}

\newcommand{\cO}{\mathcal{O}}

\newcommand{\cT}{\mathcal{T}}
\newcommand{\cU}{\mathcal{U}}

\newcommand{\cX}{\mathcal{X}}

\newcommand{\ess}{\mathrm{ess}}

\renewcommand{\a}{\alpha}
\renewcommand{\b}{\beta}
\renewcommand{\d}{\delta}
\newcommand{\e}{\varepsilon}
\newcommand{\f}{\varphi}

\newcommand{\om}{\omega}
\newcommand{\Om}{\Omega}

\newcommand{\p}{\psi}

\newcommand{\eg}{{\rm e.g.\ }} 
\newcommand{\ie}{{\rm i.e.\ }}

\newcommand{\inter}{\cdot\ldots\cdot}
\newcommand{\winter}{\wedge\dots\wedge}

\newcommand{\rE}{\mathring{E}}
\newcommand{\rsig}{\mathring{\D}}

\newcommand{\hto}{\hookrightarrow}

\DeclareMathOperator{\dd}{{d}}

\newcommand{\an}{\mathrm{an}}

\DeclareMathOperator{\ii}{I}

\DeclareMathOperator{\Cz}{C^0}

\newcommand{\redu}{\mathrm{red}}

\newcommand{\trop}{\mathrm{trop}}

\DeclareMathOperator{\Capa}{Cap}

\DeclareMathOperator{\DFS}{DFS}

\DeclareMathOperator{\FS}{FS}

\DeclareMathOperator{\MA}{MA}

\DeclareMathOperator{\Spec}{Spec}

\DeclareMathOperator{\supp}{supp}

\DeclareMathOperator{\ord}{ord}

\DeclareMathOperator{\Log}{Log}

\DeclareMathOperator{\CPSH}{CPSH}

\DeclareMathOperator{\reg}{reg}

\DeclareMathOperator{\Sk}{Sk}

\DeclareMathOperator{\Hnot}{H^0}

\DeclareMathOperator{\VCar}{VCar}

\newcommand{\ddc}{dd^c}

\newcommand{\val}{\mathrm{val}}

\newcommand{\re}{\mathrm{ref}}

\newcommand{\D}{\Delta}

\newcommand{\cro}[1]{[\![#1]\!]}

\newcommand{\simto}{\overset\sim\to}

\newcommand{\hyb}{\mathrm{hyb}}

\numberwithin{equation}{section}       

\newtheorem{prop} {Proposition} [section]
\newtheorem{thm}[prop] {Theorem} 
\newtheorem{defi}[prop] {Definition}
\newtheorem{lem}[prop] {Lemma}
\newtheorem{cor}[prop]{Corollary}
\newtheorem{prop-def}[prop]{Proposition-Definition}

\newtheorem{exam}[prop]{Example}
\newtheorem{rmk}[prop]{Remark}

\newtheorem{conj}[prop]{Conjecture}
\newtheorem{ass}[prop]{Assumption}
\theoremstyle{remark}
\newtheorem*{ackn}{Acknowledgment}

\title[From amoebas to pluripotential theory]{From amoebas to pluripotential theory on hybrid analytic spaces}
\date{\today}

\author{S{\'e}bastien Boucksom}

\address{Sorbonne Universit\'e and Universit\'e Paris Cit\'e\\
CNRS\\
IMJ-PRG\\
F-75005 Paris\\
France}
\email{sebastien.boucksom@imj-prg.fr}
\begin{document}

\begin{abstract} These lecture notes are an introduction to the use of non-Archimedean geometry in the study of meromorphic degenerations of complex algebraic varieties. They provide a self-contained discussion of hybrid spaces, which fill in one-parameter degenerations with the associated non-Archimedean Berkovich space as a central fiber. The main focus is on the interplay between complex and non-Archimedean pluripotential theory, and on the relation between convergence of psh metrics and the associated Monge--Amp\`ere measures in the hybrid space, following work of the author with M.~Jonsson, and recent breakthrough work by Y.~Li. 
\end{abstract}

\maketitle

\setcounter{tocdepth}{1}
\tableofcontents
%
%
\section*{Introduction}
The main purpose of these lecture notes is to provide an introduction to the use of non-Archimedean geometry in the study of one-parameter meromorphic degenerations $X\to\DD^\times$ of complex algebraic varieties, \ie complex analytic families over a small punctured disc attached to algebraic varieties $X_K$ over the field $K=\C\{t\}$ of convergent Laurent series. 

\smallskip

The appearance of non-Archimedean techniques in this context can be traced back at least to the pioneering work of George Bergman~\cite{Berg}, which gave birth to tropical geometry. Bergman considered a meromorphic degeneration of algebraic subvarieties $X_t\subset T=(\C^\times)^n$ of an algebraic torus, $t\in\DD^\times$, and provided a realization---at first partly conjectural, later fully settled by Bieri--Groves~\cite{BG}---of the rescaled limit as $t\to 0$ of the \emph{amoeba} of $X_t$, \ie its image under the log map $T\to\R^n$, as the non-Archimedean analogue of the amoeba, defined in terms of (semi)valuations on the algebraic variety $X_K$. 

From a more recent perspective, the space $X_K^\an$ of semivaluations on an algebraic $K$-variety $X_K$, compatible with the natural valuation of $K$, is an instance of \emph{Berkovich analytification} \cite{BerkBook}. Assuming now $X_K$ to be smooth and projective, which is the main case of interest in these notes, $X_K^\an$ can alternatively be described in terms of \emph{snc models} $\cX\to\DD$, \ie extensions of the family $X\to\DD^\times$ obtained by adding in a simple normal crossing divisor $\cX_0$ as a central fiber. To each snc model $\cX$ is attached a finite simplicial complex $\D_\cX$ encoding the combinatorics of the intersections of components of $\cX_0$, and which can be canonically realized as a subspace of $X_K^\an$ using monomial valuations; this further comes with a natural retraction map $X_K^\an\to\D_\cX$, and $X_K^\an$ can then be recovered as the projective limit of all such dual complexes. 

Inspired by the analogy between the retraction maps $X_K^\an\to\D_\cX$ and Lagrangian fibrations, Kontsevich proposed in the early 2000's an ambitious program~\cite{KS} aiming to approach the Strominger--Yau--Zaslow conjecture---and, more generally, study the limiting behavior of degenerations $X\to\DD^\times$ of Calabi--Yau manifolds from a differential geometric perspective---in terms of the non-Archimedean analogue of a Calabi--Yau metric on the associated Berkovich space $X_K^\an$. The first steps of this program were taken in the unpublished notes~\cite{KT}, and later developed in~\cite{nama,siminag}, laying the basis of a non-Archimedean analogue of complex pluripotential theory. 

The relation between a meromorphic degeneration $X\to\DD^\times$ and the associated Berkovich space $X_K^\an$ is best materialized in the \emph{hybrid space} $X^\hyb\to\DD$, a topological extension of $X\to\DD^\times$ that adds $X_K^\an=X_0$ as a central fiber. This construction, originally introduced by Berkovich in \cite{BerNA}, allows to formulate the convergence of measures and metrics on the complex analytic fibers $X_t$ to similar objects on the non-Archimedean space $X_0$, and was exploited in~\cite{konsoib} to establish the convergence of Calabi--Yau volume forms to their non-Archimedean analogue, defined in terms of the Kontsevich--Soibelman \emph{essential skeleton}~\cite{KS,MN,NX}. 

This was taken much further in the recent breakthrough work of Y.~Li~\cite{LiSYZ}, who managed to establish a similar convergence result for the Calabi--Yau potentials, under a conjectural invariance property of the non-Archimedean Calabi--Yau potentials that was in turn established for certain degenerations of Calabi--Yau hypersurfaces in~\cite{HJMM,AH,Litor}. 

\smallskip

The contents of these notes basically consist in a reasonably self-contained introduction to the above developments: 
\begin{itemize}
\item the introductory Section~\ref{sec:amoeba} discusses Bergman's work on amoebas and tropicalizations; 
\item Section~\ref{sec:Berkan} introduces Berkovich's general notion of analytification of schemes over a normed ring, including the notion of continuous metrics and Fubini--Study metrics on line bundles in this context; 
\item Section~\ref{sec:hybrid} introduces the main character of these notes, to wit the hybrid space of a meromorphic degeneration, and establishes its basic topological properties; 
\item Section~\ref{sec:NAmetrics} makes a deeper study of the Berkovich space of a meromorphic degeneration, introducing the notion of models and the associated model metrics and functions, as well as the description of Berkovich analytification as a limit of dual complexes;
\item Section~\ref{sec:hybmod} introduces the hybrid version of model functions and log maps, and shows that they characterize the hybrid topology;
\item building on this, Section~\ref{sec:cvmeas} focuses on the convergence of measures in the hybrid space, and provides a detailed proof of the main result of~\cite{konsoib}, which includes the case of convergence of Calabi--Yau volume forms;
\item Section~\ref{sec:pluripot} provides a (rather quick) introduction to complex and non-Archimedean pluripotential theory, and establishes a slightly improved version of Favre's convergence result for Monge--Amp\`ere measures~\cite{Fav}; 
\item Section~\ref{sec:hybstab} provides a proof of the main result of~\cite{LiSYZ}, with an occasional mild simplification and in the slightly more general setting considered in~\cite{konsoib};
\item finally, Appendix~\ref{sec:stab} reviews Ko\l{}odziej's stability results for solutions to complex Monge--Amp\`ere equations, and their crucial asymmetric version due to Yang Li. 
\end{itemize}
We also refer to Yang Li's very nice survey~\cite{Li22} for another introduction to the present circle of ideas. For lack of time, the present notes unfortunately do not cover many important more recent developments, including~\cite{HJMM,AH,Litor,Li23,Li25}. 

\begin{ackn} These notes are an expanded version of a series of lectures I delivered during the CIME school `Calabi--Yau varieties' in July 2024 in Cetraro. I am very grateful to Simone Diverio, Vincent Guedj and Hoang Chinh Lu for organizing this very nice and successul event, and to the CIME foundation for providing the opportunity to publish these lecture notes. I would also like to thank Mattias Jonsson for our many joint works and discussions on the topics of these notes. Finally, I thank the referees for their careful reading and helpful suggestions. 
\end{ackn}
%
%
%
%
%
%
\section{Amoebas, valuations and tropicalization}\label{sec:amoeba}
This introductory section aims to motivate the use of (semi)valuations in the study of degenerations of algebraic varieties through the historical example of tropicalizations as limits of amoebas. 
%
%
\subsection{The asymptotic cone of an amoeba}
The complex algebraic torus $T:=(\C^\times)^n$ is equipped with a natural \emph{log map}\footnote{The minus sign is here to ensure compatibility with semivaluations, see below.} 
$$
\Log\colon T\to\R^n\quad x\mapsto(-\log|x_1|,\dots,-\log|x_n|),
$$
which is continuous, proper, and in fact a principal $(S^1)^n$-bundle.

\begin{defi} The \emph{amoeba} of an algebraic subvariety $Z\subset T$ is defined as its image $\Log(Z)\subset\R^n$ under the log map.
\end{defi}
Here we identify $Z$ with the corresponding complex analytic variety, \ie its set of complex points endowed with the Euclidean topology. The purpose of what follows is to discuss the pioneering work~\cite{Berg}, where George Bergman discovered that the large scale geometry of the amoeba, \ie its asymptotic cone $\lim_{\e\to 0}\e\Log(Z)$, can be realized as a non-Archimedean version of the amoeba, defined in terms of (semi)valuations, and further admits a piecewise linear structure (see Figure~1), thereby giving birth to what is now called \emph{tropical geometry}. 

\begin{figure}\label{fig:am}
\centering
\includegraphics[scale=0.4]{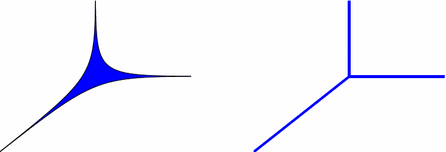}
\caption{Amoeba and tropicalization of $Z=(z_1+z_2+1=0)\subset(\C^\times)^2$}
\end{figure}

\subsubsection{Semivaluations}\label{sec:Berkan}
We view $T$ as an affine algebraic variety with ring of functions 
$$
\cO(T)=\C[z_1^\pm,\dots,z_n^\pm]
$$
the ring of Laurent polynomials 
\begin{equation}\label{equ:Laurent}
f=\sum_{\a\in \Z^n} a_\a z^\a, 
\end{equation}
where we use the standard notation $z^\a:=\prod_i z_i^{\a_i}$ and $a_\a\in\C$ is nonzero for only finitely many $\a$. For any (complex) point $x\in T$ and $\a\in\Z^n$, we have 
$$
\log|x^\a|=\sum_i\a_i\log|x_i|=-\a\cdot\Log(x)
$$
and hence 
\begin{equation}\label{equ:Laurentbd}
-\log|f(x)|\ge \min_{a_\a\ne 0}\a\cdot\Log(x)-C
\end{equation}
for each Laurent polynonial $f$, where $C>0$ only depends on $f$. 

\medskip

Consider next a sequence $(x_j)$ of points of $T$, and $\e_j\in (0,1)$ such that $\e_j\to 0$ and 
$\e_j\Log(x_j)$ admits a limit in $\R^n$. For each $f\in \cO(T)$, $-\e_j\log|f(x_j)|$ is then bounded below, by~\eqref{equ:Laurentbd}. Using Tychonoff's theorem, we may thus assume (after passing to a generalized subsequence, \ie a subnet) that $-\e_j\log|f(x_j)|$ admits a limit $v(f)\in\R\cup\{+\infty\}$ for all $f\in\cO(T)$ (the use of nets being due to the uncountability of this ring). Using the trivial relations
$$
-\e_j\log|(fg)(x_j)|=-\e_j\log|f(x_j)|-\e_j\log|g(x_j)|,
$$
$$
-\e_j\log|(f+g)(x_j)|\ge\min\left\{-\e_j\log|f(x_j)|,-\e_j\log|g(x_j)|\right\}-\e_j\log 2
$$
for all $f,g\in \cO(T)$, we see that the pointwise limit defines a \emph{semivaluation}\footnote{It is a valuation if further $v(f)=\infty\Rightarrow f=0$.} on $\cO(T)$,  \ie a map
$$
v\colon\cO(T)\to\R\cup\{+\infty\}
$$
such that 
$$
v(fg)=v(f)+v(g),\quad v(f+g)\ge\min\{v(f),v(g)\},\quad v(0)=+\infty,
$$ 
whose restriction to the ground field $\C$ further coincides with the \emph{trivial valuation} $v_0\colon\C\to\R\cup\{+\infty\}$, such that $v_0(a)=0$ for $a\in\C^\times$ and $v_0(0)=+\infty$. This leads to the following definition. 

\begin{defi}\label{defi:Berktriv} For any complex affine variety $X$ with ring of functions $\cO(X)$, we denote by $X^\an$ the set of all semivaluations $v$ on $\cO(X)$ whose restriction to $\C$ coincides with the trivial valuation $v_0$. 
\end{defi}
Endowed with the topology of pointwise convergence, the space $X^\an$ is Hausdorff and locally compact (see~\S\ref{sec:Berkan}). 

The following useful simple fact is left as an exercise: 

\begin{lem}\label{lem:valequ} Pick $v\in X^\an$ and a finite set $(f_i)$ in $\cO(X)$. If $v(\sum_i f_i)>\min_i v(f_i)$, then the minimum is achieved at least twice. 
\end{lem}


\subsubsection{The non-Archimedean amoeba}\label{sec:naam}
Consider now a subvariety $Z\subset T$, corresponding to a (radical) ideal $I_Z\subset\cO(T)$, with amoeba $\Log(Z)\subset\R^n$. 

By definition, each point in the asymptotic cone $\lim_{\e\to 0}\e\Log(Z)$ is the limit in $\R^n$ of $\e_j\Log(x_j)$ for some sequence (or net) of complex points $x_j\in Z$ and $\e_j\in\R_{>0}$ converging to $0$. The considerations of~\S\ref{sec:Berkan} thus apply and show that, after passing to a subnet, $f\mapsto-\e_j\log|f(x_j)|$ converges pointwise to a semivaluation $v$ on $\cO(T)$, trivial on $\C$. Since all $x_j$ lie in $Z$, any $f\in I_Z$ further satisfies $-\e_j\log|f(x_j)|=+\infty$, and hence $v(f)=+\infty$, which means that $v$ descends to a semivaluation on $\cO(T)/I_Z\simeq\cO(Z)$, \ie an element of $Z^\an\subset T^\an$. In particular, each component $-\e_j\log|z_i(x_j)|$, $i=1,\dots,n$, of $\e_j\Log(x_j)\in\R^n$ converges to $v(z_i)$, which is finite since $z_i$ is invertible. 

We conclude that the asymptotic cone $\lim_{\e\to 0}\e\Log(Z)$ is contained in the \emph{non-Archimedean amoeba} $\Log(Z^\an)$, \ie the image of $Z^\an\subset T^\an$ under the \emph{non-Archimedean log map}
$$
\Log\colon T^\an\to\R^n
$$
defined by
$$
\Log(v):=(v(z_1),\dots,v(z_n)). 
$$
Like its complex counterpart, the non-Archimedean log map is continuous and proper---and can in fact be interpreted as an affinoid torus fibration in the language of Berkovich geometry. Perhaps more surprinsigly, we have: 

\begin{lem}\label{lem:sec} The map $\Log\colon T^\an\to\R^n$ admits a \emph{unique} continuous section, which takes $w\in\R^n$ to the simplest valuation with value $w_i$ on $z_i$, \ie the \emph{monomial valuation} 
$$
\val_w(\sum_\a a_\a z^\a):=\min_{a_\a\ne 0} \a\cdot w. 
$$
\end{lem}
\begin{proof} It is clear that $w\mapsto\val_w$ is a continuous section. To see uniqueness, pick $v\in T^\an$ and $f=\sum_\a a_\a z^\a\in\cO(T)$. For each $\a\in\Z^n$ we have $v(z^\a)=\val_w(z^\a)=\a\cdot w$ with $w:=\Log(v)$. The semivaluation property of $v$ yields 
\begin{equation}\label{equ:vval}
v(f)\ge\min_{a_\a\ne 0} v(z^\a)=\val_w(f). 
\end{equation} 
If $w$ has $\Q$-linearly independent entries, then $\a\cdot w\ne\b\cdot w$ when $\a\ne\b$, and Lemma~\ref{lem:valequ} thus shows that equality holds in~\eqref{equ:vval}. This means that any section of $\Log$ necessarily coincides with $\val_w$ on the dense subset of $w\in\R^n$ with $\Q$-linearly independent entries, and uniqueness of continuous sections follows. 
\end{proof}

\subsubsection{Tropicalization} 
We next claim that the non-Archimedean amoeba $\Log(Z^\an)$ is in turn contained in the \emph{tropicalization} 
$$
Z^\trop\subset\R^n
$$ 
of $Z$, defined as the intersection of all \emph{tropical hypersurfaces} $V(f^\trop)$ attached to the (concave) piecewise linear functions 
$$
f^\trop(w):=\min_{a_\a\ne 0}\a\cdot w=\val_w(f)
$$ 
with $f=\sum_\a a_\a z^\a$ in $I_Z$. Here $V(f^\trop)\subset\R^n$ is the non-linearity locus of $f^\trop$, \ie the set of points at which the minimum defining it is achieved by at least two indices $\a\ne\b$. 

To see the claim, note that for any $v\in Z^\an\subset T^\an$ and $f=\sum_\a a_\a z^\a$ in $I_Z$, we have $v(f)=\infty$, while $v(z^\a)$ is finite for all $\a$. Setting $w:=\Log(v)$, Lemma~\ref{lem:valequ} shows that the right-hand minimum in 
$$
v(f)>\min_{a_\a\ne 0} v(z^\a)=\min_{a_\a\ne 0} w\cdot\a=f^\trop(w)
$$ 
is achieved at least twice, and the claim follows. At this point, we have thus established the chain of inclusions
\begin{equation}\label{equ:inc}
\lim_{\e\to 0}\e\Log(Z)\subset\Log(Z^\an)\subset Z^\trop. 
\end{equation}
In the present setting, the fundamental theorem of tropical geometry can now be stated as follows: 

\begin{thm}\label{thm:trop}  The inclusions in~\eqref{equ:inc} are equalities. Furthermore, $Z^\trop$ is a rational polyhedral cone complex, \ie a finite union of rational polyhedral cones, of dimension equal to $\dim Z$. 
\end{thm}
This result was essentially already established in~\cite{Berg}: Bergman proved $Z^\trop\subset\Log(Z^\an)$ (see Proposition~\ref{prop:NAtrop} below), and showed that $Z^\trop\subset\lim_{\e\to 0}\e\Log(Z)$ would follow from the last point of the theorem, which was later established by Bieri--Groves~\cite{BG}. While the latter is clear when $Z$ is a hypersurface, defined by a single equation, it becomes nontrivial in the general case because, for a given finite set $(f_i)$ of generators of $I_Z$, it need not be the case that $Z^\trop$ coincides with the intersection of the tropical hypersurfaces $V(f_i^\trop)$. However, this can be shown to hold for \emph{some} well-chosen set of generators (see~\cite[Theorem~2.6.5]{MS}), which yields the last point. 

More recently, the equality between the asymptotic cone and the non-Archimedean amoeba was directly obtained by Jonsson using hybrid spaces~\cite{Jon}. We shall return to this point of view in Example~\ref{exam:amoebas} below.

%
%
\subsection{Degenerations of amoebas}\label{sec:degam}
Consider now a \emph{meromorphic degeneration} of algebraic subvarieties 
$$
Z_t\subset T=(\C^\times)^n,
$$
parametrized by $t\in\DD^\times$ in a small punctured . By definition, this means that the $Z_t$'s are cut out by Laurent polynomials whose coefficients lie in the field
$$
K:=\C\{t\}
$$
of convergent\footnote{On a sufficiently small punctured disc depending on the given series.} Laurent series, and hence correspond to a  $K$-subvariety $Z_K$ of the $K$-torus $T_K$, \ie the affine $K$-variety with ring of functions 
$$
\cO(T_K)=K[z_1^\pm,\dots,z_n^\pm]. 
$$ 
For each $t\in\DD^\times$ we define 
$$
\Log_t\colon T\to\R^n
$$
by 
$$
\Log_t(x)=\left(\frac{\log|x_1|}{\log|t|},\dots,\frac{\log|x_n|}{\log|t|}\right)=\e_t\Log(x),
$$
where the scaling factor 
\begin{equation}\label{equ:et}
\e_t:=(\log|t|^{-1})^{-1}\in (0,1), 
\end{equation}
which tends to $0$ as $t\to 0$, will be ubiquitous in these notes.

We are interested in the asymptotic behavior as $t\to 0$ of the rescaled amoebas
$$
\Log_t(Z_t)=\e_t\Log(Z_t)\subset\R^n. 
$$
The field $K$ is equipped with the discrete valuation $v_K\colon K\to\Z\cup\{+\infty\}$ defined by 
$$
v_K(a)=\ord_0(a):=\min\{k\in\Z\mid a_k\ne 0\}
$$
for any (convergent) Laurent series $a(t)=\sum_{k\ge -N} a_k t^k$; its valuation ring 
$$
\cO_K:=\{f\in K\mid v_K(f)\ge 0\}=\cO_{\C,0}
$$
coincides with the ring of germs of holomorphic functions at $0\in\C$. 

\begin{defi}\label{defi:Berknontriv} The \emph{Berkovich space}\footnote{See Example~\ref{exam:Berkaff} for the terminology.} $Z_K^\an$ is defined as the space of all semivaluations $v$ on $\cO(Z_K)$ whose restriction to $K$ coincides with $v_K$. 
\end{defi}
Again, this space is Hausdorff and locally compact for the topology of pointwise convergence on $\cO(Z_K)$. 

As in~\S\ref{sec:naam}, a compactness argument shows that $\lim_{t\to 0} \Log_t(Z_t)$
is contained in the \emph{non-Archimedean amoeba} $\Log(Z_K^\an)$, \ie the image of $Z_K^\an\subset T_K^\an$ under the non-Archimedean log map 
$$
\Log\colon T_K^\an\to\R^n\quad v\mapsto (v(z_1),\dots,v(z_n)). 
$$
As in Lemma~\ref{lem:sec}, this map admits a unique continuous section, which takes $w\in\R^n$ to the monomial valuation
$$
\val_w\left(\sum_{\a\in\Z^n} a_\a z^\a\right)=\min_\a \val_w(a_\a z^\a)=\min_\a\left\{\a\cdot w+v_K(a_\a)\right\}
$$
Lemma~\ref{lem:valequ} again implies that the non-Archimedean amoeba $\Log(Z_K^\an)$ is contained in the tropicalization $Z^\trop$, defined as the intersection of all tropical hypersurfaces $V(f^\trop)$ attached to the piecewise affine functions 
$$
f^\trop(w)=\val_w(f)=\min_{\a\in\Z^n}\{\a\cdot w+v_K(a_\a)\}
$$
with $f\in I_Z$, \ie the locus where the min is achieved twice. The fundamental theorem of tropical geometry now states that 
$$
\lim_{t\to 0}\Log_t(Z_t)=\Log(Z_K^\an)=Z_K^\trop, 
$$
and that $Z_K^\trop$ is a rational polyhedral complex of dimension $\dim Z$, obtained as the intersection of the tropical hypersurfaces $V(f_i^\trop)$ for some finite set of generators $(f_i)$ of $I_Z$. 
%
%
\subsection{Toric degenerations}\label{sec:toric} 
We will later consider degenerations of smooth varieties to normal crossing divisors, which are (locally and transversally) modelled on the meromorphic degeneration 
\begin{equation}\label{equ:toricdeg}
Z_t:=\left\{t=\prod_{i\in J} z_i^{b_i}\right\}\subset (\C^\times)^J
\end{equation}
for a finite set $J$ and $b\in\Z^J$ nonzero. Here the rescaled amoeba $\Log_t(Z_t)$ is actually independent of $t$, equal to the affine hyperplane
$$
H=H_{J,b}:=\{b\cdot w=1\}\subset\R^J
$$
For later use, we make a number of observations in this context. Consider quite generally a principal bundle $\pi\colon P\to M$ with respect to a compact Lie group $G$. Each fiber $\pi^{-1}(m)$ admits a unique $G$-invariant probability measure $\rho_m$, corresponding to the (normalized) Haar measure of $G$. Any function $f\in \Cz(P)$ gives rise to a function $\int_{P/M} f\in \Cz(M)$, with value $\int_{\pi^{-1}(m)} f\,\rho_m$ at $m\in M$. Dually, any measure $\mu$ on $M$ induces a measure on $P$, which we shall (somewhat abusively) denote by $\pi^\star\mu$, such that 
\begin{equation}\label{equ:pullmes}
\int_P f\,\pi^\star\mu=\int_M\left(\int_{P/M} f\right)\,\mu=\int_M\mu(dm)\int_{\pi^{-1}(m)} f\,\rho_m
\end{equation}
for $f\in C^0_c(P)$. The measure $\pi^\star\mu$ is $G$-invariant, and conversely any $G$-invariant measure $\nu$ on $P$ can be written as $\nu=\pi^\star\mu$ with $\mu=\pi_\star\nu$.  

\begin{exam}\label{exam:logpull1} The log map $\Log\colon(\C^\times)^n\to\R^n$ is a principal $(S^1)^n$-bundle. The $(\C^\times)^n$- invariant holomorphic volume form $\Om:=\bigwedge_i d\log z_i$ (with $d\log z_i:=dz_i/z_i$) induces an $(S^1)^n$-invariant (smooth positive) volume form 
$$
|\Om|^2:=i^{n^2}\Om\wedge\overline\Om
$$ 
on $(\C^\times)^n$, which satisfies 
\begin{equation}\label{equ:Omlog}
|\Om|^2=(2\pi)^n\Log^\star\sigma
\end{equation}
with $\sigma$ the Lebesgue measure on $\R^n$. 
\end{exam}

Returning to~\eqref{equ:toricdeg}, for any $t\ne 0$ the map $\Log_t\colon Z_t\to H$ is similarly a principal bundle with respect to the compact Lie group 
$$
G=G_{J,b}:=\{\tau\in(S^1)^J\mid\prod_i \tau_i^{b_i}=1\}. 
$$
Note that $G$ (and hence $Z_t$) has $m:=\gcd(b_i)$ components, the identity component being the compact torus 
$$
G^0=\{\tau\in(S^1)^J\mid\prod_i \tau_i^{b_i'}=1\},
$$
where $b_i':=b_i/m$. The invariant holomorphic volume form $\Om=\bigwedge_i d\log z_i$ on $(\C^\times)^J$ induces a holomorphic volume form $\Om_t$ on $Z_t$, by requiring
$$
\Om_t\wedge\frac{dt}{t}=\Om
$$
with $t=\prod_i z_i^{b_i}$. Applying~\eqref{equ:Omlog} to the each component of the complexification
$$
G_\C=\{\tau\in(\C^\times)^J\mid \prod_i \tau_i^{b_i}=1\},
$$
one checks  (see~\cite[\S 1.4]{konsoib}) that the associated $G$-invariant smooth positive volume form $|\Om_t|^2$ on $Z_t$ satisfies
\begin{equation}\label{equ:Omt}
\e_t^p |\Om_t|^2=(2\pi)^p\Log_t^\star\sigma_H
\end{equation}
with $p:=\dim Z_t=|J|-1$ and $\sigma_H$ the Lebesgue measure of the affine hyperplane $H$ induced by its equation in $\R^J$. 

\medskip

As a final remark, pick a smooth convex function $f$ on a convex open subset $U\subset H$. Then $\Log_t^\star f$ is a plurisubharmonic (psh) function on the $G$-invariant open subset $\Log_t^{-1}(U)\subset Z_t$, and a computation reveals that the complex and real Monge--Amp\`ere measures are related by
\begin{equation}\label{equ:MAlog}
\left(\ddc\Log_t^\star f\right)^p=\e_t^p\Log_t^\star\MA_\R(f). 
\end{equation}
By approximation, this remains true for any convex function, the Monge--Amp\`ere measures being respectively understood in the sense of Bedford--Taylor and Alexandrov.

%
%

\section{Berkovich analytification}\label{sec:Berkan}
This section reviews Berkovich's general construction of analytic spaces attached to schemes over a normed ring~\cite{BerkBook,BerNA}. We take here an elementary approach, merely viewing these analytifications as topological spaces, and refer to~\cite{Poi,LP} for a more advanced study. 
%
%
\subsection{The Berkovich spectrum}\label{sec:Berkspec}
Let $A$ be a ring\footnote{All rings in these notes are assumed to be commutative and with a unit.}. A \emph{seminorm} $\|\cdot\|\colon A\to\R_{\ge 0}$ is defined as a function that satisfies 
\begin{itemize}
\item $\|0\|=0$, $\|\pm 1\|=1$;
\item $\|a+b\|\le \|a\|+\|b\|$ for all $a,b\in A$. 
\end{itemize}
It is a \emph{norm} if further $\|a\|=0\Rightarrow a=0$. A seminorm is \emph{submultiplicative} if $\|ab\|\le \|a\|\cdot\|b\|$ for all $a,b$, and \emph{multiplicative} if equality holds. 

We usually denote a multiplicative seminorm by the `absolute value' notation $|\cdot|$. Any multiplicative seminorm is induced by a \emph{character} of  $A$, \ie a ring map to a complete valued field. Indeed, the kernel of $|\cdot|$ is a prime ideal $P$, and $|\cdot|$ induces a multiplicative norm on the fraction field of $A/P$, whose completion is called the \emph{residue field} $\cH=\cH(|\cdot|)$. This comes with a character $A\to\cH$ inducing $|\cdot|$, and any other character with the same property factors through $\cH$. 

A multiplicative seminorm $|\cdot|$ is called \emph{non-Archimedean} if it is bounded on the image of $\Z$ in $A$. As is well-known, this is the case iff the ultrametric inequality 
 $$
 |a+b|\le\max\{|a|,|b|\}
 $$
 holds, which equivalently means that 
 $$
 v:=-\log|\cdot|
 $$
 is a semivaluation $v\colon A\to\R\cup\{+\infty\}$, \ie
 $$
 v(ab)=v(a)+v(b),\quad v(a+v)\ge\min\{v(a),v(b)\}. 
 $$
 Otherwise, $|\cdot|$ is called \emph{Archimedean}. In that case, $\Q$ injects in the residue field $\cH$, which thus contains a copy of $\R$, and hence is isomorphic to either $\R$ or $\C$, by the Gelfand--Mazur theorem. 
 
 \begin{defi} Let $A=(A,\|\cdot\|_A)$ be a normed ring, \ie a ring equipped with a submultiplicative norm. The \emph{Berkovich spectrum} $\cM(A)$ is defined as the set of all \emph{bounded} multiplicative seminorms $|\cdot|\colon A\to \R_{\ge 0}$, equipped with the topology of pointwise convergence.
\end{defi} 
Here bounded means $|\cdot|\le C\|\cdot\|_A$ for some $C>0$, which can always be taken equal to $1$, as one easily sees using the multiplicativity (resp.~submultiplicativity) of $|\cdot|$ (resp.~$\|\cdot\|$). 

A point of $\cM(A)$ will usually be labelled by $x$, the corresponding multiplicative seminorm being denoted by $A\ni f\mapsto |f(x)|$, where $f\mapsto f(x)$ is understood as the character $A\to\cH(x)$ to the residue field. Each $f\in A$ thus induces a continuous function 
$$
|f|\colon\cM(A)\to\R_{\ge 0},
$$
such that $\sup_{\cM(A)}|f|\le\|f\|_A$, and the topology of $\cM(A)$ is generated by such functions. 

The space $\cM(A)$ is always nonempty~\cite[Theorem~1.2.1]{BerkBook}. It is further compact Hausdorff, as an easy consequence of Tychonoff's theorem. 

\begin{rmk}\label{rmk:seminorm} The above definition is usually introduced when $A$ is complete for its norm, \ie a Banach ring. Since we merely view analytifications as topological spaces, this makes no difference to us, since any bounded multiplicative seminorm on $A$ automatically extends to the completion of $A$ (compare Proposition~\ref{prop:basean} below). Similarly, $\|\cdot\|_A$ can be allowed to be merely a seminorm, since the set $N\subset A$ of elements with zero norm is then a closed ideal, and $\|\cdot\|_A$ descends to a norm on $A/N$ such that $\cM(A)\simeq\cM(A/N)$. 
\end{rmk}

\begin{rmk}\label{rmk:Berg} When the norm $\|\cdot\|_A$ is ultrametric, $\cM(A)$ was already introduced in~\cite{Bergval}, and proved to be non-empty in Theorem~1 therein. 
\end{rmk}


\begin{exam} For a complex Banach algebra $A$, the Gelfand--Mazur theorem implies that the residue field of any element in $\cM(A)$ is equal to $\C$, so that $\cM(A)$ can be identified with the usual Gelfand spectrum of $A$, \ie the space of bounded characters $A\to\C$. 
\end{exam}

\begin{exam} If $K=(K,|\cdot|_K)$ is a valued field, \ie a field equipped with a multiplicative norm, then $\cM(K)$ reduces to the point $|\cdot|_K$. 
\end{exam}

\begin{exam} The ring of integers $\Z$ equipped with the usual absolute value $|\cdot|_\infty$ is a (complete) normed ring. As a consequence of the well-known Ostrowski theorem, each element of $\cM(\Z)$ is either of the form $|\cdot|_p^\e$ for some prime $p\in\Z$ and $\e\in [0,\infty]$, the case $\e=0$ (resp.~$\e=\infty$) being understood as the trivial absolute $|\cdot|_0$ on $\Z$ (resp.~on $\Z/p\Z$), or $|\cdot|_\infty^\e$ with $\e\in[0,1]$, the case $\e=0$ being again $|\cdot|_0$. Thus $\cM(\Z)$ is a star-shaped $\R$-tree rooted at $|\cdot|_0$, endowed with the \emph{weak} tree topology, for which each neighborhood of $|\cdot|_0$ contains all but finitely many of the branches. Analytic geometry over $\Z$ has been extensively studied by Poineau~\cite{Poi}. 
\end{exam}

\begin{exam}\label{exam:hybrid} The \emph{hybrid field} $\C^\hyb$ is defined as the field $\C$ equipped with the \emph{hybrid norm} 
$$
\|\cdot\|^\hyb:=\max\{|\cdot|_\infty,|\cdot|_0\}, 
$$
where $|\cdot|_\infty$ is the usual absolute value and $|\cdot|_0$ is the trivial one. The hybrid field is a (complete) normed ring, and one checks that each element of $\cM(\C^\hyb)$ is of the form $|\cdot|_\infty^\e$ with $\e\in [0,1]$, the case $\e=0$ being again interpreted as $|\cdot|_0$. This yields an identification 
$$
\cM(\C^\hyb)\simeq [0,1].
$$ 
\end{exam}
Analytic geometry over $\C^\hyb$ plays a central role in these notes, and will be studied in more detail in~\S\ref{sec:hybrid}. 
%
%
\subsection{General Berkovich analytification}
We use~\cite[\S 1]{BerNA} as a reference for what follows. 

\begin{defi}\label{defi:Berkgen} Let $A=(A,\|\cdot\|_A)$ be a normed ring, and $X$ an affine scheme of finite type over $A$, corresponding to an $A$-algebra of finite type $\cO(X)$. The \emph{Berkovich analytification} $X^\an$ is defined as the set of all multiplicative seminorms $|\cdot|\colon \cO(X)\to \R_{\ge 0}$ whose pullback to $A$ is bounded by $\|\cdot\|_A$. It is equipped with the topology of pointwise convergence. 
\end{defi}
As above, a point of $X^\an$ is usually labelled by $x$, $\cO(X)\ni f\mapsto |f(x)|$ being the associated multiplicative seminorm. Each $f\in\cO(X)$ thus defines a continuous function 
$$
|f|\colon X^\an\to\R_{\ge 0},
$$
the topology of $X^\an$ being generated by such functions. The space $X^\an$ is Hausdorff, and comes with a continuous \emph{structure map} 
$$
X^\an\to\cM(A).
$$
Analytification is functorial, \ie any map $X\to Y$ of affine $A$-schemes induces a continuous map $X^\an\to Y^\an$ compatible with the structure maps, and with composition. 

\begin{lem}\label{lem:loccomp} The space $X^\an$ is locally compact and countable at infinity. 
\end{lem}
Note, however, that $X^\an$ is typically not first countable, so that its topology should be described in terms of nets instead of sequences. This already happens in the case of main interest for these notes, \ie for algebraic varieties over $A=\C\{t\}$. 

\begin{proof} Pick a finite set of generators $(f_i)$ of the $A$-algebra $\cO(X)$. For each $m\in\N$, the set 
$$
K_m:=\{\max_i|f_i|\le m\}\subset X^\an
$$
is then compact, as a consequence ofTychonoff. The result follows since any given point of $X^\an$ admits $K_m$ as a compact neighborhood for $m\gg 1$. 
\end{proof}

\begin{exam}\label{exam:Berkaff} Assume $A=(K,|\cdot|_K)$ is a non-Archimedean valued field, with associated (possibly trivial) valuation $v_K=-\log|\cdot|_K$. Every multiplicative seminorm $|\cdot|$ in $X^\an$ is then automatically non-Archimedean, being bounded on the image of $\Z\to K$. Setting $v:=-\log|\cdot|$ yields an identification of $X^\an$ with the space of all semivaluations $v\colon\cO(X)\to\R\cup\{+\infty\}$ such that $v|_K=v_K$ (compare Definition~\ref{defi:Berknontriv}). 
\end{exam}

If $X\hto Y$ is a closed embedding of affine $A$-schemes, then $X^\an$ can be identified with the closed subspace of seminorms in $Y^\an$ that vanish on each $f\in\cO(Y)$ lying in the ideal of $X$. 

Similarly, if $X\hto Y$ is an open embedding, then the induced map $X^\an\hto Y^\an$ is an open embedding of topological spaces. This allows to globalize the construction and define the analytification 
$$
X^\an\to\cM(A)
$$
of any $A$-scheme of finite type $X$ by gluing together the analytifications of finitely many affine open subschemes covering $X$. The topology of $X^\an$ is thus the coarsest one such that, for any affine open subscheme $U\subset X$, we have
\begin{itemize}
\item $U^\an\hto X^\an$ is an open embedding; 
\item $|f|\colon U^\an\to\R_{\ge 0}$ is continuous for each $f\in\cO(U)$.
\end{itemize}
\begin{prop}\label{prop:Berktop}\cite[Lemma~1.2]{BerNA} For any scheme $X$ of finite type over a normed ring $A$, the topological space $X^\an$ is: 
\begin{itemize}
\item[(i)] locally compact and countable at infinity; 
\item[(ii)] Hausdorff if $X$ is separated over $A$ (\eg a variety over a field); 
\item[(iii)] compact Hausdorff if $X$ is projective over $A$. 
\end{itemize}
\end{prop}
\begin{proof} Since $X$ can be written as a finite union of affine open subschemes, (i) follows from Lemma~\ref{lem:loccomp}. We next prove (iii). Assume $X$ is projective, and pick a closed embedding $X\hto\P^N_A$. Then $X^\an\hto\P^{N,\an}_A$ is a closed embedding, and it therefore suffices to show that $\P^{N,\an}_A$ is compact. If we denote by $z_0,\dots,z_n$ the homogeneous coordinates of $\P^N_A$, then each point of $\P^{N,\an}_A$ lies in the unit polydisc $\{\max_{j\ne i}|z_j/z_i|\le 1\}$ in one of the coordinate charts $\{z_i\ne 0\}$, and (iii) follows since each such polydisc is compact, by Lemma~\ref{lem:loccomp} and its proof. Finally, (ii) follows from (iii) at least when $X$ is quasi-projective, which is enough for our purposes. 
\end{proof}

\begin{exam}\label{exam:GM} The Gelfand--Mazur theorem implies that the analytification of any complex variety $X$ with respect to the usual absolute $|\cdot|_\infty$ on $\C$ coincides (as a topological space) with the usual complex analytic space attached to $X$, \ie the set of complex points with the Euclidean topology. The identification is given by sending a complex point $x$ lying in some affine open $U\subset X$ to the point of $U^\an$ defined by the multiplicative seminorm $\cO(U)\ni f\mapsto|f(x)|_\infty$. 
\end{exam}

\begin{exam}\label{exam:NAsemival} Let $X$ be a variety over a non-Archimedean valued field $K$, with valuation $v_K$. Then each valuation 
$$
v\colon K(X)^\times\to\R
$$
on the function field of $X$ such that $v|_K=v_K$ defines a point in $X^\an$. This point is Zariski dense, in the sense that it lies in $U^\an$ for any affine open subscheme $U$, the corresponding multiplicative seminorm
$$
\cO(U)\ni f\mapsto |f(v)|:=e^{-v(f)}
$$
being a norm. More generally, any point of $X^\an$ corresponds to a valuation 
$$
v\colon K(Y)^\times\to\R
$$
on the function field of some (closed) subvariety $Y\subset X$, and will loosely be referred to as a \emph{semivaluation} on $X$. 
\end{exam}
%
\subsection{Tropicalization, reprise}
To illustrate the above ideas, we return to the setting of~\S\ref{sec:amoeba}, and establish the equality between non-Archimedean amoeba and tropicalization. 

\begin{prop}\label{prop:NAtrop} 
Let $K$ be a field endowed with a (possibly trivial) valuation $v_K$, and 
$$
Z\subset T=\Spec K[z_1^\pm,\dots,z_n^\pm]
$$ 
a subvariety of a (split) torus over $K$, with non-Archimedean log map 
$$
\Log\colon T^\an\to\R^n,\quad v\mapsto (v(z_1),\dots, v(z_n)).
$$
Then the non-Archimedean amoeba $\Log(Z^\an)$ coincides with the tropicalization $Z^\trop$, \ie the intersection of the non-linearity loci $V(f^\trop)$ of all piecewise affine concave functions 
$$
f^\trop(w):=\min_\a\{\a\cdot w+v_K(a_\a)\}
$$
with $f=\sum_{\a\in\Z^n} a_\a z^\a$ in the ideal $I_Z$ of $Z$. 
\end{prop}

\begin{proof} As in \S\ref{sec:amoeba}, the inclusion $\Log(Z^\an)\subset Z^\trop$ is a direct consequence of Lemma~\ref{lem:valequ}. Conversely, pick $w\in\R^n$, and consider the monomial valuation $\val_w\in T^\an$ such that
$$
\val_w(f)=\min_\a\val_w(a_\a z^\a)=f^\trop(w). 
$$
Following~\cite{Berg}, we note that $w$ lies in $Z^\trop$ iff, for any $\b\in\Z^n$, we have 
$$
\val_w(z^\b)=\sup_{g\in I_Z}\val_w(z^\b+g). 
$$
Indeed, by definition $w\notin Z^\trop$ iff there exists $g=\sum_\a a_\a z^\a$ in $I_Z$ and $\b\in\Z^n$ such that $a_\b\ne 0$ and $$
\min_{\a\ne\b}\{\a\cdot w+v_K(a_\a)\}=\val_w(a_\b z^\b-g)>\b\cdot w+v_K(a_\b)=\val_w(a_\b z^\b). 
$$
Dividing $g$ by $a_\b\in K^\times$ yields the claim. 

The multiplicative seminorm $|\cdot|_w=e^{-\val_w}$ of $\cO(T)=K[z_1^\pm,\dots,w_n^\pm]$ induces a quotient seminorm $\|\cdot\|_{Z,w}$ on $\cO(Z)=\cO(T)/I_Z$, defined by 
$$
\|f\|_{Z,w}:=\inf_{g\in I_Z}|f+g|_w=\exp(-\sup_{g\in I_Z}\val_w(f+g))
$$
for $f\in\cO(T)$, and which is \emph{a priori} merely submultiplicative. By~\cite[Theorem~1.2.1]{BerkBook}, the Berkovich spectrum of $(\cO(Z),\|\cdot\|_{Z,w})$ is non-empty (see Remark~\ref{rmk:seminorm}), which means that there exists $v\in Z^\an$ such that 
$$
v(f)\ge\sup_{g\in I_Z}\val_w(f+g)
$$ 
for all $f\in\cO(T)$. For $w\in Z^\trop$, the above observation yields 
$$
\sum_i\b_i v(z_i)=v(z^\b)\ge\sup_{g\in I_Z}\val_w(z^\b+g)=\val_w(z^\b)=\sum_i \b_i w_i
$$ 
for all $\b\in\Z^n$. This proves $v(z_i)=w_i$ for all $i$, and hence $w=\Log(v)\in\Log(Z^\an)$. 
\end{proof}

\begin{rmk} The above argument is exactly the one used in~\cite{Berg}, which relied on~\cite[Theorem~1]{Bergval} where we have used~\cite[Theorem~1.2.1]{BerkBook}. 
\end{rmk}

%
\subsection{Base change}

In preparation to the study of hybrid spaces, we study the effect of base change on Berkovich analytification. Consider as above a scheme $X$ of finite type over a normed ring $A$, and assume given a normed ring map $\tau\colon A\to B$, \ie a ring map such that $\|\tau(a)\|_B\le \|a\|_A$ for all $a\in A$. The analytification of the base change $X_B$ sits in a commutative diagram of continuous maps
\begin{equation}\label{equ:basean}
\xymatrix{X_B^\an\ar[r]\ar[d]  & X^\an\ar[d]\\
\cM(B)\ar[r]& \cM(A)}
\end{equation}

\begin{prop}\label{prop:basean} Assume the normed ring map $\tau\colon A\to B$ has dense image. Then $\cM(B)\hto\cM(A)$ is a homeomorphism onto its image, and $X_B^\an\hto X^\an$ is a homeomorphism onto the preimage of $\cM(B)\subset\cM(A)$. If we further assume that $\tau\colon A\to B$ is an isometry with dense image (\eg the completion of $A$), then 
$$
\cM(B)\simto\cM(A),\quad X_B^\an\simto X^\an
$$
are both homeomorphisms. 

\end{prop}
\begin{proof} Writing $X$ as a finite union of affine open subschemes, we may assume without loss that $X$ is affine. Pick $f\in\cO(X_B)=\cO(X)\otimes_A B$, and write $f=\sum_{k=1}^N b_k\otimes g_k$ with $b_k\in B$ and $g_k\in\cO(X)$. Since $\tau(A)$ is dense in $B$, we can then find a sequence $f_n\in\cO(X)$ of the form $f_n=\sum_{k=1}^N a_{n,k} g_k$ where $a_{n,k}\in A$ satisfies $\tau(a_{n,k})\to b_k$ for each $k$. We shall then say that $(f_n)$ is an \emph{admissible approximation} of $f$. 

To show that $X_B^\an\to X^\an$ is a homeomorphism onto its image, pick a net $(y_i)$ in $X_B^\an$ and $y\in X_B^\an$. Assuming their images $x_i,x$ in $X^\an$ satisfy $x_i\to x$, we need to show $y_i\to y$ in $X_B^\an$, \ie $|f(y_i)|\to |f(y)|$ for any $f\in\cO(X_B)=\cO(X)\otimes_A B$. Pick an admissible approximation $f_n=\sum_{k=1}^N a_{n,k} g_k$ of $f$ as above, and set 
$$
\e_n:=\max_{1\le k\le N}\|\tau(a_{n,k})-b_k\|_B.
$$
Since $\lim_i |g_k(x_i)|=|g_k(x)|$ for all $k$, we may assume $\sup_{k,i}|g_k(x_i)|\le C<\infty$, where $C>0$ denotes a uniform constant that is allowed to vary from line to line. Then 
$$
\big||f(y_i)|-|f_n(x_i)|\big|=\left|(f-\tau(f_n))(y_i)\right|\le \e_n C,\quad\big||f(y)|-|f_n(x)|\big|\le \e_n C. 
$$
Since $x_i\to x$, we have $|f_n(x_i)|\to |f_n(x)|$ for each $n$. We infer 
$$
\limsup_i\left||f(y_i)-|f(y)|\right|\le\e_n C
$$
for all $n$, and hence $|f(y_i)|\to |f(y)|$. This shows that $X_B^\an\hto X^\an$ is a homeomorphism onto its image. Specializing to $X=\Spec A$, this implies that $\cM(A)\hto\cM(B)$ also is a homeomorphism onto its image. 

Next pick $x\in X^\an$ in the preimage of $\cM(B)\hto\cM(A)$, \ie $|a(x)|\le\|\tau(a)\|_B$ for all $a\in A$. Pick $f\in\cO(X_B)$, and let $f_n\in\cO(X)$ be an admissible approximation of $f$. Then 
$$
\big||f_n(x)|-|f_m(x)|\big|\le|(f_n-f_m)(x)|\le(\e_n+\e_m) C\max_k |g_k(x)|^C. 
$$
The sequence $(|f_n(x)|)_n$ is thus Cauchy, and hence admits a limit in $\R_{\ge 0}$. For any other choice of admissible approximation $(f'_n)$ of $f$, we similarly see that $|f_n(x)|-|f'_n(x)|$ tends to $0$. It follows that $|f(y)|:=\lim_n |f_n(x)|$ is independent of the choice of admissible approximation, and hence defines a point $y\in X_B^\an$ mapping to $x$, since admissible approximations are stable under sums and products. 

Finally, assume that $\tau$ is an isometry. A similar Cauchy sequence argument then shows that $\cM(B)\to\cM(A)$ is onto, and the rest follows.  
\end{proof}

%
%
\subsection{Continuous metrics}\label{sec:metrics}
Let $X$ be a scheme of finite type over a normed ring $A$, and pick a line bundle $L$ on $X$. 

\begin{defi}\label{defi:metric} We define a \emph{continuous metric} $\phi$ on $L^\an$ as the data, for any Zariski open $U\subset X$ and $s\in\Hnot(U,L)$, of a continuous function 
$$
|s|_\phi\colon U^\an\to\R_{\ge 0},
$$
compatible with restriction to open subschemes, and subject to the conditions 
\begin{itemize}
\item[(i)] $|fs|_\phi=|f||s|_\phi$ for any $f\in\cO(U)$; 
\item[(ii)] $|s|>0$ on $U^\an$ when $s$ is a trivializing section on $U$. 
\end{itemize}
\end{defi}
Fix a trivializing open cover $(U_i)$ of $X$ for $L$, with trivializing sections $s_i\in\Hnot(U_i,L)$ and associated cocyle $g_{ij}\in\cO^\times(U_{ij})$. Setting $\phi_i:=-\log|s_i|_\phi$ then defines a 1--1 correspondence between the set of continuous metrics $\phi$ on $L^\an$ and that of finite families of functions $\phi_i\in \Cz(U_i^\an)$ such that $\phi_i-\phi_j=\log|g_{ij}|$ on $U_{ij}^\an$. Using a partition of unity, we easily infer that $L^\an$ always admits a continuous metric. 

We use additive notation for line bundles and metrics thereon, \ie if $\phi,\phi'$ are continuous metrics on $L^\an$, $L'^\an$ for some line bundles $L,L'$ on $X$, we denote by $-\phi$ the induced metric on the dual line bundle $-L:=L^\vee$, and $\phi+\phi'$ the induced metric on $L+L':=L\otimes L'$. 

When $L=\cO_X$ is trivial, we further identify a continuous metric $\phi$ on $\cO_X^\an$ with the continuous function 
$$
\phi=-\log|1|_\phi\in \Cz(X^\an).
$$
Given two continuous metrics $\phi,\phi'$ on the same line bundle $L$, $\phi-\phi'$ is thus viewed as a continuous function on $X^\an$, in line with the fact that the space $\Cz(L^\an)$ of continuous metrics on $L^\an$ is an affine space modeled on $CX^0(X^\an)$. 

\begin{exam}\label{exam:FS} Assume that, for some $m\ge 1$, $mL$ is generated by finitely many global sections $s_i\in\Hnot(X,mL)$, and pick constants $\lambda_i\in\R$. This defines a \emph{tropical Fubini--Study metric}
$$
\phi=\tfrac 1m\max_i\{\log|s_i|+\lambda_i\}
$$ 
on $L^\an$, where $\log|s_i|$ denotes the (singular) metric induced by $s_i$. Explicitly, 
$$
|s|_\phi=\left(\max_i|s_i/s^m|e^{\lambda_i}\right)^{-1/m}
$$
on $U^\an$ for any trivializing section $s\in\Hnot(U,L)$ with $U\subset X$ open. If further $\lambda_i\in\Q$ we say that $\phi$ is \emph{rational}. 
\end{exam}
Such metrics give rise to a convenient space of `test functions'. More specifically, assume $X$ is projective, so that $X^\an$ is compact (see Proposition~\ref{prop:Berktop}). Pick an ample line bundle $L$, and denote by 
$$
\DFS(X^\an)\subset \Cz(X^\an)
$$
the set of functions of the form $f=m(\phi-\p)$ with $\phi,\p$ rational tropical Fubini--Study metrics on $L^\an$ and $m\in\Z_{>0}$. Then: 

\begin{prop}\label{prop:DFS} The set $\DFS(X^\an)$ is independent of the choice of ample line bundle $L$, and is a dense $\Q$-linear subspace of $\Cz(X^\an)$, stable under max and containing $1$. 
\end{prop}
\begin{proof} As one easily sees, $\DFS(X^\an)$ is a $\Q$-linear subspace stable under max and containing $1$. It is also not hard to check that it is independent of $L$ (see~\cite[Lemma~2.6]{nakstabold}), and that it separates the points of $X^\an$~\cite[Theorem~2.7]{nakstabold}. Density is then a consequence of the Stone--Weierstrass theorem. 
\end{proof}

%
%
\section{Hybrid spaces}\label{sec:hybrid}
This section provides a self-contained introduction to the main character of these notes, \ie the hybrid space attached to a meromorphic degeneration of algebraic varieties. 
%
%
\subsection{The hybrid space of a complex variety}\label{sec:hybtriv}
To any complex algebraic variety $Z$ can be attached two types of analytifications in the sense of Definition~\ref{defi:Berkgen}, corresponding respectively to the usual absolute value $|\cdot|_\infty$ and the trivial absolute value $|\cdot|_0$ on $\C$. 

As mentioned in Example~\ref{exam:GM}, the Gelfand--Mazur theorem implies that the first analytification recovers the usual one, \ie the set of complex points of $Z$ with the Euclidean topology. As in Example~\ref{exam:NAsemival}, we view the analytification $Z^\an$ with respect to $|\cdot|_0$ as a space of semivaluations, which extends Definition~\ref{defi:Berktriv} to the non-affine case. 

Following Berkovich~\cite{BerNA}, we now bring together these two analytifications into a single hybrid object, as follows. 

\begin{defi} The \emph{hybrid space} $Z^\hyb$ of a complex algebraic variety $Z$ is defined as its analytification with respect to the hybrid norm $\|\cdot\|^\hyb=\max\{|\cdot|_\infty,|\cdot|_0\}$ on $\C$. 
\end{defi} 
By Proposition~\ref{prop:Berktop}, the space $Z^\hyb$ is Hausdorff and locally compact; it is further compact if $Z$ is projective. 

Recall from Example~\ref{exam:hybrid} that the spectrum $\cM(\C^\hyb)$ of the hybrid field $\C^\hyb=(\C,\|\cdot\|^\hyb)$ can be canonically identified with the interval $[0,1]$ by mapping $\e\in [0,1]$ to $|\cdot|_\infty^\e$, interpreted as $|\cdot|_0$ when $\e=0$. The hybrid space of $Z$ thus comes with a structure map
$$
Z^\hyb\to \cM(\C^\hyb)=[0,1]. 
$$
The fiber $Z^\hyb_\e$ over $\e>0$ is the analytification with respect $|\cdot|_\infty^\e$, which is canonically identified with the usual analytification (by Gelfand--Mazur), while the fiber $Z^\hyb_0$ over $\e=0$ is the Berkovich space $Z^\an$. This yields an identification
$$
Z^\hyb\simeq (Z^\an\times\{0\})\coprod(Z\times (0,1])
$$
compatible with the projection to $[0,1]$, the topology being the coarsest one such that for any affine open $U\subset Z$ we have
\begin{itemize}
\item $U^\hyb\subset Z^\hyb$ is open; 
\item for any $f\in\cO(U)$, the function $|f|^\hyb\colon U^\hyb\to\R_{\ge 0}$ defined by $|f|^\hyb(x,\e):=|f(x)|^{\e}$ for $(x,\e)\in Z\times (0,1]$ and $|f|^\hyb(v,0):=e^{-v(f)}$ for $v\in Z^\an$, is continuous. 
\end{itemize}
A sequence (or net) $(x_j,\e_j)$ of complex points $x_j\in X$ and $\e_j\to 0_+$ thus converges to a semivaluation $v\in Z^\an$ iff, for any Zariski open $U\subset X$ with $v\in U^\an$, $x_j$ ultimately lies in $U^\an$, and 
$$
-\e_j\log|f(x_j)|\to v(f)
$$
for all $f\in\cO(U)$. Intuitively, this means that $v$ is realized as a `vanishing order' along the net $(x_i)$, with rate $(\e_j)$. 

By~\cite[Theorem~C]{Jon}, the structure map $Z^\hyb\to [0,1]$ is open, \ie given any $v\in Z^\an$ and any net $\e_j\to 0$, we can find complex points $x_j\in Z$ such that $(x_j,\e_j)\to (v,0)$ in $Z^\hyb$. In particular, the Archimedean part $Z\times(0,1]$ is dense in $Z^\hyb$, \ie any semivaluation can be realized as a vanishing order along some net of complex points (see also Corollary~\ref{cor:dense} below). 

\begin{exam}\label{exam:amoebas} Let $T=\Spec \C[z_1^\pm,\dots,z_n^\pm]$ be a torus, with hybrid analytification $T^\hyb\to [0,1]$. The \emph{hybrid log map}
$$
\Log\colon T^\hyb\to\R^n
$$ 
with components $-\log|z_i|$ is then continuous. Given any subvariety $Z\subset T$, we have $\Log(Z^\hyb_\e)=\e\Log(Z)$ for $\e>0$, and $\Log(Z^\hyb_0)=\Log(Z^\an)$. As in~\cite{Jon}, we thus recover the equality 
$$
\Log(Z^\an)=\lim_{\e\to 0}\e\Log(Z)
$$
between the non-Archimedean amoeba and the asymptotic cone of the (complex) amoeba $\Log(Z)$ (see Theorem~\ref{thm:trop}) from the openness of the structure map $Z^\hyb\to [0,1]$.
\end{exam}
Combined with Proposition~\ref{prop:NAtrop}, this recovers most of the `fundamental theorem of tropical geometry' (Theorem~\ref{thm:trop}).

%
%
\subsection{The hybrid space of a meromorphic degeneration} 
We next consider the case of main interest in these notes, \ie that of a meromorphic degeneration 
$$
\pi\colon X\to\DD^\times
$$
of complex algebraic varieties $X_t=\pi^{-1}(t)$, parametrized by a small punctured disc centered at $0\in\C$. By definition, this means that $X$ is the complex analytic space attached to an algebraic variety $X_K$ over the non-Archimedean field $K=\C\{t\}$ of convergent Laurent series. 

Define
$$
X_0:=X_K^\an
$$
as the Berkovich analytification of $X_K$, viewed as a space of semivaluations $v$ on $X$ (see Example~\ref{exam:NAsemival}). Recall also the scaling factor
$$
\e_t:=(\log|t|^{-1})^{-1}\in (0,1),
$$
which tends to $0$ as $t\to 0$  (see~\S\ref{sec:degam}). Inspired by~\S\ref{sec:degam} and~\S\ref{sec:hybtriv}, we introduce: 

\begin{defi}\label{defi:hybrid} The \emph{hybrid space} of the meromorphic degeneration $X\to\DD^\times$ 
is defined as the disjoint union
$$
X^\hyb:=\coprod_{t\in\DD} X_t
$$
endowed with the \emph{hybrid topology}, \ie the coarsest topology such that, for any affine open $U_K\subset X_K$, we have:
\begin{itemize}
\item[(i)] $U^\hyb\subset X^\hyb$ is open;
\item[(ii)] for any $f\in\cO(U_K)$, the function $|f|^\hyb\colon U^\hyb\to\R_{\ge 0}$ equal to $x\mapsto |f(x)|^{\e_t}$ on $X_t$ for $t\ne 0$ and $v\mapsto |f(v)|=e^{-v(f)}$ for $v\in X_0$ is continuous. 
\end{itemize}
\end{defi}
Equivalently, the hybrid topology is characterized by: 
\begin{itemize}
\item $X\hto X^\hyb$ is an open embedding; 
\item $X_0\hto X^\hyb$ is a closed embedding; 
\item a net of complex points $x_i\in X_{t_i}$ with $t_i\to 0$ converges in $X^\hyb$ to a semivaluation $v\in X_0$ iff, for any affine open $U_K\subset X_K$ such that $v\in U_0$, we have $x_i\in U_{t_i}$ for $i$ large enough, and 
\begin{equation}\label{equ:hybcv}
\lim_i\frac{\log|f(x_i)|}{\log|t_i|}=v(f)
\end{equation}
for all $f\in\cO(U_K)$. 
\end{itemize}
As in~\S\ref{sec:hybtriv}, we interpret~\eqref{equ:hybcv} as saying that $v$ is realized as a `vanishing order' along the net $(x_i)$. As we shall later see, $X$ is dense in $X^\hyb$ (cf.~Corollary~\ref{cor:dense}), \ie any semivaluation can be realized as a vanishing order along some net of complex points. 

\medskip

By definition, the hybrid space comes with a natural continuous extension
$$
\pi\colon X^\hyb\to\DD
$$
of $\pi\colon X\to\DD^\times$, such that $\pi^{-1}(0)=X_0$. The construction is functorial, in that any morphism $X_K\to Y_K$ of algebraic $K$-varieties induces a continuous map $X^\hyb\to Y^\hyb$ over a sufficiently small disc $\DD$, compatible with composition. 


In analogy with~\S\ref{sec:hybtriv}, our next goal is to provide a description of the hybrid space in terms of Berkovich analytification, closely related to~\cite[\S 3]{BerNA}. As a consequence, we will obtain the following basic topological properties. 

\begin{prop}\label{prop:hybcomp} For any algebraic variety $X_K$ over $K$, the hybrid space $X^\hyb$ is Hausdorff and locally compact. If $X_K$ is further projective, then the map $\pi\colon X^\hyb\to\DD$ is proper.
\end{prop}

In the projective case, which is the one we will mostly consider afterwards, we will later provide another, more global description of the hybrid topology in terms of hybrid model functions (see Proposition~\ref{prop:hybchar}). 

For each $r\in (0,1)$ we denote by $K_r$ the ring of complex Laurent series 
$$
a(t)=\sum_{k\in\Z} a_k t^k\in\C\cro{t}
$$
that converge on the punctured closed disc $\bar\DD_r^\times$ of radius $r$, \ie such that $\sum_k|a_k| r^k<\infty$, where $|a_k|=|a_k|_\infty$ is the usual complex absolute value. Since 
$$
\|a_k\|^\hyb=\max\{|a_k|,|a_k|_0\}\le 1+|a_k|,
$$
the \emph{hybrid norm}
$$
\|a\|^\hyb_r:=\sum_k\|a_k\|^\hyb r^k
$$
is then finite; we denote by 
$$
K_r^\hyb:=(K_r,\|\cdot\|_r^\hyb)
$$
the resulting normed ring. As one easily sees, $K_r$ coincides with the ring of (a priori) doubly infinite series 
$$
\left\{a=\sum_{k\in\Z} a_k t^k\in\C\cro{t^\pm}\mid\sum_{k\in\Z}\|a_k\|^\hyb r^k<\infty\right\}. 
$$
In the language of Berkovich geometry, the spectrum $\cM(K_r^\hyb)$ thus corresponds to the circle 
$$
C_r^\hyb:=\{|t|=r\}\subset\A^{1,\hyb}_\C=(\Spec\C[t])^\hyb
$$
in the hybrid affine line. As observed in~\cite[Proposition~A.4]{konsoib}, this hybrid circle is topologically a disc: 

\begin{lem}\label{lem:hybspec} For any $r\in (0,1)$, there is a canonical homeomorphism 
$$
\bar\DD_r\simto\cM(K_r^\hyb),
$$
which takes $0\in\bar\DD_r$ and $t\in\bar\DD_r^\times$ to the multiplicative seminorms 
$$
K_r\ni a\mapsto r^{v_K(a)}\quad\text{and}\quad K_r\ni a\mapsto |a(t)|^{\e_{r,t}},\quad\e_{r,t}:=\frac{\log r}{\log|t|}\in (0,1].
$$
\end{lem}

\begin{proof} Pick $a\in K_r$ and set $k_0:=v_K(a)$. For any $t\in\bar\DD_r\setminus\{0\}$ we have 
$$
\e_{r,t}\log |a(t)|\le\e_{r,t}\log\sum_{k\ge k_0} |a_k| |t|^k\le\e_{r,t}\log\sum_{k\ge k_0}\|a_k\|^\hyb |t|^k
$$
$$
=\e_{r,t}\log\sum_{k\ge k_0} \|a_k\|^\hyb |t|^{k-k_0}+k_0\log r\le\e_{r,t}\log \sum_{k\ge k_0}\|a_k\|^\hyb r^{k-k_0}+k_0\log r
$$
$$
\le\log\sum_{k\ge k_0}\|a_k\|^\hyb r^{k-k_0}+k_0\log r=\|a\|_r^\hyb,
$$
where the last inequality holds because $\e_{r,t}\le 1$ and
$$
0\le\log\|a_{k_0}\|^\hyb\le\log\sum_{k\ge k_0}\|a_k\|^\hyb r^{k-k_0}, 
$$
since $a_{k_0}\ne 0$. This yields $|a(t)|^{\e_{r,t}}\le\|a\|_r^\hyb$, which proves that the multiplicative seminorm $K_r\ni a\mapsto|a(t)|^{\e_{r,t}}$ belongs to $\cM(K_r)$. Furthermore, $|a(t)|^{\e_{r,t}}=r^{\frac{\log|a(t)|}{\log|t|}}$ depends continuously on $t\in\bar\DD_r\setminus\{0\}$, and converges pointwise to $r^{v_K(a)}$ as $t\to 0$. This shows that the map $\bar\DD_r\to\cM(K_r^\hyb)$ in the statement is well-defined and continuous. As noted above, $\cM(K_r^\hyb)$ can be viewed as the Berkovich circle of radius $r$ in the hybrid line $\A^{1,\hyb}_\C$. The Gelfand--Mazur theorem thus yields an identification of the fiber over $\e\in (0,1]$ of the structure map
$$
\cM(K_r^\hyb)\to\cM(\C^\hyb)\simeq [0,1]
$$
with the circle of radius $r$ with respect to the absolute value $|\cdot|_\infty^\e$ on $\C$, \ie the usual circle in $\C$ of radius $r^{1/\e}$, while the fiber over $0$ consists of the non-Archimedean absolute value $r^{v_K}$. This implies that $\bar\DD_r\to\cM(K_r^\hyb)$ is bijective, and hence a homeomorphism, by compactness. 
\end{proof}
Since $K=\bigcup_{r\in (0,1)} K_r$, the $K$-scheme $X_K$, which is separated and of finite type, induces for each $0<r\ll 1$ a $K_r$-scheme $X_{r}$ that is separated and of finite type as well, and projective when $X$ is projective. We denote by 
$$
X_{K_r}^\hyb\to\cM(K_r^\hyb)
$$
its analytification with respect to the norm $\|\cdot\|^\hyb_r$. By Proposition~\ref{prop:Berktop}, $X_{K_r}^\hyb$ is Hausdorff, locally compact, and compact if $X$ is projective over $K$. Proposition~\ref{prop:hybcomp} is thus a consequence of the following result: 

\begin{thm}\label{thm:hybcomp} For any $0<r\ll 1$, the preimage of $\bar\DD_r$ under $\pi\colon X^\hyb\to\DD$ admits a canonical homeomorphism $\pi^{-1}(\bar\DD_r)\simto X_{K_r}^\hyb$ compatible with the projection to $\bar\DD_r\simto\cM(K_r^\hyb)$. 
\end{thm} 
\begin{proof} Since $X_K$ is the base change of $X_{r}$ with respect to the inclusion $K_r\subset K$, which is dense for the $r^{v_K}$-topology, Proposition~\ref{prop:basean} shows that the fiber of the structure map 
 $$
 X_{r}^\hyb\to\cM(K_r^\hyb)
 $$
 over $r^{v_K}\in\cM(K_r^\hyb)$ coincides with the analytification of $X_K$ with respect to the absolute value $r^{v_K}$ on $K$, and hence can be canonically identified with the space of semivaluations $v\in X_0$, using the map $v\mapsto r^v$. Combined with Lemma~\ref{lem:hybspec} and the Gelfand--Mazur theorem, this yields a set theoretic identification $\pi^{-1}(\bar\DD_r)\simto X_{K_r}^\hyb$, compatible with the projections to $\bar\DD_r\simto\cM(K_r^\hyb)$. 
 
By definition of the topology of $X_{r}^\hyb$, $\pi^{-1}(\bar\DD_r^\times)\hto X_{r}^\hyb$ is an open embedding, $X_0\hto X_{r}^\hyb$ is a closed embedding, and a net of complex points $x_i\in X_{t_i}$ with $t_i\to 0$ converges in $X_{r}^\hyb$ to a semivaluation $v\in X_0$ iff, for any affine open $U_{K_r}\subset X_{r}$ and $f\in\cO(U_{K_r})$, we have $x_i\in U_{t_i}$ for $i$ large enough and 
$$
|f(x_i)|^{\e_{r,t_i}}\to r^{v(f)}.
$$
Since $\e_{r,t}=\frac{\log r}{\log|t|}$, the last condition is equivalent to~\eqref{equ:hybcv}, except that we now require $f$ to be defined over $K_r\subset K$. Comparing with Definition~\ref{defi:hybrid}, we infer that the topology of $X^\hyb$ is the coarsest one such that $\pi^{-1}(\DD_r)\to\ X_{r}^\hyb$ is continuous for all $0<r\ll 1$. For $0<s\le r\ll 1$, the injection $K_r\to K_s$ yields a normed ring map $K_r^\hyb\to K_s^\hyb$ with dense image. Using Proposition~\ref{prop:basean} again, we conclude that $X_{K_s}^\hyb\to X_{K_r}^\hyb$ is a homeomorphism onto the preimage of $\bar\DD_s$, and the result follows. 
\end{proof}

%
%
\subsection{Hybrid continuous metrics}\label{sec:hybridlim}
Consider as above a meromorphic degeneration $\pi\colon X\to\DD^\times$ associated to an algebraic variety $X_K$ over the field $K=\C\{t\}$ of convergent Laurent series, and assume given a \emph{meromorphic line bundle} $L$ on $X$, \ie a holomorphic line bundle induced by a line bundle $L_K$ on $X_K$. For each $t\in\DD^\times$, we denote by $L_t$ the induced holomorphic line bundle on $X_t=\pi^{-1}(t)$, and by $L_0=L_K^\an$ the induced line bundle on $X_0=X_K^\an$. 

For $0<r\ll 1$ we get an induced line bundle $L_{K_r}$ on the $K_r$-scheme $X_{K_r}$, and we can thus consider continuous metrics on the analytification $L_{K_r}^\hyb$ over $X_{K_r}^\hyb$ in the general sense of \S\ref{sec:metrics}. Thanks to Theorem~\ref{thm:hybcomp}, this admits the following more concrete description. 

\begin{defi}\label{defi:hybridmetric} We say that a family $\phi=(\phi_t)_{t\in\DD}$ of continuous metrics $\phi_t$ on each $L_t$ is \emph{hybrid continuous} if, for any Zariski open $U_K\subset X_K$ and $s\in\Hnot(U_K,L_K)$, the function 
$$
|s|_{\phi^\hyb}\colon U^\hyb\to\R_{\ge 0}
$$
equal to $|s|_{\phi_0}$ on $U_0$ and $|s|_{\phi_t}^{\e_t}$ on $U_t$ for $t\ne 0$ is continuous in the hybrid topology. 
\end{defi}
It is enough to check the condition when $s$ is a trivializing section. Note that a hybrid continuous metric $\phi$ restricts to a continuous metric on the holomorphic line bundle $L$. 

\begin{exam}\label{exam:conf} A family $\phi=(\phi_t)_{t\in\DD}$ of continuous functions $\phi_t\in \Cz(X_t)$, viewed as metrics on the trivial line bundle, is hybrid continuous iff the rescaled function $\phi^\hyb\colon X^\hyb\to\R$ defined by 
$$
\phi^\hyb:=\left\{
\begin{array}{ll}
\phi_0 & \text{ on }X_0\\ 
\e_t\phi_t & \text{ on }X_t,\,t\ne 0
\end{array}
\right. 
$$
is continuous in the hybrid topology. 
\end{exam}

As a  consequence we get the following useful criterion: 

\begin{lem}\label{lem:hybcrit} Let $\phi=(\phi_t)_{t\in\DD}$, $\tphi=(\tphi_t)_{t\in\DD}$ be two families of continuous metrics on the $L_t$'s, and assume that
\begin{itemize}
\item $\tphi$ is hybrid continuous; 
\item $\phi_0=\tphi_0$.
\end{itemize}
The following are then equivalent:
\begin{itemize}
\item[(i)] $\phi$ is hybrid continuous;
\item[(ii)] $\phi|_L$ is continuous and $\e_t\sup_{X_t}|\phi_t-\tphi_t|\le 0$ as $t\to 0$. 
\end{itemize}
\end{lem}
The last condition holds as soon as $\phi_t-\tphi_t$ is uniformly bounded. 

\begin{exam}\label{exam:FShyb} Assume $s_1,\dots,s_N\in\Hnot(X_K,mL_K)$ have no common zeroes for some $m\ge 1$, and pick $\lambda_i\in\R$. Then the family of metrics 
$$
\phi_0=\frac 1m\max\{\log|s_i|+\lambda_i\},\quad\phi_t=\frac{1}{2m}\log\sum_i|s_i|^2|t|^{-2\lambda_i}\text{ for }t\ne 0
$$
is hybrid continuous. Indeed, setting
$$
\tphi_t:=\frac 1m\max_i\{\log|s_i|+\e_t^{-1}\lambda_i\}
$$
for all $t\in\DD$ defines a hybrid continuous family, corresponding to a tropical Fubini--Study metric on $L_{K_r}^\hyb$, see Example~\ref{exam:FS}. Now $|\phi_t-\tphi_t|\le\frac{\log N}{m}$ is bounded independently of $t$, and $(\phi_t)$ is thus hybrid continuous as well by Lemma~\ref{lem:hybcrit}.
 \end{exam}

%
%
\section{Models, non-Archimedean metrics and dual complexes}\label{sec:NAmetrics}
From now on, we consider a projective meromorphic degeneration $\pi\colon X\to\DD^\times$, associated to a projective variety $X_K$ over the field $K=\C\{t\}$ of convergent Laurent series. In this section we make a deeper study of the associated Berkovich space, denoted as above by $X_0=X_K^\an$. We introduce in this section several important classes of metrics and functions on $X_0$, and relate this space to the dual complexes of snc models, when $X_K$ is smooth. 
%
\subsection{Non-Archimedean Fubini--Study metrics}\label{sec:NAFS}
Let $L$ be a meromorphic line bundle on $X$, corresponding to a line bundle $L_K$ on $X_K$, and consider as in Example~\ref{exam:FS} a rational tropical Fubini--Study metric on $L_0$, \ie 
\begin{equation}\label{equ:FStrop}
\phi=\tfrac 1m\max_i\{\log|s_i|+\lambda_i\}
\end{equation}
for a finite set of sections $s_i\in\Hnot(X_K,mL_K)$, $m\in\Z_{>0}$, and $\lambda_i\in\Q$. Since any $v\in X_0=X_K^\an$ restricts to the valuation of $K=\C\{t\}$, and hence satisfies $v(t)=1$, we have $\log|t|\equiv-1$ on $X_0$, and it follows that $\phi$ can be rewritten as 
$$
\phi=\tfrac{1}{m'}\max_i\log\big|s'_i\big|
$$
where $m'=km$ with $k$ sufficiently divisible and $s'_i:=s_i^k t^{-k\lambda_i}\in\Hnot(X_K,kmL_K)$. In other words, we can assume $\lambda_i=0$ in~\eqref{equ:FStrop}. We shall simply say that $\phi$ is a  \emph{(non-Archimedean) Fubini--Study metric} on $L_0$, and denote by 
 $$
 \FS(L_0)\subset \Cz(L_0)
 $$
 the set of such metrics. By definition, this set is nonempty iff $L$ is \emph{semiample}, \ie $mL$ is basepoint free for $m$ sufficiently divisible. Taking $L$ is ample, Proposition~\ref{prop:DFS} shows that the $\Q$-vector space
$$
\DFS(X_0)\subset \Cz(X_0)
$$
of functions of the form $f=m(\phi-\p)$ with $\phi,\p\in\FS(L_0)$ and $m\in\Z_{>0}$ is independent of $L$, and is dense in $\Cz(X_0)$.

%
\subsection{Models of $X$}\label{sec:models}
The meromorphic degeneration $X\to\DD^\times$ admits an extension to a surjective and projective holomorphic map $\cX\to\DD$, called a \emph{model} of $X$, whose fiber $\cX_0$ over $0\in\DD$ is called the \emph{central fiber}. 

Indeed, pick an embedding $X_K\hto\P^N_K$. This corresponds to an embedding $X\hto\P^N\times\DD^\times$ over a possibly shrunk disc, and the closure $\cX$ of $X$ in $\P^N\times\DD$ is then a model of $X$. 

A model can equivalently be seen as a flat, projective scheme over the valuation ring $\cO_K=\cO_{\C,0}$, together with an identification of its generic fiber with $X_K$. 

Models are far from unique: any blow up of a model along a subvariety of the central fiber is also a model. We say that a model $\cX'$ \emph{dominates} another model $\cX$ if the canonical meromorphic map $\cX'\dashrightarrow\cX$ extends to a morphism $\cX'\to\cX$. This defines a partial order on the set of models (modulo canonical isomorphism), which is further inductive, \ie any two models can be dominated by a third one. 

When $X$ is normal, the normalization of any model is also a model. In the general case, the integral closure of any model in its generic fiber is also a model, which is said to be \emph{integrally closed} (see for instance~\cite[\S 1.4]{trivval}). Integrally closed models give rise to points of the Berkovich space $X_0$, as follows: 

\begin{exam} Assume $\cX$ (and hence $X$) is normal, and pick an irreducible component $E$ of $\cX_0$, with multiplicity $b_E:=\ord_E(\cX_0)$. Then 
\begin{equation}\label{equ:vE}
v_E:=b_E^{-1}\ord_E
\end{equation}
defines a valuation on the function field of $X_K$ that restricts to $v_K$ on $K$, and hence an element $v_E\in X_0$. More generally, this remains true when $X$ is possibly non-normal and  $\cX$ is integrally closed (see for instance~\cite[Theorem~1.13]{trivval}). 
\end{exam} 
Points of $X_0$ of the form $v_E$ as above are called \emph{divisorial valuations}. They form a dense subset of $X_0$.

%
\subsection{Models of a line bundle}\label{sec:modelslb}

Any meromorphic line bundle $L$ on $X$ can similarly be extended to a line bundle $\cL$ on some model $\cX$ of $X$. It is convenient to allow $\cL$ to be a $\Q$-line bundle, which means that $m\cL$ is a line bundle extending $mL$ for some $m\in\Z_{>0}$. This leads to: 

\begin{defi} A \emph{model} $\cL$ of a line bundle $L$ on $X$ is defined as a $\Q$-line bundle $\cL$ extending $L$ to a model $\cX$ of $X$.  
\end{defi}
We then say that $\cL$ is \emph{determined} on $\cX$. The pull-back of $\cL$ to any higher model of $X$ is also model of $L$. If $\cL$, $\cL'$ are models of $L,L'$ determined on the same model $\cX$, then $\cL\pm\cL'$ is a model of $L\pm L'$, where we use additive notation for tensor products of line bundles. 

\begin{exam} Every model of the trivial line bundle $\cO_X$, determined on a given model $\cX$ of $X$, is of the form $\cL=\cO_\cX(D)$ where $D$ is a $\Q$-Cartier divisor on $\cX$ which is \emph{vertical}, \ie supported in $\cX_0$. 
\end{exam}
We denote by $\VCar(\cX)$ the (finite dimensional) $\Q$-vector space of vertical $\Q$-Cartier divisors on $\cX$. When $\cX$ is nonsingular, it is simply generated by the irreducible components of $\cX_0$. 

Any two models $\cL$, $\cL'$ of the same line bundle $L$ can be pulled back to some common model $\cX$, and then differ by some $D\in\VCar(\cX)$, by linearity. 
%
\subsection{Centers}\label{sec:centers}
Let $\cX$ be a model of $X$. Viewed as an $\cO_K$-scheme, $\cX$ is projective (and hence proper), and the valuative criterion of properness thus implies that any semivaluation $v\in X_0$ admits a \emph{center} 
$$
c_\cX(v)\in\cX_0
$$
characterized as the unique (scheme) point $\xi\in\cX_0$ such that $v\ge 0$ on $\cO_{\cX,\xi}$ and $v>0$ on the maximal ideal. Equivalently, an affine Zariski open $\cU\subset\cX$ contains $c_\cX(v)$ iff $v(f)\ge 0$ for all $f\in\cO(\cU)$, and we then have $v(f)>0$ iff $f$ vanishes at $c_\cX(v)$. Centers are compatible with domination of models: if a model $\cX'$ dominates $\cX$, then $c_\cX(v)$ is the image of $c_{\cX'}(v)$. 

The corresponding complex subvariety 
$$
Z_\cX(v)=\overline{\{c_\cX(v)\}}\subset\cX_0
$$
can be roughly viewed as the `locus that matters' for $v$ on $\cX$. As an illustration, we have: 

\begin{lem}\label{lem:center} Assume $v\in X_0$ is the limit of in $X^\hyb$ of a net of complex points $x_j\in X_{t_j}$. Then every limit point $x\in\cX_0$ of $(x_j)$ in $\cX$ (for the complex topology) lies in $Z_\cX(v)$.
\end{lem}
\begin{proof} After passing to a subnet, assume $x_j\to x\in\cX_0$, and pick an affine Zariski open neighborhood $\cU$ of $x$ in $\cX$ (viewed as an $\cO_K$-scheme). For every $f\in\cO(\cU)$, $f(x_j)$ is (ultimately) bounded, and hence 
$$
v(f)=\lim_j\frac{\log|f(x_j)|}{\log|t_j|}\ge 0.
$$
This proves $c_\cX(v)\in\cU$, so that $Z_\cX(v)\cap\cU$ is the intersection of the zero loci of all $f\in\cO(U)$ such that $v(f)>0$, as noted above. Now each such $f$ ultimately satisfies $|f(x_j)|\le |t_j|^\e$ for any $\e\in(0,v(f))$, and hence $f(x)=\lim_j f(x_j)=0$. This shows $x\in Z_\cX(v)$ and concludes the proof. 
\end{proof}

\begin{exam} Pick a divisorial valuation $v\in X_0$. By definition, $v=v_E$ for an irreducible component $E$ of the central fiber of an integrally closed model $\cX'$, which can be assumed to dominate the given model $\cX$ (see~\eqref{equ:vE}). Then $Z_{\cX'}(v)=E$, and $Z_\cX(v)$ is thus the image of $E$ in $\cX$.
\end{exam}

%
\subsection{Model metrics}\label{sec:NAmodelmetr}

As a key construction in non-Archimedean geometry, every model $\cL$ of a line bundle $L$ gives rise to a continuous metric $\phi_\cL$ on $L_0$, called a \emph{model metric}. The metric $\phi_\cL$ is unchanged upon pulling back $\cL$ to a higher model, and depends linearly on $\cL$ (in additive notation for $\Q$-line bundles and metrics). From a complex geometric perspective, the model metric $\phi_\cL$ can be understood as the `hybrid limit' of any family of metrics on $L_t$ that extends continuously to $\cL$ (see Proposition~\ref{prop:hybmodel} below). 

To describe the construction of model metrics, note first that mapping a semivaluation $v\in X_0$ to its center (viewed as a scheme point of $\cX_0$) defines a map 
$$
c_\cX\colon X_0\to\cX_0
$$
with the somewhat unusual property of being \emph{anticontinuous}, \ie the preimage of a (Zariski) open subset is closed (and hence compact, since $X_0$ is compact by projectivity of $X$). 

By definition of the center, every function $f\in\cO(\cU)$ on a Zariski open subset $\cU$ of $\cX$ (viewed as an $\cO_K$-scheme) satisfies
$$
|f(v)|=e^{-v(f)}\le 1
$$
for each $v$ in the compact set $c_\cX^{-1}(\cU)$, the inequality being strict iff $f$ vanishes at $c_\cX(v)$, \ie along the subvariety $Z_\cX(v)$. In particular, $|f|\equiv 1$ on $c_\cX^{-1}(\cU)$ if $f$ is a unit on $\cU$. 

The model metric $\phi_\cL$ associated to a model $\cL$ of $L$ is defined by imposing a similar condition for sections. Assume first that $\cL$ is an honest line bundle on $\cX$, and pick a trivializing open cover $(\cU_i)$ of $\cX$ for $\cL$, with trivializing sections $s_i\in\Hnot(\cU_i,\cL)$. The compact subsets $c_\cX^{-1}(\cU_i)$ then cover $X_0$, and the model metric $\phi_\cL$ is defined by requiring 
$$
|s_i|_{\phi_\cL}\equiv 1\text{ on } c_\cX^{-1}(\cU_i)
$$
for all $i$. The generic fibers $U_{i,K}$ of the $\cU_i$'s yield an open cover $(U_{i,0})_i$ of $X_0$, and we get
$$
|s_i|_{\phi_\cL}=|s_i/s_j|\text{ on }U_{i,0}\cap c_\cX^{-1}(\cU_j)
$$
for all $i,j$, which uniquely determines the continuous function $|s_i|_\cL$ on $U_{i,0}$, and hence the metric $\phi_\cL$.  

When $\cL$ is merely a $\Q$-line bundle $\cL$, $m\cL$ is a line bundle for some $m\in\Z_{>0}$, hence defines model metric $\phi_{m\cL}$ on $mL_0$, and we set
$$
\phi_{\cL}:=m^{-1}\phi_{m\cL}. 
$$
Fubini--Study metrics form a special class of model metrics: 

\begin{prop}\label{prop:FSmodel} Assume $L$ is semiample. Then $\FS(L_0)$ coincides with the set of model metrics $\phi_\cL$ with $\cL$ a semiample model of $L$. 
\end{prop}
\begin{proof} Assume first $\phi$ is a Fubini--Study metric, and write $\phi=\tfrac 1m\max_i\log|s_i|$ for a finite set of sections $s_i\in\Hnot(X_K,mL_K)$ without common zeroes. Replacing $L$ with $mL$ we may further assume $m=1$. The morphism $X\to\P^N_K$ with homogeneous coordinates $(s_i)$ extends to a morphism $\cX\to\P^N_{\cO_K}$ for some model $\cX$ of $X$. Denoting by $\cL$ the pullback $\cO(1)$, each $s_i$ extends to a section $s_i\in\Hnot(\cX,\cL)$. Then $\cU_i:=\{s_i\ne 0\}$ defines a trivializing open cover of $\cX$ for $\cL$. For all $i,j$, $s_j/s_i$ is a regular function on $\cU_i$, and hence $|s_j/s_i|\le 1$ on $c_\cX^{-1}(\cU_i)$, with equality when $i=j$. Thus 
$$
|s_i|_\phi=\left(\max_j|s_j/s_i|\right)^{-1}\equiv 1
$$
on $c_\cX^{-1}(\cU_i)$, and it follows that $\phi=\phi_\cL$. Conversely, if $\phi=\phi_\cL$ where $m\cL$ is generated by finitely many sections $(s_i)$, then the same argument yields $m\phi_\cL=\phi_{m\cL}=\max_i\log|s_i|$. 
\end{proof}
%
%
\subsection{Model functions}\label{sec:modelfunc}
Pick a model $\cX$ of $X$ and $v\in X_0$. By definition of the center $\xi=c_\cX(v)$, $v$ vanishes on each unit of $\cO_{\cX,\xi}$. This allows to evaluate $v$ on any vertical Cartier divisor $D$ on $\cX$ by setting
$$
v(D):=v(f_D)\in\R
$$
for any choice of local equation $f_D$ of $D$ at $\xi$, which yields a linear map $D\mapsto v(D)$ that can be extended to all vertical $\Q$-Cartier divisors $D\in\VCar(\cX)$. 

If $D$ is further effective, then $v(D)\ge 0$, and $v(D)>0$ iff $c_\cX(v)$ lies in (the support of) $D$. 

\begin{exam} By definition, any $v\in X_0$ coincides with $v_K$ on $K$, and the vertical Cartier divisor $\cX_0$ thus satisfies $v(\cX_0)=1$. 
\end{exam}

On the other hand, each $D\in\VCar(\cX)$ determines a model $\cO_\cX(D)$ of $\cO_X$, and hence a model metric on the trivial line bundle, \ie a \emph{model function}
$$
\phi_D\in \Cz(X_0).
$$
Tracing the definitions, one checks that the two constructions are compatible, \ie $\phi_D(v)=v(D)$ for all $v\in X_0$.

\begin{prop}\label{prop:modelfunc} The $\Q$-vector space of model functions coincides with the dense subspace $\DFS(X_0)\subset \Cz(X_0)$ (see~\S\ref{sec:NAFS}). 
\end{prop}
\begin{proof} Pick an ample line bundle $L$. By Proposition~\ref{prop:FSmodel}, every metric in $\FS(L_0)$ is a model metric, and every function in $\DFS(X_0)$, being a difference of model metrics on $L_0$, is thus a model function. Conversely, pick a model function $\phi_D$ with $D\in\VCar(\cX)$. After possibly replacing $\cX$ by a higher model of $X$, $L$ admits an ample model $\cL$ on $\cX$. For $m\gg 1$, $\cL+m^{-1} D$ is ample as well. By Proposition~\ref{prop:FSmodel}, 
$$
\phi:=\phi_{\cL+m^{-1}D},\quad\p:=\phi_\cL
$$
both lie in $\FS(L_0)$, and $\phi_D=m(\phi-\p)$ thus lies in $\DFS(X_0)$. 
\end{proof}

%
%
\subsection{Snc models and dual complexes}\label{sec:snc}
We assume here that the projective variety $X_K$ is \emph{smooth}, \ie the holomorphic map $\pi\colon X\to\DD^\times$ is a submersion, after perhaps shrinking $\DD$. As we now explain, one can then use the combinatorics of the central fibers of a special class of models to get a description of the Berkovich space $X_0$ as a limit of simplicial complexes. While this goes back to the fundamental work of Berkovich, we use instead~\cite[\S 3]{siminag} as a reference for what follows. 

By resolution of singularities, any model of $X$ is dominated by a regular model $\cX$ such that $\cX_0$ has simple normal crossing support. Denoting by
$$
\cX_0=\sum_{i\in I} b_i E_i
$$
its irreducible decomposition, this means that the $E_i$'s are smooth hypersurfaces that intersect transversally. After further blowing up intersections of the $E_i$'s, we may further achieve that each intersection 
$$
E_J:=\bigcap_{i\in J} E_i
$$
with $J\subset I$ is \emph{connected}. We then say that $\cX$ is an \emph{snc model}. The set of snc models is cofinal in the poset of all models; in particular, any model metric on a line bundle $L$ can be written as $\phi_\cL$ for a model $\cL$ of $L$ determined on an snc model $\cX$ of $X$.

A nonempty intersection $E_J$ with $J\subset I$ is called a \emph{stratum} of $\cX$, with \emph{relative interior} 
$$
\rE_J:=E_J\setminus\bigcup_{i\in I\setminus J} E_i. 
$$
Each $x\in\cX_0$ belongs to $\rE_J$ for a unique stratum $E_J$, with $J\subset I$ the set of components $E_i$ passing through $x$.

The combinatorics of the intersections of the components $E_i$ is encoded in the \emph{dual (intersection) complex}, \ie the simplicial complex $\D_\cX$ having a vertex $e_i$ for each $E_i$, an edge between $e_i$ and $e_j$ if $E_i$ and $E_j$ intersect, and so on. More precisely, we associate to each stratum $E_J$ the $(|J|-1)$-dimensional simplex 
\begin{equation}\label{equ:simplex}
\D_J:=\{w\in\R_{\ge 0}^J\mid \sum_{i\in J} b_i w_i=1\},
\end{equation}
and $\D_\cX$ is then defined by identifying $\D_J$ with a face of $\D_{J'}$ whenever $E_{J'}$ is a substratum of $E_J$. 

Equivalently, $\D_\cX$ can be globally realized as the intersection of the affine hyperplane 
\begin{equation}\label{equ:affhyp}
H_\cX:=\{w\in\R^I\mid \sum_{i\in I}b_i w_i=1\}\subset\R^I
\end{equation}
with the family of cones $\R_{\ge 0}^J\times\{0\}$ attached to the strata $E_J$. 
\begin{defi}\label{defi:NAlog} The \emph{non-Archimedean log map} 
$$
\Log_\cX\colon X_0\to H_\cX\subset\R^I
$$
associated to an snc model $\cX$ is defined as the map with components the model functions $\phi_{E_i}$, $i\in I$. 
\end{defi} 
In other words, we have 
$$
\Log_\cX(v)=(v(E_i))_{i\in I}. 
$$
for $v\in X_0$ (see~\S\ref{sec:modelfunc}). Since $v(E_i)\ge 0$, and $v(E_i)>0$ iff the center of $v$ is contained in $E_i$, the above description of the dual complex yields an inclusion
$$
\Log_\cX(X_0)\subset\D_\cX. 
$$
In analogy with  Lemma~\ref{lem:sec}, we have: 

\begin{prop}\label{prop:valglob} The map $\Log_\cX\colon X_0\to\D_\cX$ admits a unique continuous section 
$$
\val_\cX\colon\D_\cX\hto X_0
$$
with the property that, for each $w\in\mathring{\D}_J$, the center of $\val_\cX(w)$ coincides with the generic point of $E_J$. 
\end{prop}
In particular, this shows that $\Log_\cX(X_0)=\D_\cX$, which can thus be viewed as a `global tropicalization' of $X_K$ with respect to the toroidal embedding $X_K\hto\cX$ (compare~\S\ref{sec:degam}). 

\begin{proof} Pick a stratum $E_J$, denote by $\xi_J$ its generic point, and pick a local equations $z_i\in\cO_{\cX,\xi_J}$ of each $E_i$, $i\in J$. By the snc condition, $(z_i)_{i\in J}$ induces a system of coordinates (\ie a regular system of parameters) on $\cO_{\cX,\xi_J}$. After choosing a field of representatives, and each $f\in\cO_{\cX,\xi_J}$ can thus be expanded as a formal power series 
$$
f=\sum_{\a\in\N^J} a_\a z^\a
$$
where each $a_\a$ lies in the residue field\footnote{In fact, in a chosen field of representatives of the residue field.} of $\cX$ at $\xi_J$, \ie the function field of $E_J$. Setting 
$$
\val_\cX(w)(f):=\min_{a_\a\ne 0} \a\cdot w
$$
defines, for every $w\in\rsig_J$, a monomial valuation $\val_\cX(w)$ on the function field of $X_K$. Since $b\cdot w=1$, the restriction of $\val_\cX(w)$ to $K$ coincides with $v_K$, and hence $\val_\cX(w)\in X_0$. For each stratum $E_J$, we thus get a continuous sections $\val_\cX\colon\rsig_J\hto X_0$ of $\Log_\cX$, and one shows that they glue to a continuous section on $\D_\cX$. As in Lemma~\ref{lem:sec}, uniqueness follows from the fact that any valuation centered at the generic point of $E_J$ and whose values $(w_i)_{i\in J}$ on the $E_i$ are $\Q$-linearly independent, is necessarily monomial, and hence equal to $\val_\cX(w)$ (see~\cite[Theorem~3.1]{siminag}). 
\end{proof}

\begin{exam}\label{exam:embver} The embedding $\val_\cX\colon\D_\cX\hto X_0$ takes each vertex $e_i$ to the divisorial valuation $v_{E_i}$ (see~\eqref{equ:vE}). 
\end{exam}
The compact space
$$
\Sk(\cX):=\val_\cX(\D_\cX)\subset X_0
$$
is called the \emph{skeleton} of the snc model $\cX$. It comes with a continuous \emph{retraction}
$$
p_\cX\colon X_0\to\Sk(\cX)\subset X_0, 
$$
obtained by composing $\Log_\cX\colon X_0\to\D_\cX$ with the section $\val_\cX\colon\D_\cX\simto\Sk(\cX)$. 

\begin{rmk}\label{rmk:dualdlt} More generally, the dual complex $\D_\cX$ and the embedding 
$$
\val_\cX\colon\D_\cX\simto\Sk(\cX)\subset X_0
$$
can be defined for any \emph{dlt\footnote{This stands for \emph{divisorially log terminal}.} model} $\cX$, \ie a normal model such that the pair $(\cX,\cX_{0,\redu})$ has dlt singularities in the sense of the Minimal Model Program. Indeed, the latter condition implies in particular that $\cX$ is snc near the generic point of any stratum $E_J$, which is enough to construct $\D_\cX$ and $\val_\cX\colon\D_\cX\simto\Sk(\cX)\subset X_0$. In contrast, to define $\Log_\cX\colon X_0\to\D_\cX$ (or, equivalently, the retraction $p_\cX\colon X_0\to\Sk(\cX)$) we need to further assume that each component $E_i$ is $\Q$-Cartier, in order to make sense of $v(E_i)$ for $v\in X_0$ (compare~\cite[(2.4)]{NXY}). 
\end{rmk}

In order to analyze the dependence of the previous definitions on the snc model $\cX$, it is convenient to identify 
$$
H_\cX=\{w\in\R^I\mid b\cdot w=1\}\subset\R^I\simeq\VCar(\cX)^\vee
$$
with the affine hyperplane of linear forms on $\VCar(\cX)_\R$ that take value $1$ on $\cX_0$. Given any other snc model $\cX'$ that dominates $\cX$, the pullback map $\VCar(\cX)\hto\VCar(\cX')$ induces an affine linear surjection 
$$
\pi_{\cX,\cX'}\colon H_{\cX'}\twoheadrightarrow H_\cX,
$$
and the associated log maps 
$$
\Log_\cX\colon X_0\to H_\cX,\quad \Log_{\cX'}\colon X_0\to H_{\cX'}
$$
are related by 
$$
\Log_\cX=\pi_{\cX,\cX'}\circ\Log_{\cX'}. 
$$
In particular $\pi_{\cX,\cX'}(\D_{\cX'})=\D_\cX$. The uniqueness part of Proposition~\ref{prop:valglob} further implies $p_{\cX'}\circ\val_\cX=\val_\cX$, which is equivalent to 
$$
\Sk(\cX)\subset\Sk(\cX')\subset X_0. 
$$
The following result, originally due to Berkovich, is now an easy consequence of the above constructions (see for instance~\cite[Corollary~3.2]{siminag}). 

\begin{thm} The log maps of all snc models induce a homeomorphism 
$$
X_0\simto\varprojlim_\cX\D_\cX.
$$
\end{thm} 
Following Sam Payne~\cite{Pay}, we summarize this by saying that \emph{analytification is the limit of tropicalization}.

%
%
\section{Hybrid model functions and hybrid log maps}\label{sec:hybmod}
%
%
We introduce in this section `hybrid versions' of the classes of non-Archimedean metrics and functions considered in~\S\ref{sec:NAmetrics}, and show that they determine the hybrid topology. 

As in the previous section, we consider a projective meromorphic degeneration $\pi\colon X\to\DD^\times$ over a small enough punctured disc, associated to a projective variety $X_K$ over the field $K=\C\{t\}$ of convergent Laurent series. Recall that $\pi\colon X^\hyb\to\DD$ is proper (see Proposition~\ref{prop:hybcomp}). 

%
\subsection{Hybrid model metrics}\label{sec:hybmodmetr}
Consider a meromorphic line bundle $L$ on $X$, associated to a line bundle $L_K$ on $X_K$. Recall from~\S\ref{sec:NAmodelmetr} that any model $\cL$ of $L$ induces a (continuous, non-Archimedean) model metric $\phi_\cL$ on $L_0$. 

\begin{defi} We define a \emph{hybrid model metric} on $L$ as a family $\phi=(\phi_t)_{t\in\DD}$ of continuous metrics on each $L_t$ such that 
\begin{itemize}
\item $\phi_0$ is a model metric;
\item $\phi|_L$ is the restriction of a smooth metric $\Phi$ on some model $\cL$ of $L$ such that $\phi_0=\phi_\cL$. 
\end{itemize}
\end{defi} 

\begin{lem}\label{lem:unique} For any hybrid model metric $\phi$, the model metric $\phi_0$ is uniquely determined by $\phi|_L$. 
\end{lem}
\begin{proof} Assume that $\phi|_L$ is the restriction of a smooth metric $\Phi'$ on another model $\cL'$ of $L$. We need to show $\phi_\cL=\phi_{\cL'}$. After pulling back $\cL$, $\cL'$ to a high enough model of $X$, we may assume that $\cL$, $\cL'$ are both determined on the same model $\cX$ of $X$, and hence $\cL'-\cL=D$ for some $D\in\VCar(\cX)$. The zero function, viewed as a metric on $\cO_X\simeq\cO_\cX(D)|_X$ extends to the smooth metric $\Phi-\Phi'$ on $\cO_\cX(D)$. As a consequence,  for any local equation $f_D$ of $D$, $\log|f_D|$ is locally bounded. If $\cX$ is normal, this implies that $D=0$. In the general case, we get that the pullback of $D$ to the integral closure of $\cX$ in its generic fiber vanishes (compare~\cite[Lemma~1.23]{trivval}). This implies $\phi_D=0$, and hence $\phi_{\cL'}=\phi_\cL$. 
\end{proof}

In analogy with Proposition~\ref{prop:FSmodel}, we have: 

\begin{exam}\label{exam:hybFS} Assume $s_0,\dots,s_N\in\Hnot(X_K,mL_K)$ have no common zeroes for some $m\ge 1$, and consider the \emph{hybrid Fubini--Study metric} $\phi=(\phi_t)_{t\in\DD}$ defined by
$$
\phi_0=\tfrac 1m\max_i\log|s_i|,\quad\phi_t=\tfrac{1}{2m}\log\sum_i|s_i|^2,\,t\ne 0
$$
on $L_0$ and $L_t$, respectively. Then $\phi$ is a hybrid model metric. Indeed, the map $X\to\P^N\times\DD^{\times}$ with homogeneous coordinates $(s_i)$ extends to a morphism $\cX\to\P^N\times\DD$ for some model $\cX$ of $X$, and $\phi$ is induced by the pullback by this morphism of $\frac 1m\cO(1)$ with its usual (Hermitian) Fubini--Study metric. 
\end{exam}
By Example~\ref{exam:FShyb}, hybrid Fubini--Study metrics are hybrid continuous (see Definition~\ref{defi:hybridmetric}). More generally: 

\begin{prop}\label{prop:hybmodel} Every hybrid model metric $\phi$ is hybrid continuous. 
\end{prop}

\begin{proof} Pick a model $\cL$ of $L$ to which $\phi$ extends smoothly. Since any two smooth metrics on $\cL$ differ by a smooth function on $\cX$, which is thus bounded near $\cX_0$, it suffices to produce \emph{some} smooth metric on $\cL$ for which the result holds true (see Lemma~\ref{lem:hybcrit}). By linearity, writing $\cL$ as a difference of ample line bundles, we may thus assume without loss that $\cL$ is ample. For $m$ sufficiently divisible, we can then pick a finite set of sections $s_i\in\Hnot(\cX,m\cL)$ without common zeroes, and the corresponding hybrid Fubini--Study metric does the job. 
\end{proof}

%
%
%
%
%

%
\subsection{Hybrid model functions} 
Recall that the data of a vertical $\Q$-divisor $D\in\VCar(\cX)$ on a model $\cX$ determines a model function $\phi_D\in \Cz(X_0)$, corresponding to the model metric on the trivial line bunde $L=\cO_X$ induced by its model $\cL=\cO_\cX(D)$, and explicitly given by $\phi_D(v)=v(D)$ for $v\in X_0$, see~\S\ref{sec:modelfunc}. 

Now choose a smooth metric $\Psi$ on $\cO_\cX(D)$. This defines a hybrid model metric on the trivial line bundle, which is hybrid continuous by Proposition~\ref{prop:hybmodel}, and hence determines a continuous function 
$$
\phi_D^\hyb\in \Cz(X^\hyb)
$$
which we call a \emph{hybrid model function}. By Example~\ref{exam:conf}, it is explicitly given by 
\begin{equation}\label{equ:hybmod}
\phi_D^\hyb:=\left\{
\begin{array}{ll}
\frac{\log|\sigma_D|_{\Psi}}{\log|t|}& \text{ on }X_t\\ 
\phi_D& \text{ on }X_0
\end{array}
\right. 
\end{equation}
with $\sigma_D$ the canonical meromorphic section of $\cO_\cX(D)$. Strictly speaking, this only makes sense when $D$ is a Cartier divisor, but it is convenient to keep the same notation when $D$ is a $\Q$-divisor, by setting 
$$
|\sigma_D|_\Psi:=|\sigma_{mD}|_{m\Psi}^{1/m}
$$
with $m\in\Z_{>0}$ sufficiently divisible. 

Unlike $\phi_D$, the hybrid model function $\phi_D^\hyb$ does depend on the choice of continuous metric $\Psi$ on $\cO_\cX(D)$, but only up to $O(\e_t)$ (see~\eqref{equ:et}). Furthermore, $D\mapsto\phi_D^\hyb$ is linear, again up to $O(\e_t)$. 

%
%

As we next show, hybrid model functions characterize the hybrid topology: 

\begin{prop}\label{prop:hybchar} The hybrid topology of $X^\hyb=X\coprod X_0$ is the coarsest topology such that:
\begin{itemize}
\item[(i)] $X\hto X^\hyb$ is an open embedding;
\item[(ii)] every hybrid model function $\phi_D^\hyb$ is continuous on $X^\hyb$. 
\end{itemize}
\end{prop}
The choice of metric on $\cO_\cX(D)$ is immaterial in (ii), since it only affects $\phi_D^\hyb$ by $O(\e_t)$. 
\begin{proof} Denote by $\cT$ the coarsest topology for which (i) and (ii) hold. By continuity of hybrid model functions, $\cT$ is coarser than the hybrid topology. For any closed disc $\bar\DD_r\subset\DD$, the hybrid topology of $\pi^{-1}(\bar\DD_r)\subset X^\hyb$ is compact Hausdorff, by Proposition~\ref{prop:hybcomp}. As is well-known, any compact Hausdorff topology on a given set is minimal among all Hausdorff topologies, and it therefore suffices to show that the topology induced by $\cT$ on $\pi^{-1}(\bar\DD_r)$ is Hausdorff. By (i), $\cT$ separates the points of $X$. It therefore suffices to show that points of $X_0$ are separated by $\cT$, which follows from the fact that the (non-Archimedean) model functions $\phi_D$ separate the points of $X_0$ (see Proposition~\ref{prop:modelfunc}).
\end{proof}

As an interesting consequence, we get:

\begin{cor}\label{cor:dense} For any (non-necessarily projective) meromorphic degeneration $Z\to\DD^\times$ of algebraic varieties, the complex part $Z$ is dense in $Z^\hyb$.
\end{cor}
For algebraic families over an algebraic curve, this also follows from~\cite[Theorem~C']{Jon}.

\begin{proof} The general case reduces to the affine case, which in turn reduces to projective case by passing to the closure in an ambient projective space. The result now follows from Proposition~\ref{prop:hybchar}, together with the fact that hybrid model functions are uniquely determined by their restriction to $X$, as a consequence of Lemma~\ref{lem:unique}. 
\end{proof}

This yields in turn the following generalization of Lemma~\ref{lem:unique}. 

\begin{cor}\label{cor:unique} For any hybrid continuous metric $\phi$, the non-Archimedean metric $\phi_0\in \Cz(L_0)$ is uniquely determined by the metrics $\phi_t\in \Cz(L_t)$ with $t\ne 0$. 
\end{cor}

\begin{rmk} Arguing as in Lemma~\ref{lem:unique}, one can show that the restriction to $X$ of a hybrid model function $\phi_D^\hyb$ only extends continuously to some model $\cX$ of $X$ in the trivial case $\phi_D=0$. Conversely, the restriction to $X$ of a continuous function on a model $\cX$ typically does not continuously extend to $X^\hyb$: if for instance $X=Z\times\DD^\times$ is a constant family with fiber $X_t=Z$ and $f\in \Cz(X)$ factors through $Z$, then $f$ extends to the trivial model $\cX=Z\times\DD$, but $f$ extends to $X^\hyb$ iff $f$ is constant (this can be checked using for instance Proposition~\ref{prop:cxvshyb} below). 
\end{rmk}

%
%
\subsection{Hybrid and local log maps} 
From now on, we assume that $X_K$ is smooth, \ie $\pi\colon X\to\DD^\times$ is a submersion (after perhaps shrinking $\DD$). 

As in~\S\ref{sec:snc}, consider an snc model $\cX$ with central fiber $\cX_0=\sum_{i\in I} b_i E_i$. The canonical sections 
$$
\sigma_i\in\Hnot(\cX,\cO_\cX(E_i))
$$
satisfy $t=\prod_{i\in I} \sigma_i^{b_i}$. Pick a smooth metric $\Psi_i$ on each line bundle $\cO_\cX(E_i)$ over $\cX$ such that $|t|=\prod_i|\sigma_i|_{\Psi_i}^{b_i}$, and denote by $\phi_{E_i}^\hyb\in \Cz(X^\hyb)$ the associated hybrid model functions, see~\eqref{equ:hybmod}. 

\begin{defi} The \emph{hybrid log map} 
$$
\Log_\cX\colon X^\hyb\to H_\cX=\{b\cdot w=1\}\subset\R^I
$$
is defined as the map with components $\phi^\hyb_{E_i}$, $i\in I$ (see~\eqref{equ:affhyp}). 
\end{defi}
The hybrid log map is thus continuous, its restriction to $X_0$ coincides with the (canonical) non-Archimedean log map (Definition~\ref{defi:NAlog}), while its restriction to $X$ has components
$$
\frac{\log|\sigma_i|_{\Psi_i}}{\log|t|},
$$
see~\eqref{equ:hybmod}. In analogy with~\S\ref{sec:degam}, one can view $\Log_\cX(X_t)\subset H_\cX$ as a `global amoeba' of $X_K$, converging to the `global tropicalization' $\Log_\cX(X_0)=\D_\cX$ as $t\to 0$. 

As with hybrid model functions, the hybrid log map depends on the choice of metrics $\Psi_i$, but only up to $O(\e_t)$ on $X_t$. For any higher snc model $\cX'$, we futher have 
\begin{equation}\label{equ:hyblogcomp}
\Log_\cX=\pi_{\cX,\cX'}\circ\Log_{\cX'}+O(\e_t),
\end{equation}
with $\pi_{\cX,\cX'}\colon H_{\cX'}\twoheadrightarrow H_\cX$ the induced affine map. 

As a consequence of Proposition~\ref{prop:hybchar}, hybrid log maps also characterize the hybrid topology: 

\begin{prop}\label{prop:logchar} The hybrid topology of $X^\hyb=X\coprod X_0$ is the coarsest topology such that 
\begin{itemize}
\item[(i)] $X\hto X^\hyb$ is an open embedding;
\item[(ii)] for each snc model $\cX$, the hybrid log map $\Log_\cX\colon X^\hyb\to H_\cX$ is continuous. 
\end{itemize}
\end{prop}
Again, the implicit choice of metrics is immaterial in (ii), since $\Log_\cX$ is unique up to $O(\e_t)$. 

\begin{proof} According to Proposition~\ref{prop:hybchar}, it suffices to see that (ii) is equivalent to the continuity of all hybrid model functions, assuming (i) holds. Each such function is of the form $\phi_D^\hyb$ for a vertical $\Q$-Cartier divisor $D$ on a model $\cX$ and a smooth metric on $\cO_\cX(D)$. Pulling back $D$ and the metric to a higher model does not modify $\phi_D^\hyb$, and we may thus assume that $\cX$ is snc. Then $D=\sum_{i\in I} a_i E_i$ with $a_i\in\Q$, and hence $\phi_D^\hyb=\sum_i a_i \phi^\hyb_{E_i}+O(\e_t)$. The continuity of the hybrid log map $\Log_\cX=(\phi_{E_i}^\hyb)_{i\in I}$ is thus equivalent to that of $\phi_D^\hyb$ for all $D\in\VCar(\cX)$, and the result follows. 
\end{proof}

Hybrid log maps admit a more concrete description in terms of local log maps. To this end, pick $x\in\cX_0$. Denote by $J_x\subset I$ the set of $E_i$'s containing $x$, and by $E_x:=E_{J_x}$ he unique stratum containing $x$ in its relative interior, with dual face $\D_x:=\D_{J_x}$. 

\begin{defi}\label{defi:adapted} An \emph{adapted chart} at $x$ is defined as a holomorphic coordinate chart 
$$
\cU\simeq\DD^{n+1}=\DD^{J_x}\times\DD^{J_x^c}
$$
of $\cX$ centered at $x$ such that $t=\prod_{i\in J_x} z_i^{b_i}$, with $z_i$ a local equation of $E_i$. 
\end{defi}
The fibers $\cU_t=\cU\cap X_t$, $t\ne 0$, are thus modelled on the product of the toric model $Z_t$ of~\S\ref{sec:toric} with $\DD^{J_x^c}$, while
$$
\cU_0\cap E_{J}\simeq\{0\}\times\DD^{J^c}\subset\DD^{J}\times\DD^{J^c}=\DD^{n+1}
$$
for any stratum $E_{J}$ containing $E_{J_z}$, \ie $J\subset J_x$. We define the associated \emph{local log map} 
$$
\Log_t\colon\cU_t\to\rsig_x
$$
as the map with components $\frac{\log|z_i|}{\log|t|}$, $i\in J_x$, \ie the pullback of the toric log map $\Log_t\colon Z_t\to\R^{J_x}$ from~\S\ref{sec:toric}. 

For each $i\in J_x$, we have $\log|\sigma_i|_{\Psi_i}=\log|z_i|+O(1)$ locally uniformly near each point of $\cU_0$, and hence 
\begin{equation}\label{equ:logloc}
\Log_\cX=\Log_t+O(\e_t)
\end{equation}
on $\cU_t$. This implies:

\begin{lem}\label{lem:contained} Assume $\dim\D_x=n$, \ie $E_x=\{x\}$, and pick a compact subset $\Sigma\subset\rsig_x$. Then $\Log_\cX^{-1}(\Sigma)\cap X_t$ is contained in $\cU_t$ for all $t$ small enough. 
\end{lem}

For later use we also note: 
\begin{lem}\label{lem:retrlim} For all $f\in \Cz(X^\hyb)$ and $g\in \Cz(\D_x)$ we have 
$$
\limsup_{t\to 0}\sup_{\cU_t}|f-g\circ\Log_t|\le\sup_{p_\cX^{-1}(\D_x)}|f-g\circ p_\cX|,
$$
with $p_\cX\colon X_0\to\D_\cX$ the canonical retraction map. 
\end{lem}
\begin{proof} Set $F:=\Log_\cX^{-1}(\D_x)\subset X^\hyb$, and pick any continuous extension of $g$ to $H_\cX$. Since $f$ and $g\circ\Log_\cX$ are continuous on $X^\hyb$, we have 
$$
\limsup_{t\to 0}\sup_{F\cap X_t}|f-g\circ\Log_\cX|\le\sup_{F\cap X_0}\sup|f-g\circ\Log_\cX|, 
$$
and the result follows using~\eqref{equ:logloc} and $\Log_\cX|_{X_0}=p_\cX$. 
\end{proof}

%
%
\section{Hybrid convergence of measures}\label{sec:cvmeas}
As above, $\pi\colon X\to\DD^\times$ denotes the projective meromorphic degeneration associated to a smooth projective variety $X_K$. In this section we study the weak convergence in the hybrid space $X^\hyb=X\coprod X_0$ of measures on the complex fibers $X_t$ to a measure on the Berkovich space $X_0$. In particular, we provide a detailed treatment of~\cite{konsoib}, which deals with families of smooth volumes with analytic singularities, with Calabi--Yau degenerations as a main special case.

%
%
\subsection{General convergence criteria}
In what follows, we consider a family of measures\footnote{All measures in these notes are Radon measures, and are also positive unless otherwise specified.}  $\mu_t$ on the fibers $X_t$, $t\ne 0$, of the degeneration $X\to\DD^\times$. As a consequence of Propositions~\ref{prop:hybchar}, \ref{prop:logchar}, weak convergence in the hybrid space can be characterized as follows: 

\begin{lem}\label{lem:cvmeas} Given a measure $\mu_0$ on $X_0$, the following are equivalent:
\begin{itemize}
\item[(i)] $\mu_t\to\mu_0$ weakly in $X^\hyb$ as $t\to 0$; 
\item[(ii)] $\int_{X_t}\phi_D^\hyb\,\mu_t\to\int_{X_0}\phi_D^\hyb\,\mu_0$ for every hybrid model function $\phi_D^\hyb$;
\item[(iii)] for each snc model $\cX$, $(\Log_\cX)_\star\mu_t\to(\Log_\cX)_\star\mu_0$ weakly in $H_\cX$.
\end{itemize}
\end{lem}
Recal that $\Log_\cX\colon X^\hyb\to H_\cX$ denotes the hybrid log map, uniquely defined up to $O(\e_t)$, whose restriction $\Log_\cX\colon X_0\to\D_\cX$ is canonically defined and admits a canonical section $\val_\cX\colon\D_\cX\hto X_0$ with image the skeleton $\Sk(\cX)$. The case where the limit measure $\mu$ is supported in a skeleton can then be characterized as follows:

\begin{lem}\label{lem:cvmeas2} Assume given a compact set $\Sigma\subset X_0$ such that, for any sufficiently high snc model $\cX$, we have
\begin{itemize}
\item[(a)] $\Sigma$ is contained in $\Sk(\cX)$;
\item[(b)] $(\Log_\cX)_\star\mu_t$ converges weakly in $H_\cX$ to a measure $\sigma_\cX$ with support in $\Log_\cX(\Sigma)\subset\D_\cX$. 
\end{itemize}
Then $\mu_0:=(\val_\cX)_\star\sigma_\cX$ is independent of $\cX$, and $\mu_t\to\mu_0$ weakly in $X^\hyb$. 
\end{lem}
\begin{proof} For any $\cX'$ dominating $\cX$, \eqref{equ:hyblogcomp} shows that $(\pi_{\cX,\cX'})_\star\sigma_{\cX'}=\sigma_\cX$. Thus $(\val_{\cX'})_\star\sigma_{\cX'}$ is a measure supported in $\Sigma\subset\Sk(\cX)\subset\Sk(\cX')$, and mapped to $(\val_\cX)_\star\sigma_\cX$ by the retraction $p_\cX\colon X_0\to\Sk(\cX)$. The result follows. 
\end{proof}

%
%
\subsection{From complex to hybrid convergence}
In practice for us, condition (iii) of Lemma~\ref{lem:cvmeas} will be obtained from the following local criterion, inspired by~\cite[Theorem~B]{PS}. 

\begin{lem}\label{lem:PS} Pick an snc model $\cX$, and assume that:
\begin{itemize}
\item[(i)] $\mu_t$ converges weakly in the complex space $\cX$ to a measure $\mu_{\cX_0}$ on $\cX_0$;
\item[(ii)] for each stratum $E_J$ such that $\mu_{\cX_0}(\rE_J)>0$, there exists a probability measure $\sigma_J$ on $\mathring{\D}_J$ such that, for some (hence any) adapted chart $\cU$ at some point $x\in\rE_J$ with associated local log map $\Log_t\colon\cU_t\to\rsig_J$ and any $\chi\in C^0_c(\cU)$ we have
$$
(\Log_t)_\star(\chi\mu_t)\to\left(\int_\cU\chi\,\mu_{\cX_0}\right)\sigma_J
$$
weakly in $\D_J$. 
\end{itemize}
Then $(\Log_\cX)_\star\mu_t$ converges weakly in $H_\cX$ to $\sigma_\cX:=\sum_J\mu_{\cX_0}(\rE_J)\sigma_J$, where the sum ranges over all strata such that $\mu_{\cX_0}(\rE_J)>0$. 
\end{lem}
Note that (ii) trivially holds when $E_J=E_i$ is a component of $\cX_0$, since $\D_J$ is then reduced to a point. 

\begin{proof} Since $\cX_0=\coprod_J\rE_J$, we have
$$
\sigma_\cX(H_\cX)=\sum_{\mu_{\cX_0}(\rE_J)>0}\mu_{\cX_0}(\rE_J)=\mu_{\cX_0}(\cX_0)=\lim_{t\to 0}\mu_t(X_t)=\lim_{t\to 0}\left((\Log_\cX)_\star\mu_t\right)(H_\cX). 
$$
By a standard weak compactness argument, it therefore suffices to show that each test function $0\le f\in C^0_c(H_\cX)$ satisfies $\liminf_{t\to 0}\int_{H_\cX} f\,(\Log_\cX)_\star\mu_t\ge\int_{H_\cX} f\,\sigma_\cX$, \ie
\begin{equation}\label{equ:liminfmes}
\liminf_{t\to 0}\int_{H_\cX}( f\circ\Log_\cX)\mu_t\ge\sum_J\mu_{\cX_0}(\rE_J)\int_{\D_J} f\,\sigma_J.
\end{equation}
Pick $J$ such that $\mu_{\cX_0}(\rE_J)>0$, and a compact subset $\Sigma_J\subset\rE_J$. We can then find finitely many adapted charts $\cU_\a$ for $\rE_J$ and $0\le\chi_\a\in C^\infty_c(\cU_\a)$ such that $\chi:=\sum_\a\chi_\a$ satisfies $0\le\chi\le 1$ and $\chi\equiv 1$ on $\Sigma_J$. Then 
$$
\int_{\Sigma_J}( f\circ\Log_\cX)\mu_t\ge\sum_\a\int_{\Sigma_J}( f\circ\Log_\cX)\chi_\a\mu_t. 
$$
Since $\Log_\cX=\Log_t+O(\e_t)$ uniformly on the support of $\chi_\a$ (see~\eqref{equ:logloc}), (ii) yields for each $\a$ 
$$
\int_{\Sigma_J}( f\circ\Log_\cX)\chi_\a\mu_t\to\left(\int_{\Sigma_J}\chi_\a\,\mu_{\cX_0}\right)\left(\int_{\D_J} f\,\sigma_J\right).
$$
Summing over $\a$, we get
$$
\liminf_{t\to 0}\int_{H_\cX}( f\circ\Log_\cX)\mu_t\ge\left(\int_{\cX_0}\chi\mu_{\cX_0}\right)\left(\int_{\D_J} f\,\sigma_J\right)\ge\mu_{\cX_0}(\Sigma_J)
\int_{\D_J} f\,\sigma_J. 
$$
Taking the supremum over $\Sigma_J$ and summing over $J$ yields~\eqref{equ:liminfmes}, and concludes the proof. 
\end{proof}

As a consequence, we recover the following result, originally due to Pille-Schneider~\cite[Theorem~B]{PS}, which will later be used to motivate the definition of the non-Archimedean Monge--Amp\`ere operator (see Theorem~\ref{thm:ddchyb}). 

\begin{prop}\label{prop:cxvshyb} Let $\mu_t$ be a family of measures on the fibers $X_t$, and assume given an snc model $\cX$ of $X$ such that: 
\begin{itemize}
\item[(i)] $\mu_t$ converges weakly in $\cX$ to a measure $\mu_{\cX_0}$ supported in $\cX_0$;
\item[(ii)] $\mu_{\cX_0}$ puts no mass on any nowhere dense Zariski closed subset of $\cX_0$. 
\end{itemize}
Then $\mu_t$ converges weakly in $X^\hyb$ to the atomic measure $\mu_0:=\sum_i\mu_{\cX_0}(E_i)\d_{v_{E_i}}$ where $E_i$ ranges over the irreducible components of $\cX_0$, with associated divisorial valuations $v_{E_i}\in X_0$. 
\end{prop}
\begin{proof} Pick an snc model $\cX'$ dominating $\cX$. Using (ii), one easily checks that $\mu_t$, viewed as a measure on $\cX'$, converges weakly to the unique measure $\mu_{\cX'}$ on $\cX'_0$ that coincides with $\mu_0$ on the strict transform $E'_i$ of each component $E_i$ of $\cX_0$, and is zero elsewhere. In particular, a stratum $E'_J$ of $\cX'_0$ satisfies $\mu_{\cX'}(\rE'_J)>0$ iff $E'_J=E_i'$ for some $E_i$. As noted above, (ii) of Lemma~\ref{lem:PS} trivially holds in that case, and we infer $(\Log_{\cX'})_\star\mu_t\to\sum_i \mu_{\cX'}(E_i')\d_{e_i}=(\Log_{\cX'})_\star\mu_0$ (see Example~\ref{exam:embver}). We conclude by Lemma~\ref{lem:cvmeas}. 
\end{proof}

%
%
%
\subsection{Families of volume forms with analytic singularities}\label{sec:volan}
Recall that, on any complex manifold $M$, there is a 1--1 correspondence, which we denote by
$$
\p\mapsto e^{2\p},
$$
between smooth Hermitian metrics $\p$ on the canonical bundle $K_M$ and (smooth, positive) volume forms on $M$. It satisfies 
\begin{equation}\label{equ:volform}
e^{2\p}=\frac{|\Om|^2}{|\Om|_\p^2}
\end{equation} 
for each local trivialization $\Om$ of $K_M$, where 
$$
|\Om|^2:=i^{(\dim M)^2}\Om\wedge\overline{\Om}
$$
is the associated volume form and $|\Om|_\p$ is the pointwise length of $\Om$ in the metric $\p$. 

\medskip

Following~\cite{konsoib}, we next consider a smooth family $(\sigma_t)$ of volume forms on the fibers $X_t$ of our degeneration $\pi\colon X\to\DD^\times$, assumed to have \emph{analytic singularities} at $t=0$, in the sense that the corresponding smooth Hermitian metric on the relative canonical bundle $K_{X/\DD^\times}$ extends to a smooth metric $\Psi$ on some model $\cK$ of $K_{X/\DD^\times}$, determined on a model $\cX$ of $X$ that can and will be assumed to be snc. In other words, $\sigma_t=e^{2\p_t}$ for a hybrid model metric $(\p_t)_{t\in\DD}$ on $K_{X/\DD^{\times}}$ with non-Archimedean limit the model metric $\p_0=\phi_\cK$, see~\S\ref{sec:hybmodmetr}.

\begin{exam}\label{exam:CY} Assume that $X$ is a \emph{Calabi--Yau degeneration}, in the sense that $K_{X/\DD^{\times}}\simeq\cO_X$ is trivial. A trivialization of the canonical bundle of $X_K$ induces a holomorphic family of holomorphic volume forms $\Om_t$ on the fibers $X_t$, and the associated family of smooth volume forms $\sigma_t:=|\Om_t|^2$ has analytic singularities at $t=0$ (with $\cK\simeq\cO_\cX$ equipped with the trivial metric). 
\end{exam}

Returning to the above general setup, we will prove: 

\begin{thm}\label{thm:konsoib}\cite[Theorem~A]{konsoib} The total mass of $\sigma_t$ satisfies 
$$
\sigma_t(X_t)\sim c|t|^{2\kappa}\e_t^{-d}
$$
with $c\in\R_{>0}$, $\kappa\in\Q$ and $d\in\{0,\dots,n\}$. Furthermore, the rescaled measure 
$$
\mu_t:=|t|^{-2\kappa}\e_t^d\sigma_t
$$
converges weakly both in $\cX$ and in $X^\hyb$, the limit in $X^\hyb$ being a Lebesgue-type measure on the top-dimensional part of a $d$-dimensional subcomplex $\D_\cK^\ess\subset\D_\cX\hto X_0$. 
 \end{thm}
The invariants $\kappa$, $\D_\cK^\ess$, and hence also $d=\dim\D_\cK^\ess$, which we call the \emph{essential dimension} of the degeneration $(\sigma_t)$, only depend on $\cK$, and can be described as follows. Denote as usual by $\cX_0=\sum_{i\in I} b_i E_i$ the irreducible decomposition of the central fiber. Consider the log canonical divisors
$$
K^{\log}_\cX:=K_\cX+\sum_i E_i,\quad K^{\log}_\DD:=K_\DD+[0]
$$
of the pairs $(\cX,\cX_{0,\redu})$, $(\DD,[0])$, and the \emph{relative log canonical bundle}
$$
K_{\cX/\DD}^{\log}=K^{\log}_\cX-\pi^\star K_{\DD}^{\log}=K_{\cX/\DD}+\cX_{0,\redu}-\cX_0. 
$$
Since $\cK$ and $K_{\cX/\DD}^{\log}$ are both models of $K_{X/\DD^\times}$, their difference can be uniquely written as 
\begin{equation}\label{equ:ai}
K_{\cX/\DD}^{\log}-\cK=\sum_i a_i E_i
\end{equation}
with $a_i\in\Q$. The above invariant $\kappa$ is then given by 
$$
\kappa:=\min_i \kappa_i\quad\text{with}\quad\kappa_i:=a_i/b_i,
$$
while $\D_\cK^\ess\subset\D_\cX$ is the subcomplex formed by the faces $\D_J$ of $\D_\cX$ that are \emph{essential}, in the sense that $\kappa_i=\kappa$ for all $i\in J$. 

\begin{exam}\label{exam:candeg} Given any snc model $\cX$ of $X$, the choice of a smooth metric on $\cK:=K_{\cX/\DD}^{\log}$ gives rise to a family of volume forms $\sigma_t$ with analytic singularities, whose total mass grows like $\e_t^{-\dim\D_\cX}$ (see~\cite[Theorem~3.6]{BHJ2} and~\cite[Lemma~7.6]{wYTD} for applications to the asymptotics of the Mabuchi K-energy functional along geodesic rays). 
\end{exam}

The image of $\D_\cK^\ess\subset\D_\cX$ under the canonical embedding $\val_\cX\colon\D_\cX\hto X_0$ is denoted by
$$
\Sk^\ess(\cK)\subset\Sk(\cX)\subset X_0, 
$$
and called the \emph{essential skeleton} of $\cK$. It satisfies the following important invariance property: 

\begin{lem}\label{lem:essinv} The essential skeleton $\Sk^\ess(\cK)\subset X_0$ only depends on the model metric $\phi_\cK$; in particular, it is invariant under pulling back $\cK$ to a higher snc model $\cX'\to\cX$. 
\end{lem}
Thus $\Sk^\ess(\cK)$ is uniquely determined by the family of volume forms $\sigma_t=e^{2\p_t}$ (see Lemma~\ref{lem:unique}). 

\begin{proof} The result follows from the existence of a canonically defined singular lsc metric $A_X$ on the (non-Archimedean) analytification $K_X^\an$ of the canonical bundle of $X_K$, known as the \emph{Temkin metric} \cite{MN,Tem}, such that the corresponding lsc function $A_X-\phi_\cK$ on $X_0$ achieves its infimum precisely on $\Sk^\ess(\cK)$. More specifically, denote by $\phi_\cX$ the model metric of $K_X^\an$ defined by the model $K_{\cX/\DD}^{\log}$. For any higher snc model $\rho\colon\cX'\to\cX$, 
$$
K_{\cX'/\DD}^{\log}-\rho^\star K_{\cX/\DD}^{\log}=(K_{\cX'}+\cX'_{0,\redu})-\rho^\star(K_\cX+\cX_{0,\redu})
$$
is uniquely represented by a $\rho$-exceptional divisor, whose coefficients compute the \emph{log discrepancies} of the pair $(\cX,\cX_{0,\redu})$. Since this pair is log smooth, standard properties of log discrepancies imply that $\phi_{\cX'}\ge\phi_\cX$, with equality precisely on $\Sk(\cX)$ (see for instance~\cite[Proposition~5.10]{konsoib}). The Temkin metric can be described as the limit 
$$
A_X:=\sup_\cX\phi_\cX
$$
of the non-decreasing net of model metrics $(\phi_\cX)_\cX$, where $\cX$ ranges over the poset all snc models of $X$. It is now easy to see that the function $A_X-\phi_\cK\colon X_0\to\R\cup\{+\infty\}$ is lsc, affine linear on each face of $\D_\cX\simeq\Sk(\cX)$, and satisfies 
\begin{equation}\label{equ:Dmin}
\Sk^\ess(\cK)=\left\{v\in X_0\mid (A_X-\phi_\cK)(v)=\inf_{X_0}(A_X-\phi_\cK)=\kappa\right\}. 
\end{equation}
\end{proof}
More generally, \eqref{equ:Dmin} remains true for any dlt model $\cX$ (see Remark~\ref{rmk:dualdlt}). This is especially useful in the Calabi--Yau case. Indeed, assume as in Example~\ref{exam:CY} that $K_X$ is trivial, and pick a trivialization $\Om$ of the canonical bundle of $X_K$. Then $\log|\Om|=\phi_\cK$ is a model metric. Since $\Om$ is unique up to a scalar in $K=\C\{t\}$, $\phi_\cK$ is unique up to an additive constant; thus~\eqref{equ:Dmin} shows that 
$$
\Sk^\ess(X):=\Sk^\ess(\cK)
$$
is independent of the choice of $\Om$, and called the \emph{essential skeleton} of the Calabi--Yau degeneration $X$. The Minimal Model Program (MMP) further guarantees the existence of a \emph{minimal dlt model} $\cX$, such that the pair $(\cX,\cX_{0,\redu})$ is dlt and its log canonical divisor $K_\cX^{\log}$ is trivial, and we then have 
\begin{equation}\label{equ:esssk}
\Sk^\ess(X)=\Sk(\cX),
\end{equation}
compare~\cite{NX}. 

\begin{exam}\label{exam:Fermat} Consider the family of Calabi--Yau Fermat hypersurfaces 
$$
X_t:=\{z_0\dots z_{n+1}+t(z_0^{n+2}+\dots+z_{n+1}^{n+2})\}\subset\P^{n+1}. 
$$
The given equation of $X$ provides a minimal dlt model $\cX\subset\P^{n+1}\times\DD$, and $\Sk^\ess(X)=\Sk(\cX)\simeq\D_\cX$ can thus be identified with the dual intersection complex of the union of the coordinate hyperplanes, \ie the boundary of an $(n+1)$-simplex. 
\end{exam}

%
%
\subsection{Residual measures and local convergence}
We return here to the general setup of~\S\ref{sec:volan}. After replacing $\cK$ with $\cK+\kappa\cX_0$ and $\sigma_t$ with $|t|^{-2\kappa}\sigma_t$, we may and do assume from now on that $\kappa=0$; equivalently,
$$
\min_i a_i=0
$$ 
with $a_i$ defined by~\eqref{equ:ai}. A face $\D_J$ of $\D_\cX$ is thus essential, \ie a face of $\D_\cK^\ess$, iff $a_i=0$ for all $i\in J$; we then also say that the corresponding stratum $E_J$ is essential.

\begin{lem}\label{lem:res} For each essential stratum $E_J$, there is a canonical identification 
$$
\cK|_{E_J}=K_{E_J}+B_J\quad\text{where}\quad B_J:=\sum_{i\in I\setminus J} \left(1-a_i\right) E_i|_{E_J}. 
$$
Furthermore, $B_J$ has coefficients $<1$ iff $E_J$ is minimal as an essential stratum, \ie $\D_J$ is a maximal face of $\D_\cK^\ess$. 
\end{lem}
\begin{proof} The first point follows from~\eqref{equ:ai} together with the adjunction formula $K_{E_J}=\left(K_\cX+\sum_{i\in J} E_i\right)|_{E_J}$. If $E_J$ is minimal, each $E_i$ meeting $E_J$ properly necessarily satisfies $a_i>0$, which yields the second point. 
\end{proof}
Recall that $\Psi$ denotes the given smooth metric on $\cK$. For each essential stratum $E_J$, Lemma~\ref{lem:res} shows that $\Psi_J:=\Psi|_{E_J}$ induces a smooth metric on $K_{E_J}+B_J$, and hence a smooth metric on $K_{\rE_J}$, using the canonical trivialization of $\cO(B_J)$ on $\rE_J$. 

\begin{defi} The \emph{residual measure} of an essential stratum $E_J$ is defined as the volume form $e^{2\Psi_J}$ on $\rE_J$ associated to the smooth metric $\Psi_J$ on $K_{\rE_J}$. 
\end{defi}
Since $\Psi_J$ extends to a smooth metric on $K_{E_J}+B_J$, and $B_J$ has snc support, the residual measure $e^{2\Psi_J}$ has finite total mass iff $B_J$ has coefficients $<1$, \ie iff $E_J$ is a minimal essential stratum (cf.~Lemma~\ref{lem:res}). 

\medskip

We next analyze the local behaviour of the measures $\sigma_t$ and their residual measures near a given point of $x\in\cX_0$. Denote by $J_x\subset I$ the set of $E_i$'s containing $x$, so that $E_x:=E_{J_x}$ is the unique (non-necessarily essential) stratum containing $x$ in its relative interior. Pick an adapted chart 
$$
\cU\simeq\DD^{n+1}=\DD^{J_x}\times\DD^{J_x^c}
$$
at $x$, with local log map 
$$
\Log_t\colon\cU_t\to\D_x:=\D_{J_x}=\{w\in\R_{\ge 0}^{J_x}\mid b\cdot w=1\},
$$
\ie $\Log_t(z)=(\frac{\log|z_i|}{\log|t|})$ (see Definition~\ref{defi:adapted}). The meromorphic form
\begin{equation}\label{equ:Om}
\Om:=\bigwedge_{i\in {J_x}} d\log z_i\wedge\bigwedge_{j\in {J_x^c}} dz_j
\end{equation}
on $\cU$ induces a holomorphic family of holomorphic volume forms $\Om_t$ on the fibers $\cU_t$, such that $\Om_t\wedge\frac{dt}{t}=\Om$. Denoting as above by
$$
|\Om_t|^2=i^{n^2}\Om_t\wedge\overline{\Om_t}
$$
the associated smooth positive volume form on $\cU_t$, we then have:
\begin{lem}\label{lem:localnu} There exists $\rho\in C^\infty(\cU,\R_{>0})$ such that
\begin{equation}\label{equ:nuloc}
\sigma_t=\rho\prod_{i\in {J_x}}|z_i|^{2a_i}|\Om_t|^2\quad\text{on}\quad\cU_t, 
\end{equation}
and 
\begin{equation}\label{equ:resloc}
e^{2\p_{J}}=\rho\prod_{i\in J_x\setminus J}|z_i|^{2a_i}|d\log z_i|^2\prod_{j\in J_x^c} |dz_j|^2\quad\text{on}\quad\cU\cap E_{J}=\{0\}\times\DD^{(J_x\setminus J)\cup J_x^c}
\end{equation}
for any essential stratum $E_{J}$ containing $E_x$, \ie $J\subset J_x$. 
\end{lem}
\begin{proof} On the one hand, \eqref{equ:volform} yields $\sigma_t=|\Om_t|_{\p_t}^{-2} |\Om_t|^2$ with $|\Om_t|_{\p_t}\in C^\infty(\cU_t,\R_{>0})$ the pointwise length of $\Om_t$ in the metric $\p_t\in C^\infty(K_{X_t})$. On the other hand, $\Om$ defines a trivialization of $K_\cX^{\log}$ on $\cU$, which induces a trivialization 
$$
\tau:=\prod_{i\in {J_x}} z_i^{a_i}\Om\otimes (dt/t)^{-1}
$$
of $\cK=K_{\cX/\DD}^{\log}-\sum_i a_i E_i$ (see~\eqref{equ:ai}). Strictly speaking, this only makes sense after passing to some power of $\tau$, since we allow $a_i\in\Q$, but $|\tau|_{\p}\in C^\infty(\cU,\R_{>0})$ is well-defined, and satisfies $|\tau|_\p=\prod_{i\in {J_x}} |z_i|^{a_i}|\Om_t|_{\p}$ on $\cU_t$. This proves~\eqref{equ:nuloc} with $\rho:=|\tau|_\p^{-2}$, and~\eqref{equ:resloc} similarly follows by realizing the adjunction formula $K_{E_{J}}=(K_\cX+\sum_{i\in J} E_i)|_{E_J}$ in terms of Poincar\'e residues. 
\end{proof}

The stratum $E_x=E_{J_x}$ is contained in an essential stratum iff 
$$
J_x^\ess:=\{i\in J_x\mid a_i=0\}
$$
is nonempty, in which case $E_x^\ess:=E_{J_x^\ess}$ is the smallest essential stratum containing $x$. Assuming this, denote by $e^{2\p_x}:=e^{2\p_{J_x}}$ the residual measure of $E_x^\ess$, 
$$
\D_x^\ess:=\D_{J_x^\ess}\subset\D_x=\D_{J_x}
$$
its dual face, and by $d_x:=\dim\D_x^\ess=|J_x^\ess|-1$ and $\sigma_x^\ess$ the dimension and normalized (to mass $1$) Lebesgue measure of $\D_x^\ess$, respectively. We may now state the following local version of Theorem~\ref{thm:konsoib}: 

\begin{thm}\label{thm:Laplace} If $J_x^\ess$ is empty, then $\sigma_t\to 0$ weakly in $\cU$. If not, then
\begin{equation}\label{equ:loccv}
\e_t^{d_x}\sigma_t\to c_x e^{2\p_{x}}
\end{equation}
weakly in $\cU$, where $c_x=c_{J_x}\in\R_{>0}$ only depends on $J_x$. For any $\chi\in C^0_c(\cU)$, we further have 
\begin{equation}\label{equ:loclogcv}
\e_t^{d_x}(\Log_t)_\star(\chi \sigma_t)\to c_x\left(\int_{E_x^\ess}\chi e^{2\p_x}\right)\sigma_x^\ess
\end{equation}
weakly on $\D_x$. 
\end{thm}

\begin{proof} Assume first $J_x':=J_x^\ess\ne\emptyset$. In what follows, $c_x$ denotes a positive constant only depending on $J_x$, which is allowed to vary from line to line. It is enough to establish~\eqref{equ:loclogcv}, which implies~\eqref{equ:loccv} by taking the total mass over $\D_x$. Setting $J_x'':=J_x\setminus J_x'$ and $f:=\chi \rho\in C^0_c(\cU)$ with $\rho$ as in Lemma~\ref{lem:localnu}, \eqref{equ:loclogcv} amounts to
$$
\e_t^{d_x}\int_{\cU_t}(\f\circ\Log_t)\chi \sigma_t\to c_x\left(\int_{\{0\}\times\DD^{J_x''\cup J_x^c}} f\prod_{i\in J_x''} |z_i|^{2a_i}|d\log z_i|^2\prod_{j\in J_x^c} |dz_j|^2\right)\left(\int_{\D_x^\ess}\f\,\sigma_x^\ess\right). 
$$
for any $\f\in \Cz(\D_x)$,  by~\eqref{equ:resloc}. Since $\Log_t\colon\cU_t\to\D_x$ factors through $\cU_t\to\DD^{J_x}$, arguing on each fiber $\DD^{J_x}\times\{y\}\subset\cU$ with $y\in\DD^{J_x^c}$, we may and do assume for notational simplicity $J_x^c=\emptyset$, \ie $|J_x|=n+1$ and $E_x=\{x\}$. We are thus in the setup of the model toric case of~\S\ref{sec:toric}, and~\eqref{equ:Omt} yields $|\Om_t|^2=c_x\e_t^{-n}\Log_t^\star\sigma_x$ with $\sigma_x$ the Lebesgue measure of $\D_x$ normalized to mass $1$. Combining this with~\eqref{equ:nuloc}, we infer 
$$
\sigma_t=c_x \e_t^{-n} g\Log_t^\star\left(e^{-2\e_t^{-1} a\cdot w} \sigma_x\right), 
$$
and hence
$$
\e_t^{d_x}\int_{\cU_t}(\f\circ\Log_t)\chi\sigma_t=c\e_t^{d_x-n}\int_{\D_x}\f(w) e^{-2\e_t^{-1} a\cdot w} \sigma_x(dw)\int_{\Log_t^{-1}(w)} f\rho_{t,w}, 
$$
where $\rho_{t,w}$ denotes the Haar measure of the fiber $\Log_t^{-1}(w)$.  Fix $t\in\DD^\times$. Writing $w=(w',w'')\in\R^{J_x}=\R^{J_x'}\times\R^{J_x''}$, the change of variables
$$
w'=(1-\e_t b''\cdot v'')v',\quad w''=\e_t v'' 
$$
with 
$$
v'\in\D_x^\ess,\quad v''\in\D_{x,t}'':=\R_{\ge 0}^{J_x''}\cap\left\{b''\cdot v''\le\e_t^{-1}\right\}
$$
yields a diffeomorphism $\D_x\simeq\D_x^\ess\times\D_{x,t}''$ with respect to which
$$
\sigma_x(dw)=c_x(1-\e_t b''\cdot v'')^{d_x}\e_t^{n-d_x} \sigma_x^\ess(dv')\otimes dv''. 
$$
with $dv''$ the Lebesgue measure of $\R^{J_x''}$. Since $a_i=0$ for $i\in J_x'$, we get
$$
\e_t^{d_x-n}\int_{\D_x}\f(w) e^{-2\e_t^{-1} a\cdot w} \sigma_x(dw)\int f\rho_{t,w}
$$
$$
=c_x\int_{\D_x^\ess\times\D_{x,t}''}\f\left((1-\e_t b''\cdot v'')v',\e_t v''\right) \sigma_x^\ess(dv') (1-\e_t b''\cdot v'')^{d_x} e^{-2 a''\cdot v''} dv''\int f\rho_{t,v',v''}
$$
where $\rho_{t,v',v''}=\rho_{t,w}$ under the above change of variables. As $t\to 0$ this converges to 
$$
\left(\int_{\D_x^\ess}\f(v',0)\sigma_x^\ess(dv')\right)\left(\int_{\R_{\ge 0}^{J_x''}} e^{-2 a''\cdot v''} dv''\int f\rho_{v''}\right)
$$
where $\rho_{v''}$ denotes the Haar measure of the fiber over $v''$ of $\Log\colon(\C^\times)^{J_x''}\to\R^{J_x''}$. Passing to logarithmic polar coordinates further yields
$$
\int_{\R_{\ge 0}^{J_x''}} e^{-2 a''\cdot v''} dv''\int f\rho_{v''}=c_x\int_{\{0\}\times\DD^{J_x''}} f \prod_{i\in J_x''}|z_i|^{2a_i} |d\log z_i|^2, 
$$
and the result follows. 

Finally, assume $J_x'=\emptyset$, \ie $\kappa_x:=\min_{i\in J_x}\kappa_i>0$. The computation of the first part of the proof then applies to $|t|^{-2\kappa_x}\sigma_t$, and shows that $|t|^{-2\kappa_x}\e_t^{\tilde d_x}\sigma_t$ converges weakly in $\cU$ for some $\tilde d_x\in\{0,\dots,n\}$. Since $|t|^{2\kappa_x}\e_t^{-\tilde d_x}=|t|^{2\kappa_x}(-\log|t|)^{\tilde d_x}\to 0$, we infer $\sigma_t\to 0$ weakly in $\cU$. 
\end{proof}

%
%
\subsection{Hybrid convergence}
We are now in a position to prove the following more precise version of Theorem~\ref{thm:konsoib}. Recall $d=\dim\D_\cK^\ess$ denotes the essential dimension of the family $(\sigma_t)$ at $t=0$. 

\begin{thm}\label{thm:konsoib2} Consider the measures on $\cX_0$ and $\D_\cK^\ess$ respectively defined by 
$$
\mu_{\cX_0}:=\sum_J c_J e^{2\Psi_J},\quad\sigma_\cX:=\sum_J\left(\int_{\rE_J} e^{2\Psi_J}\right)\sigma_J,
$$
where both sums range over all $d$-dimensional faces $\D_J$ of $\D_\cK^\ess$, the constants $c_J>0$ are given by Theorem~\ref{thm:Laplace}, and $\sigma_J$ is the Lebesgue measure of $\D_J$, normalized to mass $1$. Then:
\begin{itemize}
\item[(i)] the rescaled measure $\mu_t=\e_t^d\sigma_t$ converges weakly to $\mu_{\cX_0}$ in $\cX$;
\item[(ii)] $\mu_t\to\mu_0:=(\val_\cX)_\star\sigma_\cX$ weakly in $X^\hyb$. 
\end{itemize}
\end{thm}
In particular, as $t\to 0$ most of the mass of $\sigma_t$ on $X_t$ accumulates towards the essential strata of minimal dimension $n-d$. 

\begin{proof} The first point is a direct consequence of Theorem~\ref{thm:Laplace}, which also yields $(\Log_\cX)_\star\mu_t\to\sigma_\cX$, by Lemma~\ref{lem:PS}. For any higher snc model $\rho\colon\cX'\to\cX$, $(\Log_{\cX'})_\star\mu_t$ similarly converges to $\sigma_{\cX'}$, with support in $\D_{\rho^\star\cK}^\ess=\Log_{\cX'}(\Sk^\ess(\cK))$ (see Lemma~\ref{lem:essinv}). The second point now follows from Lemma~\ref{lem:cvmeas2}. 
\end{proof}

 The integer $d\in\{0,\dots,n\}$ measures the `degree of degeneration' of the family $(\sigma_t)$ at $t=0$. We illustrate this in the two extreme cases: 
 
\begin{exam}\label{exam:mindeg} Consider the \emph{minimally degenerate} case $d=0$. Each essential stratum must then be equal to some $E_i$, and the limit measures in $\cX$ and $X^\hyb$ are respectively of the form
$$
\mu_{\cX_0}=\sum_i \mu_{E_i},\quad \mu_0=\sum_i\mu_i(\rE_i)\d_{v_{E_i}}
$$
where $\mu_{E_i}$ is a smooth positive volume form of finite mass on $\rE_i$. Furthermore, it follows from~\eqref{equ:nuloc} that $\mu_t\to\mu_{E_i}$ smoothly on $\cX$ near each point of $\rE_i$. 
\end{exam}

\begin{exam}\label{exam:maxdeg} Consider now the \emph{maximally degenerate} case $d=n$. Each minimal dimensional essential stratum is then reduced to a point $E_J=\{x_J\}$, corresponding to an $n$-dimensional faces $\D_J$ of $\D_\cK^\ess$, and the limit measures are of the form
$\mu_{\cX_0}=\sum_J m_J\d_{x_J}$, $\mu_0=\sum_J m_J\sigma_J$ with $m_J\in\R_{>0}$. For any adapted chart $\cU\simeq\DD^J$ at $x_J$ with local log map $\Log_t\colon\cU_t\to\D_J$, \eqref{equ:nuloc} and~\eqref{equ:Omt} further yield
\begin{equation}\label{equ:maxdeg}
\mu_t=g_J\Log_t^\star\sigma_J
\end{equation}
for some density $g_J\in C^\infty(\cU,\R_{>0})$, which necessarily satisfies $g_J(x_J)=m_J$. 
\end{exam}

%
%
\section{From complex to non-Archimedean pluripotential theory}\label{sec:pluripot}
In this section we survey a number of fundamental facts from complex pluripotential theory and its non-Archimedean analogue, emphasizing their interactions through the hybrid space. In particular, we prove general convergence result for Monge--Amp\`ere measures of psh metrics, slightly extending a result of Favre~\cite{Fav}. 
%
%
\subsection{A biased review of the complex case}

We assume here that $X$ is a smooth projective complex variety equipped with an ample $\Q$-line bundle $L$, of volume $V:=(L^n)$. Given a continuous Hermitian metric $\phi$ on $L$, we denote\footnote{We emphasize that this is a purely notational device, as this form is of course not $\ddc$-exact.} by $\ddc\phi$ its curvature current, normalized so as to represent the Chern class $c_1(L)$. The metric $\phi$ is \emph{plurisubharmonic} (\emph{psh} for short) if the closed $(1,1)$-current $\ddc\phi$ is semipositive, which amounts to requiring that the local weight $-\log|\tau|_\phi$ of $\phi$ with respect to any local trivialization $\tau$ of $L$ is a psh function. 

The set $\CPSH(L)$ of continuous psh metrics on $L$ is closed in $\Cz(L)$ for the topology of uniform convergence. As a key consequence of the Bouche--Catlin--Tian--Zelditch theorem~\cite{Bou90,Cat,Tia,Zel}, we have:

\begin{thm}\label{thm:CPSH} The set $\CPSH(L)$ coincides with the closure in $\Cz(L)$ of the set $\FS(L)$ of Fubini--Study metrics. \end{thm}

The \emph{Monge--Amp\`ere measure} of a smooth psh metric $\phi$ is defined as the smooth probability measure 
$$
\MA(\phi):=V^{-1}(\ddc\phi)^n.
$$
The Monge--Amp\`ere operator $\phi\mapsto\MA(\phi)$ satisfies the following basic Lipschitz estimate, known as the \emph{Chern--Levine--Nirenberg inequality}. 
\begin{lem}\label{lem:CLN} Assume $\phi_1,\dots\phi_4$ are smooth psh metrics on $L$. Then 
\begin{equation}\label{equ:CLN}
\left|\int_X(\phi_1-\phi_2)\left(\MA(\phi_3)-\MA(\phi_4)\right)\right|\le 2n\sup_X|\phi_3-\phi_4|.
\end{equation}
\end{lem}
\begin{proof} The identity $(\ddc\phi_3)^n-(\ddc\phi_4)^n=\ddc(\phi_3-\phi_4)\wedge\sum_{j=0}^{n-1}(\ddc\phi_3)^j\wedge(\ddc\phi_4)^{n-j-1}$ yields 
\begin{align*}
\int(\phi_1-\phi_2)\left(\MA(\phi_3)-\MA(\phi_4)\right) & =V^{-1}\sum_{j=0}^{n-1}\int (\phi_1-\phi_2)\ddc(\phi_3-\phi_4)\wedge\sum_{j=0}^{n-1}(\ddc\phi_3)^j\wedge(\ddc\phi_4)^{n-j-1}\\
& =V^{-1}\sum_{j=0}^{n-1}\int (\phi_3-\phi_4)\ddc(\phi_1-\phi_2)\wedge(\ddc\phi_3)^j\wedge(\ddc\phi_4)^{n-j-1}\\
& =\sum_{j=0}^{n-1}V^{-1}\int(\phi_3-\phi_4)\ddc\phi_1\wedge(\ddc\phi_3)^j\wedge(\ddc\phi_4)^{n-j-1}\\
& -\sum_{j=0}^{n-1}V^{-1}\int(\phi_3-\phi_4)\ddc\phi_2\wedge(\ddc\phi_3)^j\wedge(\ddc\phi_4)^{n-j-1},
\end{align*}
by integration-by-parts. For each $j$, $V^{-1}\ddc\phi_1\wedge(\ddc\phi_3)^j\wedge(\ddc\phi_4)^{n-j-1}$ is a probability measure, and hence 
$$
\left|V^{-1}\int(\phi_3-\phi_4)\ddc\phi_1\wedge(\ddc\phi_3)^j\wedge(\ddc\phi_4)^{n-j-1}\right|\le\sup|\phi_3-\phi_4|.
$$
The same holds for $V^{-1}\ddc\phi_2\wedge(\ddc\phi_3)^j\wedge(\ddc\phi_4)^{n-j-1}$, and the result follows. 
\end{proof}
As a key consequence, we get: 

\begin{prop} The Monge--Amp\`ere operator $\phi\mapsto\MA(\phi)$ admits a unique continuous extension to $\CPSH(L)$.
\end{prop}

\begin{proof} Since $\MA(\phi)$ is a probability measure, it suffices to show that $\phi\mapsto\int f\,\MA(\phi)$ admits a continuous extension to $\CPSH(L)$ for any smooth function $f\in C^\infty(X)$. A small multiple of $f$ can be written as a difference of smooth psh metrics on $L$, and the result now follows from~\eqref{equ:CLN}, sincethe set of smooth psh metrics on $L$ is dense in $\CPSH(L)$, as a consequence of Theorem~\ref{thm:CPSH}. 
\end{proof}

For any two $\phi,\p\in\CPSH(L)$, we next introduce the \emph{Aubin energy}
$$
\ii(\phi,\p):=\int(\phi-\p)\left(\MA(\p)-\MA(\phi)\right). 
$$
This functional is symmetric, invariant under translation of $\phi,\p$ by a constant, and it satisfies 
\begin{equation}\label{equ:I}
0\le\ii(\phi,\p)\le 2\sup|\phi-\p|.
\end{equation}
The right-hand inequality is indeed trivial. By density, it is enough to show nonnegativity when $\phi,\p$ are smooth, in which case Stokes yields 
$$
\ii(\phi,\p)=V^{-1}\sum_{j=0}^{n-1}\int d(\phi-\p)\wedge d^c(\phi-\p)\wedge(\ddc\phi)^j\wedge(\ddc\p)^{n-1-j}\ge 0. 
$$
Less trivially, the energy functional $\ii$ satisfies a quasi-triangle inequality
$$
\ii(\phi_1,\phi_2)\le C_n\left(\ii(\phi_1,\phi_3)+\ii(\phi_2,\phi_3)\right)
$$
for all $\phi_1,\phi_2,\phi_3\in\CPSH(L)$, and the Chern--Levine--Nirenberg inequality can be refined into the following H\"older estimate with respect to the energy
\begin{equation}\label{equ:Hodge} 
\left|\int(\phi_1-\phi_2)\left(\MA(\phi_3)-\MA(\phi_4)\right)\right|\le C_n\ii(\phi_1,\phi_2)^{\a_n}\ii(\phi_3,\phi_4)^{1/2}\max_i\ii(\phi_i,\phi_1)^{1/2-\a_n}
\end{equation}
for all $\phi_1,\dots,\phi_4\in\CPSH(L)$, with $\a_n:=2^{-n}$ (see~\cite{BEGZ} and~\cite{synthetic}). In particular, this implies the following uniqueness result:

\begin{cor}\label{cor:unique} For all $\phi,\p\in\CPSH(L)$, the following are equivalent:
\begin{itemize}
\item[(i)] $\MA(\phi)=\MA(\p)$;
\item[(ii)] $\ii(\phi,\p)=0$;
\item[(iii)] $\phi-\p$ is constant. 
\end{itemize}
\end{cor}
\begin{proof} (iii)$\Rightarrow$(i)$\Rightarrow$(ii) is clear. Assume (ii), and normalize $\phi,\p$ so that $\int(\phi-\p)\MA(\phi)=0$. Since $\ii(\phi,\p)=0$, \eqref{equ:Hodge} shows that $\int(\phi-\p)\,\MA(\rho)=0$ for all $\rho\in\CPSH(L)$. By Yau's theorem, this implies $\int(\phi-\p)\,\mu=0$ for all volume forms $\mu$, and it follows that $\phi-\p$ is constant. 
\end{proof}
This yields in turn the following \emph{domination principle}. 

\begin{cor}\label{cor:dom} Assume $\phi,\p\in\CPSH(L)$ satisfy $\p\le\phi$ $\MA(\phi)$-a.e. Then $\p\le\phi$ everywhere. 
\end{cor}
\begin{proof} After replacing $\p$ with $\max\{\phi,\p\}$, we may assume $\p\ge\phi$, with equality $\MA(\phi)$-a.e. We then need to show $\p=\phi$. Since 
$$
0\le\ii(\phi,\p)=\int(\phi-\p)(\MA(\p)-\MA(\phi))=\int(\phi-\p)\MA(\p)\le 0,
$$
Corollary~\ref{cor:unique} shows that $\phi-\p$ is constant, and hence $0$ since $\phi-\p=0$ $\MA(\phi)$-a.e. 
\end{proof}
%
%
\subsection{The non-Archimedean case}
In this section, we return to the setting of a projective meromorphic degeneration $X\to\DD^\times$ associated to a smooth projective variety $X_K$ over the field $K=\C\{t\}$ of convergent Laurent series, with Berkovich analytification $X_0:=X_K^\an$. 

Consider first a (relatively) ample meromorphic line bundle $L$ on $X$, associated to an ample line bundle $L_K$ on $X_K$. In the non-Archimedean context, we use the analogue of Theorem~\ref{thm:CPSH} as a definition.  

\begin{defi}\label{defi:CPSH}The space $\CPSH(L_0)$ of \emph{continuous psh metrics} on $L_0$ is defined as the closure of the set $\FS(L_0)$ of Fubini--Study metrics on $L$ with respect to uniform convergence. 
\end{defi}
Recall that a metric $\phi$ on $L_0$ lies in $\FS(L_0)$ iff $\phi=\phi_\cL$ is a model metric associated to a semiample model $\cL$ of $L$ (see Proposition~\ref{prop:FSmodel}). For more general model metrics, the psh condition can be characterized as follows (see~\cite[Corollary~7.9]{BE}). 

\begin{lem}\label{lem:modelpsh} Pick a model metric $\phi$ on $L_0$, and write $\phi=\phi_\cL$ for a model $(\cX,\cL)$ of $(X,L)$. Then $\phi$ is psh iff $\cL$ is (relatively) nef, \ie $\cL\cdot C\ge 0$ for every curve $C\subset\cX_0$. 
\end{lem}

Next pick a tuple of (non-necessarily ample) line bundles $L_1,\dots,L_n$ on $X$, and model metrics $\phi_1,\dots,\phi_n$ on $L_{1,0},\dots,L_{n,0}$. Then $\phi_i=\phi_{\cL_i}$ for models $\cL_1,\dots,\cL_n$ of $L_1,\dots,L_n$, which can be chosen to be determined on the same snc model $\cX$ of $X$. Writing as usual $\cX_0=\sum_i b_i E_i$, we define a signed atomic measure on $X_0$ by setting
\begin{equation}\label{equ:mixedMA}
\ddc\phi_1\winter\ddc\phi_n:=\sum_i b_i (\cL_1\inter\cL_n\cdot E_i)\,\d_{v_{E_i}}
\end{equation}
with $v_{E_i}\in X_0$ the divisorial valuation associated to $E_i$, see~\eqref{equ:vE}. Using the projection formula, this definition is easily seen to be invariant under pulling back the $\cL_i$'s to a higher model of $X$, and hence only depends on the model metrics $\phi_i$. 

One further has
\begin{equation}\label{equ:totalmass}
\int_{X_0}\ddc\phi_1\winter\ddc\phi_n=(L_1\inter L_n), 
\end{equation}
and the integration-by-parts formula
\begin{equation}\label{equ:NAIPP}
\int f\,\ddc g\wedge\ddc\phi_2\winter\ddc\phi_n=\int g\,\ddc f\wedge\ddc\phi_2\winter\ddc\phi_n.
\end{equation}
holds for all model functions $f,g\in \Cz(X_0)$ and model metrics $\phi_i\in\Cz(L_i)$. 

Finally, when all $L_i$ are ample Lemma~\ref{lem:modelpsh} yields
$$
\phi_1,\dots,\phi_n\text{ psh }\Longrightarrow\ddc\phi_1\winter\ddc\phi_n\ge 0.
$$
From the perspective of the present notes, \eqref{equ:mixedMA} is best justified as follows. 

\begin{thm}\label{thm:ddchyb} Assume given a hybrid model metric $(\phi_{i,t})_{t\in\DD}$ on $L_i$, $i=1,\dots,n$. Then 
$$
\ddc\phi_{1,t}\winter\ddc\phi_{n,t}\to\ddc\phi_{1,0}\winter\ddc\phi_{n,0}
$$
weakly in $X^\hyb$ as $t\to 0$. 
\end{thm}
\begin{proof} By definition, each $(\phi_{j,t})_{t\in\DD^\times}$ is the restriction of a smooth metric $\Phi_j$ on a model $\cL_j$ of $L_j$, and $\phi_{j,0}=\phi_{\cL_j}$ (see~\S\ref{sec:hybmodmetr}). After passing to a higher model, we may assume all $\cL_j$ are determined on the same snc model $\cX$, with central fiber $\cX_0=\sum_i b_i E_i$. After perhaps shrinking the base $\DD$ slightly, each $\Phi_j$ can further be written as a difference of smooth strictly psh metrics on ample line bundles; by multilinearity, we may thus assume that each $\cL_j$ is ample and $\Phi_j$ is (strictly) psh. Denoting by $[X_t]$ the integration current on $X_t$, the positive measure 
$$
\mu_t:=\ddc\phi_{1,t}\winter\ddc\phi_{n,t}=\ddc\Phi_1\winter\ddc\Phi_n\wedge[X_t]
$$
then converges weakly in $\cX$ to 
$$
\mu_{\cX_0}:=\ddc\Phi_1\winter\ddc\Phi_n\wedge[\cX_0]=\sum_i b_i\,\ddc\Phi_1\winter\ddc\Phi_n\wedge[E_i],
$$
which puts no mass on any nowhere dense Zariski closed subset of $\cX_0$, and has mass
$$
\mu_{\cX_0}(E_i)=b_i(\cL_1\inter\cL_n\cdot E_i)
$$
on each component $E_i$ of $\cX_0$. By Proposition~\ref{prop:cxvshyb}, we conclude, as desired, that $\mu_t$ converges weakly in $X^\hyb$ to $\sum_i\mu_{\cX_0}(E_i)\d_{v_{E_i}}=\ddc\phi_{1,0}\winter\ddc\phi_{n,0}$. 
\end{proof}

Assume as above $L$ is an ample line bundle on $X$, of volume $V=(L^n)$. By~\eqref{equ:totalmass}, for each psh model metric $\phi$ on $L_0$, the \emph{non-Archimedean Monge--Amp\`ere measure}
$$
\MA(\phi):=V^{-1}(\ddc\phi)^n
$$
is an atomic probability measure on $X_0$. Using~\eqref{equ:NAIPP}, the proof of Lemma~\ref{lem:CLN} goes through without change to show that the Chern--Levine--Nirenberg inequality~\eqref{equ:CLN} holds for psh model metrics $\phi_1,\dots,\phi_4$ on $L_0$. By Proposition~\ref{prop:modelfunc}, this yields the next result, originally proved in~\cite{CL}. 

\begin{prop} The non-Archimedean Monge--Amp\`ere operator $\phi\mapsto\MA(\phi)$ admits a unique continuous extension to $\CPSH(L_0)$.
\end{prop}

The following non-Archimedean analogue of Yau's theorem is the main result\footnote{For families over an algebraic curve, the general case being proved in~\cite{BGJKM,BGM}.} of~\cite{nama}. 
\begin{thm}\label{thm:nama} For any probability measure $\mu$ of $X_0$ supported on the skeleton $\Sk(\cX)$ of some snc model $\cX$, there exists $\phi\in\CPSH(L_0)$ such that $\MA(\phi)=\mu$.
\end{thm}

Using this in place of Yau's theorem, Corollary~\ref{cor:unique} remains valid in this context. In particular, any $\phi\in\CPSH(L_0)$ is uniquely determined by $\MA(\phi)$ up to an additive constant. 

%
\subsection{Hybrid continuity of the Monge--Amp\`ere operator}
As above, we fix an ample line bundle $L$ on $X$. Using Theorem~\ref{thm:ddchyb}, we now establish:

\begin{thm}\label{thm:cvMA} Pick a hybrid continuous family of metrics $(\phi_t)_{t\in\DD}$ on the fibers $L_t$, and assume $\phi_t\in\CPSH(L_t)$ for all $t$. Then $\MA(\phi_t)\to\MA(\phi_0)$ weakly in $X^\hyb$ as $t\to 0$. 
\end{thm}
This slightly extends~\cite[Theorem~4.2]{Fav}, which implicitly considers families that are psh with respect to $t$ small enough (compare~\cite[Remark~2.9]{Li23}). 

\begin{proof} By Lemma~\ref{lem:cvmeas}, it suffices to show $\int_{X_t}\phi_D^\hyb\,\MA(\phi_t)\to\int_{X_0}\phi_D^\hyb\,\MA(\phi_0)$ for every hybrid model function $\phi_D^\hyb$, where $D\in\VCar(\cX)$ for some model $\cX$. By Proposition~\ref{prop:DFS}, after replacing $D$ by a positive multiple we can write the model function $\phi_D=\phi_D^\hyb\big|_{X_0}$ as $\phi_D=\phi_{1,0}-\phi_{2,0}$ with $\phi_{i,0}\in\FS(L_0)$, $i=1,2$. Each $\phi_{i,0}$ extends to a hybrid Fubini--Study metric $(\phi_{i,t})_{t\in\DD}$ (see Example~\ref{exam:hybFS}). After perhaps passing to a higher model, it follows that $(\phi_{1,t}-\phi_{2,t})$ is a hybrid model metric induced by a smooth metric on $\cO_\cX(D)$, and hence 
$$
\phi_D^\hyb\big|_{X_t}=\e_t(\phi_{1,t}-\phi_{2,t})+O(\e_t). 
$$
see Example~\ref{exam:conf}. We are thus reduced to showing 
$$
\e_t\int_{X_t}(\phi_{1,t}-\phi_{2,t})\MA(\phi_t)\to\int_{X_0}(\phi_{1,0}-\phi_{2,0})\MA(\phi_0).
$$
Pick $\d>0$, and $\p_0\in\FS(L_0)$ such that $\sup_{X_0}|\phi_0-\p_0|<\d$ (see Definition~\ref{defi:CPSH}). By the non-Archimedean  version of the Chern--Levine--Nirenberg inequality, we have
$$
\left|\int (\phi_{1,0}-\phi_{2,0})\left(\MA(\p_0)-\MA(\phi_0)\right)\right|\le 2n\d.
$$
Example~\ref{exam:FShyb} yields again an extension of $\p_0$ to a hybrid Fubini--Study metric $(\p_t)_{t\in\DD}$. Since $(\p_t)$ is in particular a hybrid model metric, Theorem~\ref{thm:ddchyb} implies $\MA(\p_t)\to\MA(\p_0)$ weakly in $X^\hyb$, and hence 
$$
\left|\e_t\int_{X_t}(\phi_{1,t}-\phi_{2,t})\MA(\p_t)-\int_{X_0}(\phi_{1,0}-\phi_{2,0})\MA(\p_0)\right|\le\d
$$
for $t$ small enough. Here we have used that $(\phi_t-\p_t)$ is a hybrid continuous metric on the trivial line bundle (see Example~\ref{exam:conf}), which also implies 
$$
\e_t\sup_{X_t}|\phi_t-\p_t|\to\sup_{X_0}|\phi_0-\p_0|. 
$$
For $t$ small enough we thus have $\e_t\sup_{X_t}|\phi_t-\p_t|<\d$, and the complex Chern--Levine--Nirenberg inequality yields
$$
\e_t\left|\int(\phi_{1,t}-\phi_{2,t})\left(\MA(\phi_t)-\MA(\p_t)\right)\right|\le 2n\d.
$$
Combining these estimates, we get
$$
\left|\int \e_t(\phi_{1,t}-\phi_{2,t})\MA(\phi_t)-\int(\phi_{1,0}-\phi_{2,0})\,\MA(\phi_0)\right|\le (4n+1)\d
$$
for $t$ small enough, and the result follows. 
\end{proof}

%
\subsection{Non-Archimedean vs.~real Monge--Amp\`ere operator}\label{sec:NAMA}

Due to the intersection theoretic definition of the non-Archimedean Monge--Amp\`ere operator that we have adopted, the latter appears to be more elusive than its complex counterpart. We next formulate a result due to Vilsmeier~\cite{Vil}, which relates it to the \emph{real} Monge--Amp\`ere operator under a retraction invariance assumption. 

In what follows, we pick a model $\cL$ of $L$ determined on an snc model $\cX$ of $X$. Recall that the dual complex $\D_\cX\hto X_0$ admits a canonical retraction $p_\cX\colon X_0\to\D_\cX$, provided by the non-Archimedean log map (see~\S\ref{sec:snc}). 

Using the model metric $\phi_\cL$ as a reference metric allows to identify any other continuous metric $\phi$ on $L_0$ with the continuous function $\f:=\phi-\phi_\cL$. 

\begin{defi}\label{defi:Lpsh} We say that $\f\in \Cz(X_0)$ is \emph{$\cL$-psh} if the corresponding metric $\phi_\cL+\f$ is psh. 
\end{defi}
The set $\CPSH(\cL)\subset\Cz(X_0)$ of continuous $\cL$-psh functions is in 1--1 correspondence with $\CPSH(L_0)$. We define the Monge--Amp\`ere measure of $\f\in\CPSH(\cL)$ as that of the corresponding metric, \ie 
$$
\MA(\f):=\MA(\phi_\cL+\f). 
$$
As a first basic fact, we have (see~\cite[Proposition~3.12]{siminag}): 

\begin{prop}\label{prop:pshconvex} The restriction of any $\f\in\CPSH(\cL)$ to each face $\D_J$ of $\D_\cX\hto X_0$ is convex. 
\end{prop}

%
%
We may now state: 

\begin{thm}\label{thm:Vil}\cite[Theorem~1.2]{Vil} Pick $\f\in\CPSH(\cL)$, an $n$-dimensional face $\D_J$ of $\D_\cX$, and assume the retraction invariance property $\f=\f\circ p_\cX$ over $p_\cX^{-1}(\rsig_J)$. Then 
\begin{equation}\label{equ:Vil}
\MA(\f)=V^{-1}\MA_\R(\f|_{\rsig_J})\text{ on } p_\cX^{-1}(\rsig_J).
\end{equation} 
\end{thm}
As before, $V=(L_K^n)$ is the fiberwise volume of $L$, and $\MA_\R(\f|_{\rsig_J})$ is the real Monge--Amp\`ere measure of the convex function $\f|_{\rsig_J}$, viewed as a measure on $p_\cX^{-1}(\rsig_J)$ supported on $\rsig_J$. 

\begin{rmk} This is in fact a local result: for any convex function $f\colon\rsig_J\to\R$, $f\circ p_\cX$ is psh on the open set $p_\cX^{-1}(\rsig_J)$, and~\cite[Theorem~1.2]{Vil} computes its non-Archimedean Monge--Amp\`ere measure as
$$
\left(\ddc f\circ p_\cX\right)^n=\MA_\R(f),
$$
a non-Archimedean analogue of~\eqref{equ:MAlog} where $\MA_\R(f)$ is viewed as above as a  measure on $p_\cX^{-1}(\rsig_J)$ supported on $\rsig_J$. In the approach of Chambert-Loir and Ducros~\cite{CLD}, this identity holds (almost) by definition. 
\end{rmk}


%
%
\subsection{Maximal hybrid extension}\label{sec:maxext} 
In this final section we show that every non-Archimedean metric $\phi_0\in\CPSH(L_0)$ can be extended to a hybrid continuous family of metrics $\phi_t\in\CPSH(L_t)$ for $t$ in a slightly shrunken disc. This follows from the following more precise result, a `continuous version' of~\cite[Theorem~4.3.3]{Reb} (which reduces to~\cite[Theorem~6.6]{YTD} when $X=Z\times\DD^\times$). 

\begin{thm}\label{thm:maxext} Fix a closed disc $\bar\DD_r\subset\DD$ and a continuous family of reference metrics $\phi_{\re,t}\in\CPSH(L_t)$ for $|t|=r$. Then every non-Archimedean metric $\phi\in\CPSH(L_0)$ admits a (unique) largest hybrid continuous extension $(\phi_t)_{t\in\bar\DD_r}$ such that 
\begin{itemize}
\item[(i)] the family $(\phi_t)_{0<|t|<r}$ is psh; 
\item[(ii)] $\phi_t=\phi_{\re,t}$ for $|t|=r$. 
\end{itemize}
\end{thm}
Here we say that a family of metrics $(\phi_t)_{t\in U}$ with $U\subset\DD^\times$ open is \emph{psh} if the corresponding metric on $L|_{\pi^{-1}(U)}$ is psh. 

\begin{proof}[Proof of Theorem~\ref{thm:maxext} (sketch)] We claim that the desired extension $(\phi_t)_{t\in\bar\DD_r}$ coincides with the pointwise supremum $\left(P_t(\phi_0)\right)_{t\in\bar\DD_r}$ of the family of all hybrid continuous metrics $(\p_t)_{t\in\bar\DD_r}$ such that 
\begin{itemize}
\item $(\p_t)_{0<|t|<r}$ is psh; 
\item $\p_t\le\phi_{\re,t}$ for $|t|=r$;
\item $\p_0\le\phi_0$.
\end{itemize}
When $\phi_0$ is a model metric determined by an ample model $\cL$ of $L$, this can be deduced from~\cite[Theorem~1.4.4]{Reb}. In the general case, note that for any $\phi'_0\in\CPSH(L_0)$ and $c\in\R$ we have 
\begin{itemize}
\item[(a)] $\phi_0\le\phi'_0\Longrightarrow P_t(\phi_0)\le P_t(\phi'_0)$; 
\item[(b)] $P_t(\phi_0+c)=P_t(\phi_0)+c\log\frac{r}{|t|}$. 
\end{itemize}
As an easy consequence of Definition~\ref{defi:CPSH}, we can find a sequence of model metrics $\phi_{0,j}$, each determined by an ample model of $L$, such that $\d_j:=\sup_{X_0}|\phi_0-\phi_{0,j}|$ tends to $0$. By the first step, $(P_t(\phi_{0,j}))_{0<|t|<r}$ is psh, $P_t(\phi_{0,j})=\phi_{\re,t}$ for $|t|=r$ and $P_0(\phi_{0,j})=\phi_{0,j}$. Since $\phi_0-\d_j\le \phi_{0,j}\le \phi_0+\d_j$, (a) and (b) imply
$$
\left|P_t(\phi_{0,j})-P_t(\phi_0)\right|\le\d_j\log\frac{r}{|t|}, 
$$
and the result easily follows. 
\end{proof}

%
%
\section{Hybrid stability for Monge--Amp\`ere equations}\label{sec:hybstab}
%
Closely following~\cite{LiSYZ}, we establish in this section a general hybrid continuity result for solutions to Monge--Amp\`ere equations, under a partial retraction invariance assumption. As in Yang Li's work, this relies on a crucial `asymmetric version' of Ko\l{}odziej's stability theorem for complex Monge--Amp\`ere equations, which is presented in the Appendix. 

In what follows, $X\to\DD^\times$ is a projective meromorphic degeneration relatively ample meromorphic line bundle $L$ on $X$, associated to a smooth projective variety $X_K$ over $K=\C\{t\}$ with an ample line bundle $L_K$, of volume $V=(L_K)^n=(L_t^n)$. 
%
\subsection{General setup}\label{sec:genset}
As in~\S\ref{sec:volan}, we consider the family of volume forms $\sigma_t=e^{2\p_t}$ on the fibers $X_t$ associated to a hybrid model metric $(\p_t)$ on $K_{X/\DD^\times}$, \ie the restriction of a smooth metric $\Psi$ on some model $\cK$ of $K_{X/\DD^\times}$, determined on an snc model $\cX$ of $X$. Recall that this includes the case of Calabi--Yau degenerations.  

As a consequence of Theorem~\ref{thm:konsoib2}, the associated probability measures 
$$
\mu_t:=\frac{\sigma_t}{\sigma_t(X_t)}
$$
converge weakly in $X^\hyb$ to a Lebesgue-type probability measure $\mu_0$ with support the top-dimensional part of the essential complex 
$$
\D_\cK^\ess\subset\D_\cX\hto X_0,
$$
whose dimension 
$$
d=\dim\D_\cK^\ess
$$
measures the degree of degeneration of the family of volume forms $(\sigma_t)$. For each $t\in\DD$, Yau's theorem and its non-Archimedean version (Theorem~\ref{thm:nama}) yield a continuous psh metric $\phi_t\in\CPSH(L_t)$, unique up to normalization, such that 
$$
\MA(\phi_t)=\mu_t.
$$
Here $\phi_t$ is smooth and strictly psh for $t\ne 0$, and it depends smoothly on $t\ne 0$ modulo normalization. It is then natural to conjecture the following hybrid stability property. 

\begin{conj}\label{conj:hybstab} For an appropriate choice of normalization, the family $\phi=(\phi_t)_{t\in\DD}$ is hybrid continuous. 
\end{conj}
In~\S\ref{sec:maxdegcase}, slightly extending~\cite{LiSYZ} we are going to establish a weaker version of this conjecture in the maximally degenerate case $d=n$, under a partial retraction invariance assumption. Note that when $0<d<n$ the conjecture is fully established in~\cite{Li23} for a certain class of Calabi--Yau degenerations\footnote{While these lectures notes were nearing completion, a general solution to the conjecture for Calabi-Yau degenerations was announced by Y.~Li, see~\cite{Li25}.} 

Returning to the general setup, we first choose (after slightly shrinking $\DD$) a hybrid continuous family of comparison metrics $\tphi_t\in\CPSH(L_t)$ such that $\tphi_0=\phi_0$ (see Theorem~\ref{thm:maxext}). By Lemma~\ref{lem:hybcrit}, the conjecture is then equivalent to
\begin{equation}\label{equ:conj}
\e_t\sup_{X_t}|\phi_t-\tphi_t|\to 0
\end{equation}
for an appropriate choice of normalization. To fix the normalization, pick a reference hybrid model metric $(\phi_{\re,t})$ on $L$, determined by a smooth strictly psh metric on an ample model $\cL$ of $L$, which can be assumed to be determined on the given snc model $\cX$ after perhaps passing to a higher model. For each $t\in\DD$ we then normalize $\phi_t$ by 
$$
\int(\phi_t-\phi_{\re,t})\MA(\phi_{\re,t})=0. 
$$
Since $\MA(\phi_{\re,t})\to\MA(\phi_{\re,0})$ weakly in $X^\hyb$ (see Theorem~\ref{thm:cvMA}), the comparison metrics $\tphi_t$ satisfy
$$
\e_t\int(\tphi_t-\phi_{\re,t})\MA(\phi_{\re,t})\to\int(\tphi_0-\phi_{\re,0})\MA(\phi_{\re,0})=\int(\phi_0-\phi_{\re,0})\MA(\phi_{\re,0})=0. 
$$
and can thus be normalized without loss so as to also satisfy
$$
\int(\tphi_t-\phi_{\re,t})\MA(\phi_{\re,t})=0
$$
for all $t$. 

In what follows it will be more convenient to deal with potentials. For $t\ne 0$ we introduce the rescaled K\"ahler form and potentials
$$
\om_t:=\e_t\ddc\phi_{\re,t},\quad\f_t:=\e_t(\phi_t-\phi_{\re,t}),\quad\tf_t:=\e_t(\tphi_t-\phi_{\re,t}). 
$$ 
Thus $\f_t,\tf_t$ both lie in the space of normalized $\om_t$-psh functions 
$$
\CPSH_0(\om_t):=\{\f\in\CPSH(X_t,\om_t)\mid\int_{X_t}\f\,\om_t^n=0\}, 
$$
with
\begin{equation}\label{equ:MAft}
\MA(\f_t)=\e_t^{-n}V^{-1}(\om_t+\ddc\f_t)^n=\mu_t.
\end{equation}
Conjecture~\ref{conj:hybstab} now amounts to
\begin{equation}\label{equ:unifcv}
\sup_{X_t}|\f_t-\tf_t|\to 0. 
\end{equation}
For $t=0$ we similarly introduce the potential 
$$
\f_0:=\phi_0-\phi_{\re,0}\in\CPSH(\cL), 
$$
which satisfies $\MA(\f_0)=\mu_0$. Recall that $\f_0$ is convex on each face of $\D_\cX$ (see Proposition~\ref{prop:pshconvex}). 

As a first key step towards~\eqref{equ:unifcv}, Yang Li established in~\cite{LiSko} the following general result. 

\begin{thm} The exists $\a,C>0$ such that the measures $\mu_t$, $t\ne 0$, satisfy a uniform Skoda estimate
$$
\int_{X_t} e^{-\a u}\,d\mu_t\le C
$$
for all $u\in\CPSH(\om_t)$ normalized by $\sup_{X_t} u=0$.  
\end{thm}
By Lemma~\ref{lem:cap2}, we thus have 
\begin{equation}\label{equ:mucap}
\mu_t\le A\Capa_{\om_t}^2
\end{equation}
for a uniform constant $A>0$, and Theorem~\ref{thm:Kolo1} yields:

\begin{cor}\label{cor:bd} The potentials $\f_t\in\CPSH_0(\om_t)$ are uniformly bounded. 
\end{cor}

%
\subsection{The maximally degenerate case}\label{sec:maxdegcase}
Fron now on, we assume that the family of volume forms $(\sigma_t)$ is \emph{maximally degenerate}, \ie its essential complex $$
\D_\cK^\ess\subset\D_\cX\hto X_0\subset X^\hyb
$$
has maximal dimension
$$
\dim\D_\cK^\ess=n.
$$
As usual, we denote by $\cX_0=\sum_{i\in I} b_i E_i$ the central fiber of $\cX$. Each $n$-dimensional face $\D_J$ of $\D_\cK^\ess$, $J\subset I$, corresponds to an essential $0$-dimensional stratum $x_J\in\cX_0$, and we have the weak convergence of measures
\begin{equation}\label{equ:mulim}
\mu_t\to\sum_J m_J\d_{x_J}\text{  in  }\cX,\quad\mu_t\to\mu_0:=\sum_J\mu_J\text{  in  }X^\hyb, 
\end{equation}
where both sums run over the $n$-dimensional faces $\D_J$ of $\D_\cK^\ess$ and $\mu_J$ denotes the Lebesgue measure of $\D_J$ normalized to mass $m_J>0$ (see Example~\ref{exam:maxdeg}). 

We also pick disjoint adapted charts $\cU_J\simeq\DD^J$ at all $x_J$'s. Each associated log map 
$$ 
\Log_t\colon\cU_{J,t}\to\rsig_J
$$
is a principal bundle with respect to the compact Lie group 
$$
G_J=\{\theta\in(\R/\Z)^J\mid\sum_{i\in J} b_i \theta_i=0\}.
$$
By~\eqref{equ:maxdeg}, we have 
\begin{equation}\label{equ:mulog}
\Log_t^\star\mu_J=(1+f_J)\mu_t
\end{equation} 
on $\cU_{J,t}$, where $f_J\in C^\infty(\cU_J)$ vanishes at $x_J$, 

\medskip

As in~\cite{LiSYZ}, we make the following regularity assumption on the solution $\f_0\in\CPSH(\cL)$ to the non-Archimedean Monge--Amp\`ere equation $\MA(\f_0)=\mu_0$. 

\begin{ass}[Partial retraction invariance]\label{ass:inv} For each $n$-dimensional face $\D_J$ of $\D_\cK^\ess$, we have $\f_0=\f_0\circ p_\cX$ on $p_\cX^{-1}(\rsig_J)$.
\end{ass}

\begin{exam} In~\cite{HJMM} this assumption is established for the Fermat family 
$$
X_t:=\{z_0\dots z_{n+1}+t(z_0^{n+2}+\dots+z_{n+1}^{n+2})\}\subset\P^{n+1}
$$
by constructing the solution $\f_0$ in terms of the solution to a real Monge--Amp\`ere equation on the essential skeleton, realized as the boundary of the unit simplex in $\R^{n+1}$, see Example~\ref{exam:Fermat}.
\end{exam}

\begin{rmk} It is in general not true that a function $\f\in\CPSH(L_0)$ such that $\MA(\f)$ is supported in a dual complex is \emph{globally} retraction invariant, \ie satisfies $\f=\f\circ p_\cX$ on $X_0$ for some snc model $\cX$, see~\cite{NAGreen}.
\end{rmk}
Recall that $\f_0$ is convex on $\D_J$ (see Proposition~\ref{prop:pshconvex}). Thus $\Log_t^\star\f_0$ is psh on $\cU_{J,t}$ and satisfies 
\begin{equation}\label{equ:MAlog2}
\left(\ddc\Log_t^\star\f_0\right)^n=\e_t^n\Log_t^\star\MA_\R(\f_0), 
\end{equation}
see~\eqref{equ:MAlog}. As a first consequence of Assumption~\ref{ass:inv}, the comparison potential $\tf_t$ is $C^0$-close to $\Log_t^\star\f_0$ in $\cU_{J,t}$, \ie 
\begin{equation}\label{equ:complog}
\sup_{\cU_{J,t}}|\tf_t-\Log_t^\star\f_0|\to 0,
\end{equation}
see Lemma~\ref{lem:retrlim}. Thanks to Theorem~\ref{thm:Vil}, the assumption also crucially guarantees that the convex function $\f_0|_{\rsig_J}$ solves the real Monge--Amp\`ere equation 
\begin{equation}\label{equ:RMA}
\MA_\R(\f_0)=V\mu_J
\end{equation}
on $\rsig_J$. By~\eqref{equ:mulog} and~\eqref{equ:MAlog2}, this yields 
\begin{equation}\label{equ:MAf0}
\e_t^{-n}V^{-1}\left(\ddc\Log_t^\star\f_0\right)^n=(1+f_J)\mu_t
\end{equation}
on $\cU_{J,t}$. By the regularity theory for the real Monge--Amp\`ere equation~\cite{Caf,Moo}, \eqref{equ:RMA} further implies that $\f_0$ is smooth and strictly convex on an open subset 
$$
\D_J^{\reg}\subset\rsig_J
$$
of full Lebesgue measure. Basically following~\cite{LiSYZ}, we are then going to establish the following weaker form of~\eqref{equ:unifcv}. 

\begin{thm}\label{thm:d1cv} Under Assumption~\ref{ass:inv}, we have
$$
\dd_1(\f_t,\tf_t)\to 0,\quad\limsup_{t\to 0}\sup_{X_t}(\tf_t-\f_t)\le 0.
$$ 
\end{thm}
Here $\dd_1$ denotes the Darvas metric on $\om_t$-psh functions, the first point being equivalent to
$$
\int_{X_t}|\f_t-\tf_t|\,d\mu_t\to 0,\quad\int_{X_t}|\f_t-\tf_t|\,\MA(\tf_t)\to 0,
$$
see~\eqref{equ:d1}. Combining this with~\eqref{equ:complog} we get
\begin{equation}\label{equ:flog1}
\int_{\cU_{J,t}}|\f_t-\Log_t^\star\f_0|\,d\mu_t\to 0,\quad\limsup_t\sup_{\cU_{J,t}}(\Log_t^\star\f_0-\f_t)\le 0. 
\end{equation}
 In the `generic region' $\Log_t^{-1}(\D_J^{\reg})$, this will then be upgraded to smooth convergence:

\begin{cor}\label{cor:smgen} For each compact subset $\Sigma\subset\D_J^{\reg}$, $\f_t-\Log_t^\star\f_0$ tends to $0$ in $C^\infty$-topology on $\Log_t^{-1}(\Sigma)$ as $t\to 0$.
\end{cor}

%
%
%
\subsection{Proof of Theorem~\ref{thm:d1cv}}
The plan is to make a $C^0$-small modification of the comparison potential $\tf_t$ in the generic region in order to apply Yang Li's stability result (Theorem~\ref{thm:stab}). This is based on the following slightly simplified version of~\cite[Lemma~4.2]{LiSYZ}. 

\begin{lem}\label{lem:modcomp} For each $J$ pick $V_J\Subset\D_J^{\reg}$ open. For any $0<c\ll 1$, we can then find $\tp_t\in\CPSH_0(\om_t)$ for all $t$ small enough such that 
\begin{itemize}
\item[(i)] $|\tp_t-\tf_t|\le c$ on $X_t$;
\item[(ii)] $\tp_t-\Log_t^\star\f_0$ is constant on $\bigcup_J\Log_t^{-1}(V_J)$. 
\end{itemize} 
\end{lem}
\begin{proof} For each $J$ pick a cut-off function $\chi_J\in C^\infty_c(\D_J^{\reg})$ such that $0\le\chi_J\le 1$ and $\chi_J\equiv 1$ on $V_J$. Since $\f_0$ is convex on $\D_J$, and smooth and strictly convex on $\D_J^{\reg}\supset\supp\chi_J$, $\f_0+c\chi_J$ is convex on $\D_J$ for $0<c\ll 1$, and $\Log_t^\star(\f_0+c\chi_J)$ is thus psh on $\cU_{J,t}$. By~\eqref{equ:complog}, for $t$ small enough we have $|\tf_t-\Log_t^\star\f_0|<c/2$ on $\cU_{J,t}$. Since $\om_t\ge 0$, the function
$$
\tp_t:=\max\{\tf_t,\Log_t^\star(\f_0+c\chi_J)-c/2\}
$$
is $\om_t$-psh on $\cU_{J,t}$. It satisfies $|\tp_t-\tf_t|<c$, coincides with $\Log_t^\star\f_0+c/2$ on $\Log_t^{-1}(V_J)$, and with $\tf_t$ outside $\Log_t^{-1}(\supp\chi_J)$. It can thus be extended to an $\om_t$-psh on $X_t$ by setting $\tp_t:=\tf_t$ outside $\bigcup_J\cU_{J,t}$. Since $|\tp_t-\tf_t|<c$ and $\int\tf_t\,\om_t^n=0$, we can then modify $\tp_t$ by a small additive constant to further ensure $\int\tp_t\,\om_t^n=0$, \ie $\tp_t\in\CPSH_0(\om_t)$. 
\end{proof}

Now pick $\e,\d>0$. By~\eqref{equ:mulim} and~\eqref{equ:mulog} we can choose $V_J\subset\D_J^{\reg}$ such that 
$$
\mu_t\big(X_t\setminus\bigcup_J\Log_t^{-1}(V_J)\big)<\e
$$
for $t$ small enough. Pick $\tp_t$ as in Lemma~\ref{lem:modcomp} with $0<c<\d$, so that $|\tp_t-\tf_t|<\d$. For $t$ small enough, the normalized Monge--Amp\`ere measure 
$$
\MA(\tp_t)=\e_t^{-n} V^{-1}(\om_t+\ddc\tp_t)^n
$$
satisfies 
\begin{equation}\label{equ:MApt}
\MA(\tp_t)\ge\e_t^{-n} V^{-1}(\ddc\Log_t^\star\f_0)^n\ge (1-\e)\mu_t
\end{equation}
on $\bigcup_J\Log_t^{-1}(V_J)$, by Lemma~\ref{lem:modcomp}~(ii) and~\eqref{equ:MAf0}.  Since $\MA(\tp_t)$ and $\mu_t=\MA(\f_t)$ both have mass $1$, we infer
$$
\int_{X_t}\left|\MA(\tp_t)-\MA(\f_t)\right|\le 4\e
$$
for all $t$ small enough. Thanks to the uniform estimates
$$
\mu_t\le A\Capa_{\om_t}^2,\quad |\tp_t|\le |\tf_t|+1\le M,
$$
see~\eqref{equ:mucap} and Corollary~\ref{cor:bd}, Theorem~\ref{thm:stab} now shows that for $\e\le\e_0(\d,n,A,M)$ we have 
$$
\dd_1(\tp_t,\f_t)\le\d,\quad\tp_t\le\f_t+\d
$$
for all $t$ small enough, and hence 
$$
\dd_1(\tf_t,\f_t)\le 2\d,\quad\tf_t\le\f_t+2\d\text{ on }X_t, 
$$
since  $\sup_{X_t}|\tp_t-\tf_t|\le\d$. This concludes the proof of Theorem~\ref{thm:d1cv}.

\subsection{Smooth convergence in the generic region}
For each $J$ let $\rho_J\in C^\infty(\cU_J)$ be a local weight of the reference metric $\Phi_\re$ on $\cL$ that defines the rescaled K\"ahler form $\om_t=\e_t\ddc\Phi_{\re,t}$, so that $\om_t=\e_t\ddc\rho_J$ on $\cU_{J,t}$.

In the generic region $\Log_t^{-1}(\D_J^{\reg})$, both functions $\Log_t^\star\f_0$ and $\f_{J,t}:=\e_t\rho_J+\f_t$ are smooth and strictly psh, and they satisfy
$$
\left(\ddc\Log_t^\star\f_0\right)^n=(1+f_J)\left(\ddc\f_{J,t}\right)^n
$$
with $f_J\in C^\infty(\cU_J)$ vanishing at $x_J$, see~\eqref{equ:MAf0}. By general elliptic regularity theory~\cite{Sav}, the $C^\infty$-estimates of Corollary~\ref{cor:smgen} are thus reduced to the following $C^0$-estimate: 

\begin{lem} For any compact $\Sigma\subset\D_J^{\reg}$ we have $\sup_{\Log_t^{-1}(\Sigma)}|\Log_t^\star\f_0-\f_{J,t}|\to 0$. 
\end{lem}
\begin{proof} Pick $\d>0$. For $t$ small enough, \eqref{equ:flog1} implies
$$
\Log_t^\star\f_0\le\f_{J,t}+\d,\quad\int_{\cU_{J,t}}|\f_{J,t}-\Log_t^\star\f_0|\,d\mu_t\le\d
$$
Since $|\f_{J,t}|,|\f_0|\le M$ and $\Log_t^\star\mu_J=(1+f_J)\mu_t$, this yields
$$
\int_{\cU_{J,t}}\f_{J,t}\Log_t^\star\mu_J\le\int_{\cU_{J,t}}\Log_t^\star\f_0\Log_t^\star\mu_J+2\d=\int_{\rsig_J}\f_0\,\mu_J+2\d
$$
for $t$ small enough. Since $\f_{J,t}$ is psh, the function 
$$
f_{J,t}:=\int_{\cU_{J,t}/\rsig_J}\f_{J,t}
$$ 
obtained by integration along the fibers of the principal $G_J$-bundle $\Log_t\colon\cU_{J,t}\to\rsig_J$ is convex and satisfies $\f_{J,t}\le\Log_t^\star f_{J,t}$, by the mean value inequality (see~\S\ref{sec:toric}). This implies
$\Log_t^\star\f_0\le\f_{J,t}+\d\le\Log_t^\star f_{J,t}+\d$, and hence $\f_0\le f_{J,t}+\d$. Further, 
$$
\int_{\rsig_J} f_{J,t}\mu_J=\int_{\cU_{J,t}}\f_{J,t}\Log_t^\star\mu_J\le\int_{\rsig_J}\f_0\,\mu_J+2\d. 
$$
Since both functions $f_{J,t}$ and $\f_0$ are convex, this implies $\sup_\Sigma|f_{J,t}-\f_0|\le C\d$ for a constant $C=C(\Sigma)>0$. Thus
$$
\f_{J,t}\le\Log_t^\star f_{J,t}\le\Log_t^\star\f_0+C\d
$$
on $\Log_t^{-1}(\Sigma)$, and the result follows. 
\end{proof}

Corollary~\ref{cor:smgen} guarantees that the rescaled K\"ahler metric 
$$
\e_t\ddc\phi_t=\om_t+\ddc\f_t\in\e_t c_1(L_t)
$$
is $C^\infty$-close in the generic region to the \emph{semiflat metric} $\ddc\Log_t^\star\f_0$ on the principal bundle 
$$
\Log_t\colon\Log_t^{-1}(\D_J^{\reg})\to\D_J^{\reg}.
$$
In the Calabi--Yau case, \ie when $\sigma_t=|\Om_t|^2$ for a holomorphic volume form $\Om_t$ on $X_t$, $\Log_t$ is a special Lagrangian fibration for the semiflat metric, and a perturbation argument based on~\cite{Zha} then yields a special Lagrangian fibration for the Calabi--Yau metric $\om_{\mathrm{CY},t}=\ddc\phi_t$ in the generic region~\cite[Theorem~4.9]{LiSYZ}.

\appendix
%
\section{Stability theorems for complex Monge--Amp\`ere equations}\label{sec:stab}
In this appendix we work on a compact K\"ahler manifold $(X,\om)$, of dimension $n:=\dim X$ and volume $V=\int_X\om^n$. 
%
\subsection{Two metrics}
We denote by $\CPSH(\om)$ the set of continuous $\om$-psh functions $\f\in\Cz(X)$, and by 
$$
\MA(\f)=\MA_\om(\f):=V^{-1}(\om+\ddc\f)^n
$$
the (normalized) Monge--Amp\`ere operator. The space $\CPSH(\om)$ is complete with respect to the supnorm metric 
$$
\dd_\infty(\f,\p):=\sup_X|\f-\p|. 
$$
We shall also consider the \emph{Darvas metric} $\dd_1$, which satisfies 
\begin{equation}\label{equ:d1}
C_n^{-1}\dd_1(\f,\p)\le\int|\f-\p|\MA(\f)+\int|\f-\p|\MA(\p)\le C_n\dd_1(\f,\p)
\end{equation}
for a constant $C_n>0$ only depending on $n$, see~\cite[Theorem~3.32]{Dar}. The Monge--Amp\`ere operator is continuous in the $\dd_1$-distance --- for instance by~\eqref{equ:Hodge} combined with $\ii(\f,\p)\le C_n\dd_1(\f,\p)$. The completion of $\CPSH(\om)$ with respect to $\dd_1$ can be identified with the space $\cE^1(\om)$ of potentials of finite energy.

%
\subsection{The $\dd_\infty$-stability theorem}
The \emph{Bedford--Taylor capacity} is defined by
$$
\Capa(B):=\sup\left\{\int_B\MA(\f)\mid\f\in\CPSH(\om),\,0\le\f\le 1\right\}.
$$
for each Borel set $B\subset X$. Following Ko\l{}odziej's fundamental work~\cite{Kolo1}, we shall consider probability measures $\mu$ on $X$ that are well-dominated by the capacity, in the sense that $\mu\le A\Capa^2$ for a constant $A>0$. A broad class of such measures is provided by the following result, a well-known consequence of the Alexander--Taylor comparison theorem (see for instance~\cite[Lemma~2.9]{Li22}). 

\begin{lem}\label{lem:cap2} Assume $\mu$ satisfies a \emph{Skoda estimate} 
\begin{equation}\label{equ:Skoda}
\int_X e^{-\a\f}\,d\mu\le C
\end{equation}
for all $\f\in\CPSH(\om)$ normalized by $\sup_X\f=0$, where $\a,C>0$. Then $\mu\le A\Capa^2$ for a constant $A=A(n,\a,C)>0$. 
\end{lem} 

Measures satisfying ~\eqref{equ:Skoda} are fairly common: 

\begin{exam} By the uniform version of Skoda's theorem~\cite{Zer}, any smooth positive volume form $\mu$ on $X$ satisfies a Skoda estimate. 
\end{exam}

\begin{exam} If $\mu$ satisfies a Skoda estimate, H\"older's inequality shows that any measure with $L^{1+\e}$ density with respect to $\mu$ satisfies a Skoda estimate as well. In particular, any measure with $L^{1+\e}$-density with respect to Lebesgue measure satisfies a Skoda estimate. 
\end{exam}

\begin{exam} If $\f\in\CPSH(\om)$ is H\"older continuous, then $\MA(\f)$ satisfies a Skoda estimate, see~\cite[Corollary~1.2]{DNS}. 
\end{exam} 

In what follows we set
$$
\CPSH_0(\om):=\{\f\in\CPSH(\om)\mid\int\f\,\om^n=0\}.
$$
Ko\l{}odziej's results yield the following existence and stability theorems~\cite{Kolo1,Kolo2}  

\begin{thm}\label{thm:Kolo1} For any probability measure $\mu$ that satisfies $\mu\le A\Capa^2$ with $A>0$, there exists a unique $\f\in\CPSH_0(\om)$ such that $\MA(\f)=\mu$, which further satisfies $\sup_X|\f|\le M$ for a constant $M=M(n,A)$. 
\end{thm}

\begin{rmk} The result is more commonly stated with the normalization $\sup_X\f=0$, but the two versions are equivalent since $|\sup_X\f|$ and $V^{-1}\left|\int_X\f\,\om^n\right|$ are both bounded by $\sup_X|\f|$, which is under control. 
\end{rmk}

\begin{thm}\label{thm:Kolo2} Pick $\f,\p\in\CPSH_0(\om)$, and assume that 
$$
\MA(\f),\MA(\p)\le A\Capa^2
$$
with $A>0$. For any $\d>0$, we can then find $\e=\e(\d,n,A)$ such that 
$$
\int_X\left|\MA(\f)-\MA(\p)\right|\le\e\Longrightarrow\dd_\infty(\f,\p)\le\d. 
$$
\end{thm}
Here the left-hand side denotes the total variation of the signed measure $\MA(\phi)-\MA(\p)$, which coincides with its operator norm as an element of $\Cz(X)^\vee$. 

\begin{rmk} As in the above remark, a different normalization is used in~\cite[Theorem~4.1]{Kolo2}, but any choice of normalization leads to the same conclusion that $\f-\p$ is $C^0$-close to a constant.
\end{rmk}


%
\subsection{The $\dd_1$-stability theorem}
In~\cite[Theorem~2.5]{LiSYZ}, Yang Li adapts Ko\l{}odziej's proof of Theorem~\ref{thm:Kolo2} to establish an \emph{asymmetric} stability theorem, where the well-domination assumption only bears on the Monge--Amp\`ere measure of one of the two functions. We formulate it here in the following slightly improved version. 

\begin{thm}\label{thm:stab} Pick $\f,\p\in\CPSH_0(\om)$, and assume given $A,M>0$ such that
\begin{itemize}
\item[(i)] $\MA(\f)\le A\Capa^2$;
\item[(ii)] $\sup_X|\p|\le M$. 
\end{itemize}
For any $\d>0$, we can then find $\e=\e(\d,n,A,M)>0$ such that $\int_X\left|\MA(\f)-\MA(\p)\right|\le\e$ implies 
\begin{itemize}
\item[(a)] $\dd_1(\f,\p)\le\d$; 
\item[(b)] $\p\le\f+\d$ on $X$.
\end{itemize}
\end{thm} 

The first ingredient in the proof of Theorem~\ref{thm:stab} is the following `approximate domination principle', a direct application of Ko\l{}odziej's capacity estimates. 

\begin{lem}\label{lem:appdom}\cite[Theorem~2.7]{Li22}  Assume $\f,\p\in\CPSH(\om)$ satisfy (i), (ii) of Theorem~\ref{thm:stab}, and set $\mu:=\MA(\f)$. Then 
\begin{equation}\label{equ:appdom}
\p\le\f+C\mu\big(\f<\p\big)^{1/2n}
\end{equation}
for a constant $C=C(n,A,M)>0$. 
\end{lem}

\begin{proof} In what follows $C=C(n,A,M)>0$ is allowed to vary from line to line. Introduce the capacity relative to $\p$
$$
\Capa_\p(B):=\sup\left\{\int_B\MA(\f)\mid\f\in\CPSH(\om),\,0\le\f-\p\le 1\right\}.
$$
Since $\sup|\p|\le M$, it is easy to see that $C^{-1}\Capa\le\Capa_\p\le C\Capa$, and hence $\mu\le C\Capa_\p^2$. Consider the function $f\colon [0,\infty)\to[0,\infty)$ defined by 
$$
f(t):=\mu\big(\f<\p-t\big)^{1/2n}, 
$$
It is nonincreasing, right-continuous, tends to $0$ at infinity, and we have $f(t)=0$ iff $\p\le\f+t$, by the domination principle (see Corollary~\ref{cor:dom}). The key ingredient is then the inequality
$$
s^n\Capa_\p\big(\f<\p-s-t\big)\le\mu\big(\f<\p-t\big)
$$
for all $t\in\R$, $s\in [0,1]$, a general consequence of the comparison principle (see for instance~\cite[Lemma~2.3]{EGZ}), which combines with $\mu\le C\Capa_\p^2$ to yield the decay estimate 
$$
s f(t+s)\le C f(t)^2
$$
for $t\ge 0$ and $s\in [0,1]$. Thanks to an elementary lemma~\cite[Lemma~2.4]{EGZ}, this guarantees that $f(t)=0$ for $t\ge 4C f(0)$ as soon as $f(0)=\mu\big(\f<\p\big)^{1/2n}$ is smaller than $1/2C$. The result follows.
\end{proof}

The heart of the proof of Theorem~\ref{thm:stab} is the following statement. 

\begin{lem}\label{lem:stab}\cite[Theorem~2.7]{Li22} Pick $\f,\p\in\CPSH(\om)$, and assume that
\begin{itemize}
\item[(i)] $\mu:=\MA(\f)$ satisfies $\mu\le A\Capa^2$;
\item[(ii)] $\sup_X|\f|,\sup_X|\p|\le M$; 
\item[(iii)] $\mu\big(\p\le\f\big)\ge\d>0$;
\item[(iv)] $\int_X\left|\MA(\f)-\MA(\p)\right|\le\e$; 
\end{itemize}
If $\e\le\e_0(\d,n,A,M)$ we then have $\p\le\f+C\e^\a$ with $C=C(\d,nA,M)>0$ and $\a:=1/(2n+3)$. 
\end{lem} 
We will not reproduce the proof here, and only mention that it combines Lemma~\ref{lem:appdom}, the log concavity of the Monge--Amp\`ere operator, and the introduction of the solution to an auxiliary Monge--Amp\`ere equation to which the comparison principle gets applied, along the same lines as~\cite[Theorem~4.1]{Kolo2}. Note that Li assumes a Skoda estimate in place of (i), which is what is used in the proof, and that Li assumes $\f,\p$ to be smooth, which only gets used in the log concavity of the Monge--Amp\`ere operator and the Chebyshev-type inequality
$$
\mu\big(f\ge 1-s\big)\le s^{-1}\int|\mu-\nu|
$$ 
with $\nu:=\MA(\p)$ and $f=d\mu/d\nu$, both of which remain true under our slightly more general assumptions (see~\cite[Proposition~1.11]{BEGZ} for the former).

\begin{proof}[Proof of Theorem~\ref{thm:stab}] Set $f:=\p-\f$, and note that $\sup_X|f|\le C$ by Theorem~\ref{thm:Kolo1}, where $C>0$ denotes a constant only depending on $n,A,M$, that is allowed to vary from line to line. It will be enough to show $\mu\big(|f|\ge C\d\big)\le\d$ for $\e\le\e_0(\d,n,A,M)$. Indeed, this will yield (b), by Lemma~\ref{lem:appdom}. Since $|f|\le 2M$, it will also imply
$$
\int|f|\MA(\f)=\int|f|\,d\mu\le 2M\mu\big(|f|\ge\d\big)+\d\le C\d, 
$$
thus 
$$
\int|f|\MA(\p)\le\int|f|\,\MA(\f)+2M\int\left|\MA(\f)-\MA(\p)\right|\le C\d,
$$
and hence $\dd_1(\f,\p)\le C\d$, by~\eqref{equ:d1}. 

Setting $m:=\sup_X f$, we first claim that $\mu\big(f\le m-\d\big)\le\d$ if $\e\le\e_0(\d,n,A,M)$. Assuming by contradiction $\mu\big(f\le m-\d\big)>\d$, Lemma~\ref{lem:stab} yields
$$
f\le m-\d+C\e^{\a}\le m-\d/2
$$
whenever $\e\le\e_0(\d,n,A,M)$, and taking the sup over $X$ yields the desired contradiction. We next observe that 
$$
\ii(\f,\p)=\int_X f\,(\MA(\p)-\MA(\f))\le(\sup_X|f|)\int_X\left|\MA(\f)-\MA(\p)\right|\le C\e,
$$
and similarly $\ii(\f,0),\ii(\p,0)\le C$. Since $\int f\om^n=\int(\f-\p)\om^n=0$ by the normalization of $\f,\p$, \eqref{equ:Hodge} yields 
$$
\left|\int f d\mu\right|\le C\e^{\a_n}\le\d
$$
for $\e\le\e_0(\d,n,A,M)$. Since $\mu\big(|f-m|\ge\d\big)\le\d$ (by the above claim) and $\sup_X|f-m|\le\sup_X|f|+|m|\le C$, we infer 
\begin{align*}
|m| & \le \left|\int_{\{|f-m|\ge\d\}}(f-m)d\mu\right|+ \left|\int_{\{|f-m|<\d\}}(f-m)d\mu\right|+\int_X|f|\mu\\
& \le C\mu\big(|f-m|\ge\d\big)+2\d\le C\d, 
\end{align*}
thus $\{(|f|\ge C\d\}\subset\{|f-m|\ge\d\}$, and hence $\mu\left(|f|\ge C\d\right)\le\d$. The result follows. 
\end{proof}

%
%
%
%

%
%
%
%
%
%
\end{document}